\documentclass[a4paper]{article}
\pdfoutput=1
\usepackage{hyperref}
\hypersetup{hypertexnames = false, bookmarksdepth = 2, bookmarksopen = true, colorlinks, linkcolor = black, citecolor = black, urlcolor = black, pdfstartview={XYZ null null 1}}

\usepackage{amsfonts}
\usepackage[fleqn, leqno]{amsmath}
\usepackage{amsthm}
\usepackage{amssymb}

\usepackage[capitalise]{cleveref}

\usepackage[utf8]{inputenc}
\usepackage[backend=biber, maxbibnames=99, datamodel=mrnumber, sortcites]{biblatex}

\usepackage{mathabx}

\usepackage{booktabs}
\usepackage{colonequals}
\usepackage{diagbox}
\usepackage{enumitem}
\usepackage{fullpage}
\usepackage{mathtools}
\usepackage{parskip}
\usepackage{thmtools}
\usepackage{tikz-cd}
\usepackage[colorinlistoftodos, textsize = footnotesize]{todonotes}
\usepackage{xparse}
\usepackage{xspace}

\usepackage[T1]{fontenc}
\usepackage{libertine}

\usepackage[libertine]{newtxmath}
\usepackage[scaled=0.83]{beramono}
\usepackage{eucal}
\usepackage{microtype}
\usepackage{gitinfo2}
\usepackage[frozencache]{minted}

\DeclareFieldFormat{mrnumber}{%
  MR\addcolon\space
  \ifhyperref
    {\href{http://www.ams.org/mathscinet-getitem?mr=#1}{\nolinkurl{#1}}}
    {\nolinkurl{#1}}}

\renewbibmacro*{doi+eprint+url}{%
  \iftoggle{bbx:doi}
    {\printfield{doi}}
    {}%
  \newunit\newblock
  \printfield{mrnumber}%
  \newunit\newblock
  \iftoggle{bbx:eprint}
    {\usebibmacro{eprint}}
    {}%
  \newunit\newblock
  \iftoggle{bbx:url}
    {\usebibmacro{url+urldate}}
    {}}

\newcounter{todocounter}
\DeclareDocumentCommand\addreference{g}{\stepcounter{todocounter}\todo[color = blue!30]{\thetodocounter. Add reference\IfNoValueF{#1}{: #1}}\xspace}
\DeclareDocumentCommand\checkthis{g}{\stepcounter{todocounter}\todo[color = red!50]{\thetodocounter. Check this\IfNoValueF{#1}{: #1}}\xspace}
\DeclareDocumentCommand\fixthis{g}{\stepcounter{todocounter}\todo[color = orange!50]{\thetodocounter. Fix this\IfNoValueF{#1}{: #1}}\xspace}
\DeclareDocumentCommand\expand{g}{\stepcounter{todocounter}\todo[color = green!50]{\thetodocounter. Expand\IfNoValueF{#1}{: #1}}\xspace}

\declaretheoremstyle[
  spaceabove = 3pt,
  spacebelow = 3pt,
  bodyfont = \itshape,
]{first}
\declaretheoremstyle[
  spaceabove = 3pt,
  spacebelow = 3pt,
]{second}
\declaretheorem[numberwithin=section, style=first]{theorem}
\declaretheorem[sibling=theorem, style=first]{conjecture}
\declaretheorem[sibling=theorem, style=first]{corollary}
\declaretheorem[sibling=theorem, style=first]{lemma}
\declaretheorem[sibling=theorem, style=first]{proposition}

\declaretheorem[sibling=theorem, style=second]{example}
\declaretheorem[sibling=theorem, style=second]{remark}
\declaretheorem[sibling=theorem, style=second]{definition}
\declaretheorem[sibling=theorem, style=second]{notation}
\crefname{notation}{Notation}{Notations}

\declaretheorem[numberwithin=section, style=first, title=Theorem]{alphatheorem}
\declaretheorem[sibling=alphatheorem, style=first, title=Conjecture]{alphaconjecture}
\declaretheorem[sibling=alphatheorem, style=first, title=Corollary]{alphacorollary}

\crefname{alphatheorem}{Theorem}{Theorems}
\crefname{alphaconjecture}{Conjecture}{Conjectures}
\crefname{alphacorollary}{Corollary}{Corollaries}
\crefname{alphaproposition}{Proposition}{Propositions}

\Crefname{conjecture}{Conjecture}{Conjectures}

\makeatletter
\def\gitfootnote{\gdef\@thefnmark{}\@footnotetext}
\makeatother

\makeatletter
\newcommand\code{}
\renewcommand\code{\inputminted{macaulay2}{code/chi-tritangent.m2}}
\makeatother

\mathchardef\mhyphen="2D
\newcommand\dash{\nobreakdash-\hspace{0pt}}

\newcommand\caldararu{\texorpdfstring{C\u{a}ld\u{a}raru}{Caldararu}\xspace}

\let\oldbigwedge\bigwedge
\renewcommand\bigwedge{\oldbigwedge\nolimits}

\newcommand\hilbn[2]{\ensuremath{\operatorname{Hilb}^#1#2}}

\DeclareMathOperator\hochserre{HS}
\DeclareMathOperator\hochschild{HH}
\DeclareMathOperator\hochschilddim{hh}

\DeclareMathOperator\Ad{Ad}

\DeclareMathOperator\age{age}
\DeclareMathOperator\alb{alb}
\DeclareMathOperator\Aut{Aut}
\DeclareMathOperator\derived{\mathbf{D}}
\DeclareMathOperator\Def{Def}
\DeclareMathOperator\Sym{Sym}
\DeclareMathOperator\symgr{\mathfrak{S}}

\DeclareMathOperator\Ho{H} % TODO these are the same macro
\DeclareMathOperator\HH{H} % TODO these are the same macro
\DeclareMathOperator\hh{h}
\DeclareMathOperator\HHHH{HH}
\DeclareMathOperator\Ind{Ind}
\DeclareMathOperator\Map{Map}
\DeclareMathOperator\Ext{Ext}
\DeclareMathOperator\orb{orb}
\DeclareMathOperator\ord{ord}
\DeclareMathOperator\Pic{Pic}
\DeclareMathOperator\pr{pr}

\DeclareMathOperator\Hom{Hom}
\DeclareMathOperator\RHom{\mathbf{R}Hom}
\DeclareMathOperator\RGamma{\mathbf{R}\Gamma}
\DeclareMathOperator\serre{S}
\DeclareMathOperator\Spec{Spec}
\DeclareMathOperator\id{id}
\DeclareMathOperator\Conj{Conj}
\DeclareMathOperator\Perf{Perf}
\DeclareMathOperator\orbits{O}
\DeclareMathOperator\centraliser{C}
\DeclareMathOperator\Proj{Proj}

\newcommand\bounded{\ensuremath{\mathrm{b}}}
\newcommand\IN{\mathbb N}
\newcommand\IZ{\mathbb Z}
\newcommand\IC{\mathbb C}
\newcommand\tangent{\mathrm{T}}
\newcommand\identity{\ensuremath{\mathrm{id}}}
\newcommand\inertia{\ensuremath{\mathrm{I}}}
\newcommand\loops{\ensuremath{\mathrm{L}}}
\newcommand\sgn{\ensuremath{\mathrm{sgn}}}
\newcommand\RRR{\ensuremath{\mathbf{R}}}
\newcommand\LLL{\ensuremath{\mathbf{L}}}
\newcommand\RR{\ensuremath{\mathrm{R}}}
\newcommand\field{\mathbf{k}}
\newcommand\env{\ensuremath{\mathrm{e}}}
\newcommand\opp{\ensuremath{\mathrm{op}}}

\newcommand{\cA}{\mathcal{A}}
\newcommand{\cB}{\mathcal{B}}
\newcommand{\cE}{\mathcal{E}}

\newcommand{\cP}{\mathcal{P}}
\newcommand{\cX}{\mathcal{X}}
\newcommand{\cY}{\mathcal{Y}}
\newcommand{\cZ}{\mathcal{Z}}

\newcommand{\sym}{\mathfrak S}

\let\autoref\undefined

\title{Hochschild cohomology of Hilbert schemes of points on surfaces}
\author{Pieter Belmans \and Lie Fu \and Andreas Krug}

\begin{document}
\maketitle

%\gitfootnote{commit: \texttt{\gitAbbrevHash}\hfil date: \texttt{\gitAuthorIsoDate}\hfil \texttt{\gitReferences}}

\begin{abstract}
  We compute the Hochschild cohomology of Hilbert schemes of points on surfaces
  and observe that it is, in general, not determined solely by the Hochschild cohomology of the surface,
  but by its ``Hochschild--Serre cohomology''---the bigraded vector space
  obtained by taking Hochschild homologies with coefficients in powers of the Serre functor.
  As applications, we obtain various consequences on the deformation theory of the Hilbert schemes;
  in particular, we recover and extend results of Fantechi, Boissi\`ere, and Hitchin.

  Our method is to compute more generally for any smooth proper algebraic variety $X$
  the Hochschild--Serre cohomology
  of the symmetric quotient stack $[X^n/\symgr_n]$,
  in terms of the Hochschild--Serre cohomology of $X$.
\end{abstract}

\tableofcontents

\section{Introduction}
% something about computing invariants of Hilbert schemes
% derived McKay: done
% mention orbifold HKR
\subsection{Hochschild cohomology of Hilbert schemes of points on surfaces}
For a smooth projective surface~$S$,
the Hilbert scheme of points~$\hilbn{n}{S}$
is again a smooth projective variety \cite{MR0237496}.
The geometry and invariants of~$\hilbn{n}{S}$ are controlled by those of~$S$.
Probably the most famous result in this direction is the identification of
the direct sum of the singular cohomologies of all the Hilbert schemes of points on~$S$
with the Fock space representation of the Heisenberg Lie algebra associated with the cohomology of the surface $S$;
see \cite{MR1219901,MR1312161, MR1032930, MR1441880}.
This means,
in particular,
that we have an isomorphism of graded vector spaces (up to some degree shift)
\begin{equation}
  \label{equation:fock-vector-space-isomorphism}
  \bigoplus_{n\geq 0}\HH^*(\hilbn{n}{S},\IC)t^n
  \cong
  \Sym^\bullet\left(\bigoplus_{i\geq 1}\HH^{*}(S,\IC)t^i\right)\,,
\end{equation}
where the symmetric power $\Sym^\bullet$ on the right-hand side
is graded for the grading of the cohomology,
but ordinary for the grading given by exponents of the formal variable $t$;
see \cref{subsection:sym-of-vector-spaces} for details on graded symmetric powers.

Due to the identification of singular cohomology and Hochschild homology
via the Hochschild--Kostant--Rosenberg isomorphism,
we get an analogous isomorphism  of graded vector spaces (without degree shift)
for the \emph{Hochschild homology} of all the Hilbert schemes taken together.
\begin{equation}
  \label{equation:hochschild-homology--Hilb-introduction}
  \bigoplus_{n\geq 0}\hochschild_{*}(\hilbn{n}{S})t^n
  \cong
  \Sym^\bullet\left( \bigoplus_{i\geq 1}\hochschild_{*}(S)t^i  \right)\,.
\end{equation}
In particular, the Hochschild homology of $\hilbn{n}{S}$ is determined by that of $S$.

\emph{Hochschild cohomology} is another invariant of varieties,
whose definition is similar to that of Hochschild homology,
but whose behavior is, in general, quite different.
For some explicit examples of calculations
and its behavior in families,
the reader is referred to \cite{1911.09414v1,MR4578397}.
Hochschild cohomology plays an important role in the deformation theory of varieties
and their (derived) categories of coherent sheaves,
as discussed in \cref{subsection:deformation-theory}.

The motivation for this paper was whether
the Hochschild cohomology of the Hilbert schemes of points $\hochschild^*(\hilbn{n}{S})$
can be expressed by a formula similar to \eqref{equation:fock-vector-space-isomorphism} and \eqref{equation:hochschild-homology--Hilb-introduction},
or at least determined by the Hochschild cohomology $\hochschild^*(S)$ of the surface.
It turns out that the information of the Hochschild cohomology of $S$,
even jointly with its Hochschild homology,
is not enough for this purpose in general
(see \cref{subsection:bielliptic} for concrete examples using bielliptic surfaces),
but one needs the full non-positive part ($k\le 0$) of the \emph{Hochschild--Serre cohomology}
\begin{equation}
  \label{equation:intro-hochschild-serre}
  \hochserre_\bullet(S)\colonequals\bigoplus_{k\in \IZ} \hochserre_k(S)
  \quad\text{where}\quad
  \hochserre_k(S)\colonequals\RHom_{S\times S}(\Delta_*\mathcal{O}_S, \Delta_*\omega_S^{\otimes k}[k \dim S]).
\end{equation}
This invariant involving all powers of the Serre functor was first defined for varieties
and shown to be a derived invariant in \cite[page~535]{MR1998775} (see also \cite[page~139]{MR2244106}).
The definition of Hochschild--Serre cohomology, as well as its derived invariance,
can be naturally extended to orbifolds (\cref{definition:hochschild-serre}, \cref{cor:DerivedInvariance}),
and even to smooth proper dg categories (\cref{appendix:hochschild-serre-dg}).
In \cite{MR1998775,MR2244106} this is referred to as the \emph{Hochschild algebra},
as it is a bigraded algebra
which contains~the Hochschild cohomology $\hochserre_0(S)=\hochschild^*(S)$ as a graded subalgebra,
and Hochschild homology $\hochserre_1(S)=\hochschild_*(S)$ as a graded submodule.
But to avoid confusion with the usual algebra structure on Hochschild cohomology,
we will refer to the entire bigraded structure in \eqref{equation:intro-hochschild-serre}
as \emph{Hochschild--Serre cohomology}.

Our first main result
is a formula expressing the Hochschild--Serre cohomology of
all the Hilbert schemes of points
using the Hochschild--Serre cohomology of the surface $S$.
We should point out that we only compute the Hochschild--Serre cohomology of
Hilbert schemes as a bigraded vector space.
So for now, we have to leave the computation of the algebra structure as an open problem.

\begin{alphatheorem}[\cref{corollary:hochschild-serre-hilbert-scheme}]
  \label{theorem:intro-main}
  For any smooth projective surface $S$ defined over a field of characteristic zero, and any integer $k$, we have
  \begin{equation}
    \label{equation:hilbert-hochschild-serre-even-all-n-intro}
    \bigoplus_{n\geq 0}\hochserre_{k}(\hilbn{n}{S} )t^n
    \cong
    \Sym^\bullet\left(\bigoplus_{i\geq 1}\hochserre_{1+(k-1)i}(S)t^i\right).
  \end{equation}
  In particular, considering $k=0$, the Hochschild cohomology of the Hilbert schemes is given by
  \begin{equation}\label{equation:hochschild-cohomology-hilbert-scheme}
    \bigoplus_{n\geq 0}\hochschild^*(\hilbn{n}{S})t^n
    \cong
    \Sym^\bullet\left(\bigoplus_{i\geq 1}\hochserre_{1-i}(S)t^i\right).
  \end{equation}
\end{alphatheorem}

Notice on the right-hand side of \eqref{equation:hochschild-cohomology-hilbert-scheme},
pieces of Hochschild--Serre cohomology of $S$ other than $\hochschild^{*}(S)$ do come into play.

We are in particular interested in~$\hochschild^1(\hilbn{n}{S})$ and~$\hochschild^2(\hilbn{n}{S})$,
given the role of these spaces in understanding symmetries and deformations of~$\derived^\bounded(\hilbn{n}{S})$,
and thus also of~$\hilbn{n}{S}$,
as recalled in \cref{subsection:deformation-theory}.
As an application of \cref{theorem:intro-main}, the following corollary bootstraps the calculation of \cref{theorem:intro-main}
to describe one of the summands in the Hochschild--Kostant--Rosenberg decomposition of~$\hochschild^2(\hilbn{n}{S})$
which describes the \emph{classical} deformation theory of~$\hilbn{n}{S}$.
It reproves (with different methods) and generalizes results of Fantechi \cite[Theorems~0.1 and~0.3]{MR1354269}
and Hitchin \cite[\S4.1]{MR3024823},
so that arbitrary surfaces can now be considered.

\begin{alphacorollary}
  \label{corollary:tangent-hilbn-S}
  Let~$S$ be a smooth projective surface defined over a field $\field$ of characteristic zero.
  For all~$n\geq 2$, there exists an isomorphism
  \begin{equation}
    \label{equation:tangent-hilbn-S}
    \HH^1(\hilbn{n}{S},\tangent_{\hilbn{n}{S}})
    \cong
    \HH^1(S,\tangent_S)\oplus \left(\HH^0(S,\tangent_S)\otimes_\field\HH^1(S,\mathcal{O}_S)\right) \oplus\HH^0(S,\omega_S^\vee).
  \end{equation}
\end{alphacorollary}

This corollary suggests that one should study the deformation theory of
Hilbert schemes of points on bielliptic surfaces more closely.
In \cref{example:bielliptic-hilbert-conjecture} we explain how they admit
an additional deformation direction.

Another applications of \cref{theorem:intro-main} is \cref{corollary:boissiere},
where we give an alternative proof of Boissi\`ere's result \cite[Corollaire~1]{MR2932167}
stating that the infinitesimal automorphisms of~$S$ and~$\hilbn{n}{S}$ agree,
by means of the equality~$\dim\Aut^0(S)=\dim\Aut^0(\hilbn{n}{S})$.

\subsection{Hochschild--Serre cohomology for symmetric quotient stacks}
Our approach to \cref{theorem:intro-main} is ``non-commutative''.
We deduce it from the following more general result involving the derived category of
the \emph{symmetric quotient stack} $[\Sym^nX]$, namely, the quotient stack~$[X^n/\symgr_n]$
where the symmetric group~$\symgr_n$ acts on the cartesian power $X^n$ by permuting the factors.
The result works for varieties of arbitrary dimension, not only for surfaces.

\begin{alphatheorem}[\cref{corollary:hochschild-serre}]
  \label{theorem:hochschild-serre}
  Let $X$ be a smooth proper algebraic variety of dimension $d_X$ over a field $\field$ of characteristic zero.
  Let $k$ be a fixed positive integer.
  \begin{enumerate}
    \item[(i)] If $(k-1)d_X$ is even, we have
    \begin{equation}
      \label{equation:hochschild-serre-even-all-n-intro}
      \bigoplus_{n\geq 0}\hochserre_{k}([\Sym^nX])t^n
      \cong
      \Sym^\bullet\left(\bigoplus_{i\geq 1}\hochserre_{1+(k-1)i}(X)t^i\right)\,.
    \end{equation}
    \item[(ii)] If $(k-1)d_X$  is odd, we have
    \begin{equation}
      \label{equation:hochschild-serre-odd-all-n-intro}
      \bigoplus_{n\geq 0}\hochserre_{k}([\Sym^nX])t^n
      \cong
      \Sym^\bullet\left(\bigoplus_{\substack{i\geq 1\\ i \text{ odd}}}\hochserre_{1+(k-1)i}(X)t^i\right)\,.
    \end{equation}
  \end{enumerate}
\end{alphatheorem}
\cref{theorem:hochschild-serre} is a special case of the even more general \cref{theorem:main-hochschild}, which deals with general Hochschild homology with coefficients, not only Hochschild--Serre cohomology.

The relation between \cref{theorem:intro-main} and \cref{theorem:hochschild-serre} stems from the equivalence of categories
\begin{equation}
  \label{equation:derived-mckay}
  \derived^\bounded(\hilbn{n}{S})\cong\derived^\bounded([\Sym^nS])
\end{equation}
obtained by combining Bridgeland--King--Reid's derived McKay correspondence \cite{MR1824990}
with Haiman's description of the isospectral Hilbert scheme \cite{MR1839919}.
Thus if one is interested in derived invariants
of Hilbert schemes of points on surfaces,
such as Hochschild--Serre cohomology,
one can use the symmetric quotient stack for computations.
In particluar, \cref{theorem:intro-main}
becomes a special case of \cref{theorem:hochschild-serre}.

%Indeed, if we replace $\hilbn{n}{S}$ with the symmetric quotient stack $[\Sym^n S]$
%we can prove formula \eqref{equation:hilbert-hochschild-serre-even-all-n-intro}
%for $X=S$ a smooth projective variety of arbitrary dimension;
%see \cref{corollary:hochschild-serre}
%(for $\dim X$ odd and $k$ even, the formula has to be slightly adapted;
%in all other cases, it is literally the same).

Taking $k=1$ in \cref{theorem:hochschild-serre}, we deduce the following higher-dimensional generalization of \eqref{equation:hochschild-homology--Hilb-introduction}:
\begin{alphacorollary}
  \label{corollary:Hochschild-homology-Fock}
  Let $X$ be a smooth proper algebraic variety of dimension $d_X$ over a field $\field$ of characteristic zero. We have
  \begin{equation}
    \label{equation:hochschild-homology-Fock}
    \bigoplus_{n\geq 0}\hochschild_{*}([\Sym^nX])t^n
    \cong
    \Sym^\bullet\left( \bigoplus_{i\geq 1}\hochschild_{*}(X)t^i  \right).
  \end{equation}
    Consequently, the vector space $\bigoplus_{n\geq 0}\hochschild_{*}([\Sym^nX])$
  can be identified with the Fock space representation of the Heisenberg algebra associated with $\hochschild_{*}(X)$.
\end{alphacorollary}

Despite the similarity between the formulas in \cref{theorem:hochschild-serre} for Hochschild--Serre cohomology and \eqref{equation:hochschild-homology-Fock} for Hochschild homology, we do not see a way to equip $\bigoplus_{n\geq 0}\hochserre_{k}([\Sym^nX])t^n$ with the structure of a Fock space for $k\neq 1$;
see \cref{subsection:fock-space} for some further discussion on this.

\begin{remark}
  We expect \cref{corollary:Hochschild-homology-Fock} to hold more generally (\cref{conjecture:HH-Fock-dg}): for any smooth proper dg category $\mathcal{T}$, we should have
  \begin{equation}
  \label{equation:hochschild-homology-Fock-dg}
\bigoplus_{n\geq 0}\hochschild_{*}(\Sym^n\mathcal{T})t^n
\cong
\Sym^\bullet\left( \bigoplus_{i\geq 1}\hochschild_{*}(\mathcal{T})t^i  \right).
  \end{equation}
  We can indeed provide ample evidence to \eqref{equation:hochschild-homology-Fock-dg}, for instance, for various Kuznetsov components of (Fano) varieties. See \cref{corollary:Kuznetsov-Component}.
\end{remark}
%For $X=S$ a surface, this simply follows from the classical result of Nakajima and Grojnowski
%for the singular cohomology of the Hilbert schemes,
%together with the Hochschild--Kostant--Rosenberg isomorphism,
%and the derived McKay correspondence.
%For $X$ of arbitrary dimension however, this seems to be a new result.

\subsection{Twisted Hodge groups: Boissi\`ere's conjecture and a revised version}
As mentioned before,
our most general statement, \cref{theorem:main-hochschild},
computes the \emph{Hochschild homology with coefficients} $\hochschild_{*}([\Sym^nX], F^{\{n\}})$ of symmetric quotient stacks $[\Sym^nX]$
with coefficients in external tensor powers of objects $F\in \derived^\bounded(X)$;
see \cref{definition:hochschild-with-coefficients} and \cref{notation:cyclic-invariants} for the definitions.
The proof uses the orbifold version of the Hochschild--Kostant--Rosenberg isomorphism due to Arinkin--\caldararu--Hablicsek \cite{MR4003476}
(actually, we need a straightforward generalisation of their result; see \cref{proposition:orbifold-hkr}).

Taking the coefficient $F$ to be some powers of the appropriately shifted canonical bundle $\omega_X[d_X]$,
we obtain Hochschild--Serre cohomology,
and hence \cref{theorem:intro-main}.
Another interesting special case of \cref{theorem:main-hochschild} occurs
when we take the coefficient $F$ to be a line bundle ${L}$,
hence ${L}^{\{n\}}$ is the naturally induced line bundle on the symmetric quotient stack
(corresponding to, in the surface case, the line bundle ${L}_n$ on the Hilbert scheme).
More precisely, we prove in \cref{corollary:HH-Hilb-LineBundle} the following more general version of \cref{theorem:intro-main}:
\begin{equation}
\label{eqn:HH-Hilb-line-bundle-intro}
  \bigoplus_{n\geq 0}\hochschild_{*}(\hilbn{n}{S} , {L}_n)t^n
  \cong\Sym^\bullet\left(\bigoplus_{i\geq 1}\hochschild_{*}(S, {L}^{\otimes i})t^i\right)\,.
\end{equation}
Via the Hochschild--Konstant--Rosenberg isomorphism,
this is closely related to a conjecture of Boissi\`ere \cite[Conjecture~1]{MR2932167}
(disproven in \cite[Appendix~B]{MR3778120})
on twisted Hodge numbers of Hilbert schemes $\HH^{p,q}(\hilbn{n}{S}, {L}_n)$.
%see \cref{subsection:boissiere} and \cref{subsection:nieper-wisskirchen} for details.

In \cref{subsection:boissiere}, an alternative counterexample to Boissi\`ere's conjecture (see \cref{example:counterexample-boissiere}) is provided,
and we speculate about the bigraded version of the decomposition \eqref{eqn:HH-Hilb-line-bundle-intro}
and propose the following revision of Boissi\`ere's conjecture.

\begin{alphaconjecture}[\cref{conjecture:corrected}]
  \label{conjecture:intro-corrected}
  Let $S$ be a smooth projective surface,
  and ${L}\in \Pic S$.
  Then
  \begin{equation}
    \label{equation:intro-corrected}
    \begin{aligned}
      \sum_{n\geq 0}\sum_{p=0}^{2n}\sum_{q=0}^{2n}\hh^{p,q}(\hilbn{n}{S},{L}_n)x^py^qt^n
      &=
      \prod_{k\ge 1}\prod_{p=0}^2\prod_{q=0}^2\left( 1-(-1)^{p+q}x^{p+k-1}y^{q+k-1}t^k\right)^{-(-1)^{p+q}\hh^{p,q}(S,{L}^{\otimes k})}.
    \end{aligned}
  \end{equation}
\end{alphaconjecture}

We verify the conjecture for $x=0$, $y=0$, $y=-1$, or $x=y^{-1}$, and finally and most notably, in \cref{subsection:nieper-wisskirchen}, we use a result of Nieper--Wi\ss kirchen \cite{MR2578804} to show that \cref{conjecture:corrected} holds for line bundles ${L}$ admitting a unitary flat connection; see \cref{theorem:TwistHodgeHilb}.

\paragraph{Examples}
We work out some examples of the formula \eqref{equation:hochschild-cohomology-hilbert-scheme}
(and its generalization for arbitrary~$X$)
to illustrate the methods.

In \cref{subsection:symmetric-square-P1} we compute~$\hochschild^*([\Sym^2\mathbb{P}^1])$
using our main result,
and verify it by calculating the Hochschild cohomology of
a derived equivalent finite-dimensional algebra.

In \cref{subsection:hilbert-square-P2} we compute~$\hochschild^*([\Sym^2\mathbb{P}^2])\cong\hochschild^*(\hilbn{2}{\mathbb{P}^2})$
using our main result,
and improve upon it by calculating the Hochschild--Kostant--Rosenberg decomposition
using an explicit geometric model for the Hilbert square of~$\mathbb{P}^2$.

In \cref{subsection:bielliptic} we show by means of an example involving bielliptic surfaces
that~$\hochschild^*(S)$ and $\hochschild_{*}(S)$ do not determine~$\hochschild^*(\hilbn{n}{S})$ in general.
In other words,
the Hochschild--Serre cohomology of a surface contains strictly more information
than just its Hochschild (co)homology.
See \cref{corollary:bielliptic-disagreement} for the precise statement.

%In fact, \cref{theorem:intro-main} can be interpreted as giving a partially corrected version of \cite[Conjecture~1]{MR2932167}
%where the twist in loc.~cit.~is given by a power of the canonical bundle,
%and turning the bigrading by~$p$ and~$q$ into a single grading by~$p+q=m$.

\paragraph{Conventions}
In this paper, $\field$ is a base field. An \textit{orbifold} means a tame separated Deligne--Mumford stack
of finite type over $\field$.
In contrast to the usual definition,
we do not impose trivial generic stabilizer.
Unless otherwise specified, all the fiber products are over $\Spec\field$.
We denote by $d_X$ the dimension of a variety (or stack) $X$.

\paragraph{Acknowledgements}
We would like to thank Samuel Boissi\`ere, Martijn Kool, and
Theo Raedschelders for interesting discussions.

The second author is supported by the University of Strasbourg Institute for Advanced Study (USIAS)
and by the Agence Nationale de la Recherche (ANR), under the project number ANR-20-CE40-0023.

\section{Hochschild (co)homology, Hochschild--Serre cohomology and their decomposition}
\label{section:hochschild}
\subsection{Hochschild (co)homology with coefficients and Hochschild--Serre cohomology}
\label{subsection:hochschild-serre}
Hochschild (co)homology of varieties, and more generally orbifolds, can be defined in various ways,
we will use the approach using Fourier--Mukai transforms as used, e.g., in \cite{MR2141853}.
In \cite[page~535]{MR1998775} (see also \cite[page~139]{MR2244106})
a generalisation is introduced,
which incorporates all powers of the Serre functor,
and which we will call \emph{Hochschild--Serre cohomology}; see \cref{definition:hochschild-serre}.
We will use the slightly more general notion of \emph{Hochschild (co)homology with coefficients},
which includes all parts of the Hochschild--Serre cohomology as special cases.

Throughout, let~$\cX$ be a smooth \emph{proper} orbifold. We denote by $\omega_{\cX}$ its canonical bundle and $d_{\cX}$ its dimension.

The Serre functor of $\derived^\bounded(\cX)$ is given by
\begin{equation}
  \serre_\cX\colonequals-\otimes \omega_\cX[d_\cX],
\end{equation}
see \cite[Section 2.2]{MR2657369}, \cite{0811.1955v2},
and also \cite{MR4514209} in the presence of a projective coarse moduli space.
If~$\cX=[M/G]$ is a global quotient stack by a finite group $G$,
then under the usual identification of coherent sheaves on the quotient stack~$[M/G]$
with $G$-equivariant coherent sheaves on $M$,
the canonical bundle $\omega_\cX$ corresponds to $\omega_M$
equipped with the linearisation given by pullback of top forms along the group action.

\begin{definition}
  \label{definition:hochschild-with-coefficients}
  For $\cE\in \derived^\bounded(\cX)$, we define the
  \emph{Hochschild homology of~$\cX$ with coefficients in $\cE$} as
   \begin{equation}
    \hochschild_*(\cX, \cE) \colonequals\RGamma(\mathcal{X}\times \mathcal{X}, \Delta_{*}\mathcal{O}_{\mathcal{X}}\otimes^{\mathbf{L}} \Delta_*\mathcal{E})
    =\RHom_{\cX\times \cX}(\mathcal{O}_{\cX\times \cX}, \Delta_{*}\mathcal{O}_{\mathcal{X}}\otimes^{\mathbf{L}}  \Delta_*\mathcal{E}),
  \end{equation}
  and the \emph{Hochschild cohomology of~$\cX$ with coefficients in $\cE$} as
   \begin{equation}
    \hochschild^*(\cX, \cE)\colonequals
    \RHom_{\cX\times \cX}(\Delta_*\mathcal{O}_{\cX}, \Delta_*\cE)\,,
  \end{equation}
  where $\Delta\colon \cX\to \cX\times \cX$ is the diagonal morphism.
  By Grothendieck duality, these two invariants are related as follows (see, for example, \cite[Lemma 2.1]{kuznetsov2009hochschild}):
  \begin{equation}
    \hochschild^*(\cX, \serre_{\cX}\cE)\cong \hochschild_*(\cX, \cE)\,.
  \end{equation}
  Concretely,
  \begin{equation}
    \hochschild_*(\cX, \cE)
   \cong
    \RHom_{\cX\times \cX}(\Delta_*\mathcal{O}_{\cX}, \Delta_*(\serre_{\cX}\cE))
    =
    \RHom_{\cX\times \cX}(\Delta_*\mathcal{O}_{\cX}, \Delta_*(\omega_{\cX}[d_{\cX}]\otimes \cE))\,,
  \end{equation}
\end{definition}
There is a slight abuse of notation happening here,
where we will not always distinguish between an object in~$\derived^\bounded(\field)$
and its cohomology,
which is a graded vector space.

If we spell out what happens when we take cohomology, then for any $j\in \mathbb{Z}$,
\begin{equation}
  \hochschild_j(\cX, \cE)
  =
  \Ext^j_{\cX\times \cX}(\Delta_*\mathcal{O}_{\cX}, \Delta_*(\omega_{\cX}[d_{\cX}]\otimes \cE)).
\end{equation}

\begin{remark}
  It is possible to define Hochschild (co)homology with values in bimodules,
  i.e., elements of~$\derived^\bounded(\cX\times\cX)$,
  but we will not consider this.
  The reason is that the Hochschild--Kostant--Rosenberg decomposition
  only works for symmetric bimodules (for the affine setting of this statement, see \cite[\S1.3]{MR1600246}),
  and thus only for objects in the essential image of~$\Delta_*$.
\end{remark}

For the special case where~$\cE$ is a tensor power of the shifted canonical bundle $\omega_{\cX}[d_{\cX}]$,
we introduce the following terminology.

\begin{definition}
  \label{definition:hochschild-serre}
  The \emph{Hochschild--Serre cohomology} of $\cX$ is
  \begin{equation}
    \hochserre_\bullet(\cX)
    \colonequals
    \bigoplus_{k\in\mathbb{Z}}\hochserre_k(\cX)
  \end{equation}
  where for any $k\in \mathbb{Z}$,
  \begin{equation}
    \hochserre_k(\cX)
    \colonequals
    \RHom_{\cX\times \cX}(\Delta_*\mathcal{O}_\cX, \Delta_*\omega_\cX^{\otimes k}[kd_\cX])
  \end{equation}
  and thus
  \begin{equation}
    \hochserre_k(\cX)
    \cong
    \hochschild^*(\cX, \omega_\cX^{\otimes k}[kd_\cX])
    \cong
    \hochschild_{*}(\cX, \omega_\cX^{\otimes k-1}[(k-1)d_\cX]).
  \end{equation}
  Taking cohomology, we get the bigraded algebra
  \begin{equation}
    \hochserre_\bullet^*(\cX)
    \colonequals
    \bigoplus_{j, k\in\mathbb{Z}}\hochserre_{k}^j(\cX),
  \end{equation}
  with
  \begin{equation}
    \hochserre_{k}^j(\cX)\colonequals \hochschild^j(\cX, \omega_\cX^{\otimes k}[kd_\cX])
    =
    \Ext_{\cX\times \cX}^{j+kd_{\cX}}(\Delta_*\mathcal{O}_{\cX},\Delta_*\omega_{\cX}^{\otimes k}),
  \end{equation}
  where the algebra structure is given by composition in~$\derived^\bounded(\cX\times\cX)$.
\end{definition}

The objective of this paper is to compute,
given a smooth and proper variety $X$,
the Hochschild--Serre cohomology of the symmetric quotient stack~$[X^n/\symgr_n]$,
in terms of the Hochschild--Serre cohomology of~$X$.

\begin{remark}
  \label{remark:hochschild-serre-generalises-classical}
  We recover some classical invariants as certain graded pieces of Hochschild--Serre cohomology:
%  at least when~$\cX$ is a variety:
  \begin{itemize}
    \item Hochschild cohomology, namely~$\hochschild^*(\cX)\cong\hochserre_{0}^*(\cX)$,
    \item Hochschild homology, namely~$\hochschild_*(\cX)\cong\hochserre_{1}^*(\cX)$,
    \item the canonical ring, namely~$\mathrm{R}(\cX)=\bigoplus_{k\geq 0}\hochserre_{k}^{-kd_\cX}(\cX)$.
  \end{itemize}
  A generalisation of Hochschild--Serre cohomology for varieties is used in \cite{MR3218801}
  to find new derived invariants, through the formalism of cohomological support loci.
\end{remark}

\begin{remark}
  In \cite[\S2.1]{MR1998775}
  and \cite[\S6.1]{MR2244106} this definition is given for smooth and projective varieties,
  but a different notation and grading is used.
  The notation in op.~cit.~is~$\operatorname{HA}_{\bullet,*}(X)$
  for what we denote~$\hochserre_\bullet^*(X)$,
  and the Serre functor is not used with the shift.
  To make the definition generalise to the more general setting of smooth and proper dg~categories
  (see \cref{appendix:hochschild-serre-dg} and the next remark)
  it is necessary to use the \emph{actual} Serre functor,
  and not just~$-\otimes\omega_{\cX}$,
  as there is no notion of an unshifted Serre functor in general.
\end{remark}

\begin{remark}[Generalisation to dg categories]
  \label{remark:dg-generalisation}
  This definition has an obvious generalisation for a smooth and proper dg category,
  so that a Serre functor exists.
  This is discussed in \cref{appendix:hochschild-serre-dg}.
  See also \cite[Section 4]{MR4229602} where similar notions are discussed
  in the context of stable $\infty$-categories,
  but the coefficients used are related to group actions
  and not powers of the Serre functor.
\end{remark}

\begin{remark}[Derived invariance]
  \label{remark:functoriality}
  Hochschild--Serre cohomology of a variety is a derived invariant \cite[Theorem~2.1.8]{MR1998775},
  thus implying that Hochschild cohomology and Hochschild homology are derived invariants.
  Using the general definition of \cref{appendix:hochschild-serre-dg} for dg~categories,
  it follows from \cref{theorem:agreement,theorem:hochschild-serre-morita-invariant}
  that Hochschild--Serre cohomology, as a bigraded algebra, is also a derived invariant for smooth proper Deligne--Mumford stacks,
  which is important for geometric applications to Hilbert schemes of points on surfaces. See also \cref{cor:DerivedInvariance} for a direct proof in the geometric setting without appealing to dg categories. This is a special case of \cref{proposition:hochschild-with-values-invariant}, which shows that, under a certain compatibility condition, Hochschild homology with coefficients is a derived invariant.
\end{remark}

It is moreover possible to upgrade the functoriality properties of Hochschild--Serre cohomology of varieties
so that it becomes functorial for \'etale morphisms,
but we will not need this in the main body of the text,
so this discussion is relegated to \cref{section:etale-functoriality}.

%\begin{remark}
%  In analogy with \cref{definition:hochschild-with-coefficients},
%  we could also define \emph{Hochschild cohomology with coefficients} in some object $\cE\in \derived^\bounded(\mathcal{X})$ by
%  \begin{equation}
%    \hochschild^*(\cX, \cE)\colonequals
%    \RHom_{\cX\times \cX}(\Delta_*\mathcal{O}_{\cX}, \Delta_*\cE)\,.
%  \end{equation}
%  However, this would not give anything new as
%  \begin{equation}
%    \hochschild^*(\cX, \cE)\cong \hochschild_*(\cX, \cE\otimes \omega_\cX^{-1}[-d_\cX]) \,.
%  \end{equation}
%  The reason that we prefer to formulate everything in terms of Hochschild homology
%  (and not cohomology) with coefficients
%  is that it leads to cleaner formulae later on.
%\end{remark}

\subsection{The Hochschild--Kostant--Rosenberg decomposition}
\label{subsection:hkr}
In this section we consider the case where $\mathcal{X}$ is isomorphic to
a global quotient stack by a finite group.
Let $M$ be a smooth proper variety defined over $\field$,
and $G$ a finite group acting on $M$.
Let $\mathcal{X}\colonequals[M/G]$ be the quotient stack,
which is a smooth proper orbifold.

For any $g\in G$, we denote by $M^g$ the fixed locus of $g\in \Aut(M)$,
which is a smooth closed subvariety of $M$.
Their disjoint union is denoted by
\begin{equation}
  \inertia_GM\coloneq \coprod_{g\in G}M^g,
\end{equation}
which is sometimes called the \textit{inertia variety} of the action $G$ on $M$.
The inertia variety~$\inertia_GM$ admits a natural $G$-action as follows:
for any $g, h\in G$, there is an isomorphism
\begin{equation}
  \label{equation:isomorphism-conjugacy}
  h\cdot\colon M^g\xrightarrow{\simeq} M^{hgh^{-1}}.
\end{equation}
Taking the disjoint union for $g$ running through $G$,
we can define the action of $h$ (and~$G$) on $\inertia_GM$.

Using this action, the \emph{inertia stack}
$\inertia\mathcal{X}\coloneq\mathcal{X}\times_{\mathcal{X}\times\mathcal{X}}\mathcal{X}$
can itself be described as a global quotient stack by the same finite group $G$:
\begin{equation}
  \inertia\mathcal{X}\cong [\inertia_GM/G].
\end{equation}

The usual Hochschild--Kostant--Rosenberg isomorphism for varieties
(which we state in \cref{proposition:hkr-varieties})
admits the following orbifold version\footnote{
  The tools to prove the Hochschild--Kostant--Rosenberg decomposition
  for arbitrary smooth Deligne--Mumford stacks,
  which are not global quotients by finite groups,
  seem to be lacking from the literature.
},
essentially due to Arinkin--\caldararu--Hablicsek \cite{MR4003476}.
\begin{proposition}[Hochschild--Kostant--Rosenberg for global quotient stacks]
  \label{proposition:orbifold-hkr}
  Let $\mathcal{X}=[M/G]$ be a smooth global quotient orbifold defined over $\field$ as above.
  Let $\mathcal{E}\in \derived^\bounded([M/G])$ be given by a $G$-linearised object $E\in \derived^\bounded(M)$.
  Then
  \begin{equation}
    \label{equation:orbifold-hkr}
    \hochschild_*(\mathcal{X}, \mathcal{E})
    \cong
    \left(\bigoplus_{g\in G} \Ho^*\left(M^g, \Sym^\bullet(\Omega_{M^g}^1[1])\otimes i_g^*E\right)\right)^G.
  \end{equation}
  Equivalently,
  \begin{equation}
    \label{equation:orbifold-hkr-conjugacy}
    \hochschild_*(\mathcal{X}, \mathcal{E})
    \cong
    \bigoplus_{[g]\in \Conj(G)} \Ho^*\left(M^g, \Sym^\bullet(\Omega_{M^g}^1[1])\otimes i_g^* E\right)^{\centraliser(g)}.
  \end{equation}
 Here, $M^g$ is the fixed locus of $g$ and $i_g\colon M^g\hookrightarrow M$ is the closed immersion,
  $\Conj(G)$ is the set of conjugacy classes of $G$,
  and $\centraliser(g)$ is the centraliser of $g$ in $G$.
  Here the action on the right-hand sides of \eqref{equation:orbifold-hkr} and \eqref{equation:orbifold-hkr-conjugacy}
  are induced by the action \eqref{equation:isomorphism-conjugacy}.
\end{proposition}

The graded vector spaces
%$\Ho^*(M^g, \Sym^\bullet(\Omega_{M^g}^1[1])\otimes i_g^* E)$
occuring on
the right-hand sides of \eqref{equation:orbifold-hkr} and \eqref{equation:orbifold-hkr-conjugacy}
denote (hyper-)cohomology of the complex $\Sym^\bullet(\Omega_{M^g}^1[1])\otimes i_g^* E$.
In other words,
\begin{equation}
  \Ho^*\left(M^g, \Sym^\bullet(\Omega_{M^g}^1[1])\otimes i_g^*E\right)
  \colonequals
  \mathbf{R}\Gamma\left(M^g, \Sym^\bullet(\Omega_{M^g}^1[1])\otimes i_g^*E\right).
\end{equation}

\begin{proof}
  Since in \cite{MR4003476} only the cases $\cE=\mathcal O_\cX$ and $\cE=\omega_\cX^{-1}[-d_\cX]$
  (which give Hochschild homology and Hochschild cohomology without coefficients) are explicitly worked out,
  we include a full proof here for the sake of completeness.

%  For any $g\in G$, let $c_g=d_M-d_{M^g}$ be the codimension of the fixed locus,
%  and recall that the relative canonical bundle of $i_g\colon M^g\to M$ is $\omega_g=\omega_{M^g}\otimes i_g^*\omega_M^\vee$.

  We have the following commutative diagram
  \begin{equation}
    \begin{tikzcd}
      \loops\mathcal{X}
      \arrow[drr, bend left, "p'"]
      \arrow[ddr, bend right, "q'" ']
      & & \\
      & \inertia\mathcal{X}  \arrow[r, "p"] \arrow[d, "q"] \arrow[ul, ]
      & \mathcal{X} \arrow[d, "\Delta"] \\
      & \mathcal{X} \arrow[r, "\Delta"]
      & \mathcal{X}\times \mathcal{X}
    \end{tikzcd}
  \end{equation}
  where $\loops\mathcal{X}\coloneq \mathcal{X}\times^{\mathbf{R}}_{\mathcal{X}\times  \mathcal{X}}  \mathcal{X}$
  is the derived fiber product, called the \textit{free loop space}.

  By \cite[Corollary 1.17]{MR4003476}, $p'$ and $q'$ are homotopic,
  hence by base change and the projection formula, we have
  \begin{equation}
    \label{equation:Delta-Delta}
    \mathbf{L}\Delta^*\Delta_*\cong \mathbf{R}q'_*\mathbf{L}p'^*
    \cong
    \mathbf{R}p'_*\mathbf{L}p'^*\cong -\otimes \mathbf{R}p'_*(\mathcal{O}_{\loops\mathcal{X}})\cong -\otimes \mathbf{R}p_*\Sym^\bullet(\Omega_{\inertia\mathcal{X}}^1[1]),
  \end{equation}
  where the last isomorphism uses \cite[Theorem 1.15]{MR4003476},
  which says that $\loops\mathcal{X}$ is isomorphic (over $\mathcal{X}\times \mathcal{X}$)
  to the total space of the shifted tangent bundle of $\inertia\mathcal{X}$.

  Therefore,
  \begin{equation}
    \begin{aligned}
      \hochschild_*(\mathcal{X}, \mathcal{E})
      &\cong \RGamma(\cX\times \cX,\Delta_{*}\mathcal{O}_{\cX}\otimes \Delta_{*}\mathcal{E}) \\
      &\cong \RGamma(\cX\times \cX,\Delta_{*}\mathbf{L}\Delta^*\Delta_{*}\mathcal{E}) \\
      &\cong \RGamma(\cX,\mathbf{L}\Delta^*\Delta_{*}\mathcal{E}) \\
      &\cong \RGamma(\cX,\mathcal{E}\otimes \mathbf{R}p_*\Sym^\bullet(\Omega_{\inertia\mathcal{X}}^1[1])) \\
      &\cong \RGamma(\inertia\cX,p^*\mathcal{E}\otimes \Sym^\bullet(\Omega_{\inertia\mathcal{X}}^1[1])) \\
      &\cong\left(\bigoplus_{g\in G} \mathbf{R}\Gamma\left(M^g, \Sym^\bullet(\Omega^1_{M^g}[1])\otimes i_g^* E\right) \right)^G
    \end{aligned}
  \end{equation}
  where the fourth isomorphism uses  \eqref{equation:Delta-Delta} and the last isomorphism follows  the fact that $-^G$ is an exact functor since we are in characteristic~zero.  This proves \eqref{equation:orbifold-hkr},
  and \eqref{equation:orbifold-hkr-conjugacy} follows immediately.
\end{proof}

\begin{remark}
  For later use, let us make the group action on
  the right-hand sides of \eqref{equation:orbifold-hkr} and \eqref{equation:orbifold-hkr-conjugacy} more precise.
  In the statement of \cref{proposition:orbifold-hkr},
  for any $h\in G$,
  the action of $h$ sends the summand indexed by $g$ isomorphically
  to the one indexed by $hgh^{-1}$ via \eqref{equation:isomorphism-conjugacy}:
  \begin{equation}
    h\colon M^g\xrightarrow{\simeq} M^{hgh^{-1}}.
  \end{equation}
  In particular, $\centraliser(g)$ acts on $M^g$.
  The isomorphisms $\Omega_{M^g}^1\xrightarrow{\simeq} h^*\Omega_{M^{hgh^{-1}}}^1$ are the canonical ones.
  The isomorphisms $h^*i_{hgh^{-1}}^*(E)\cong i_g^*h^*(E)\xrightarrow{\simeq}i_g^*E$ come from the linearisation of $E$.
\end{remark}

\begin{remark}
  In the case that $M=\Spec R$ is affine,
  coherent sheaves on $[M/G]$ are the same as finitely generated modules over the crossed product algebra $R\#G$.
  For the Hochschild cohomology of these crossed product algebras,
  decompositions analogous to \eqref{equation:orbifold-hkr-conjugacy}
  are given in \cite[Section~3]{MR2058782} and \cite{anno2005HH}.
  There, it is also described what happens with the ring structure of Hochschild cohomology under this decomposition.
  For $M$ non-affine this remains an open problem;
  see \cite{arXiv:2101.06276} and \cite{MR4057490} for some partial result and speculation on this.
\end{remark}

\subsection{Orbifold Hochschild--Kostant--Rosenberg in terms of Hodge groups of inertia}
As a motivation for the definitions that will follow,
and calculations in the surface case,
we first spell out \cref{proposition:orbifold-hkr}
in the special case where~$G$ is the trivial group.
\begin{proposition}[Hochschild--Kostant--Rosenberg with coefficients on a variety]
  \label{proposition:hkr-varieties}
  Let $X$ be a smooth proper variety over $\field$ of dimension $d_X$. For any $E\in \derived^\bounded(X)$, we have
  \begin{equation}
    \label{equation:hkr-varieties}
    \hochschild_*(X, E)\cong \mathbf{R}\Gamma(X, \Sym^\bullet(\Omega_X^1[1])\otimes E),
  \end{equation}
  where $\Sym^\bullet$ is taken in the graded sense,
  so that~$\Sym^\bullet(\Omega_X^1[1])=\bigoplus_{p\geq 0} \bigwedge^p\Omega_X^1[p]$.
  Taking cohomology, for any $j\in \mathbb{Z}$,
  \begin{equation}
    \label{equation:hkr-varieties-degree}
    \hochschild_j(X, E)\cong \bigoplus_{q-p=j} \HH^q(X, \Omega_X^p\otimes E).
  \end{equation}
  In particular,
  \begin{equation}
    \label{equation:hochschild-serre-decomposition}
    \hochserre_k^j(X)\cong \bigoplus_{p+q=j+kd_X}\HH^q\left(X, \bigwedge^p\tangent_X\otimes \omega_X^{\otimes k}\right),
  \end{equation}
  which is (possibly) non-zero only in degrees~$[-kd_X,2d_X-kd_X]$.
\end{proposition}
For~$k=0$ resp.~$k=1$ we obtain the usual Hochschild--Kostant--Rosenberg decomposition for Hochschild cohomology resp.~Hochschild homology.

\begin{definition}
  Let~$X$ be a smooth and proper variety.
  Let $E\in \derived^\bounded(X)$.
  For any integers $p, q$, define the \emph{twisted Hodge group}
  \begin{equation}
    \label{equation:twisted-hodge-variety}
    \HH^{p,q}(X, E)\coloneqq \HH^q(X, \Omega_X^p\otimes E).
  \end{equation}
\end{definition}
Then \eqref{equation:hkr-varieties-degree} says that
\begin{equation}
  \label{equation:hkr-twisted-hodge}
  \hochschild_j(X, E)\cong \bigoplus_{q-p=j} \HH^{p,q}(X, E).
\end{equation}
In order to give an analogue of \eqref{equation:hkr-twisted-hodge} for quotient stacks,
we introduce the following analogue of twisted Hodge groups for orbifolds.
\begin{definition}
  Let $\mathcal{X}$ be a smooth and proper orbifold.
  Let $\mathcal{E}\in \derived^\bounded(\mathcal{X})$.
  For any integer $p, q$, define the \emph{twisted Hodge group}
  \begin{equation}
    \HH^{p,q}(\mathcal{X}, \mathcal{E})\coloneqq \HH^q(\mathcal{X}, \Omega^p_{\mathcal{X}}\otimes \mathcal{E}).
  \end{equation}
\end{definition}
Of particular interest to us are the twisted Hodge groups of the inertia stack.
In the sequel, we often abuse notation to denote $\HH^{p,q}(\inertia\mathcal{X}, p^*\mathcal{E})$
simply by $  \HH^{p,q}(\inertia\mathcal{X}, \mathcal{E})$,
where $p\colon \inertia\mathcal{X}\to \mathcal{X}$ is the canonical map.

When $\mathcal{X}=[M/G]$, $\mathcal{E}\in \derived^\bounded(\mathcal{X})$,
and $E\in \derived^\bounded(M)$ be as in \cref{proposition:orbifold-hkr},
we have
\begin{equation}
  \label{eq:InertiaHodgeGroups}
  \HH^{p,q}(\inertia\mathcal{X}, \mathcal{E})
  =
  \left(\bigoplus_{g\in G} \HH^{p,q}(M^g, E|_{M^g})\right)^G
  \cong
  \bigoplus_{[g]\in \Conj(G)} \HH^{p,q}\left(M^g, E|_{M^g}\right)^{\centraliser(g)}.
\end{equation}

\begin{corollary}
  \label{corollary:orbifold-hkr-degree}
  With notation as in \cref{proposition:orbifold-hkr}
  and for any $j\in \mathbb{Z}$,
  we have
  \begin{equation}
    \hochschild_j(\mathcal{X}, \mathcal{E})\cong \bigoplus_{q-p=j}\HH^{p,q}(\inertia\mathcal{X}, \mathcal{E}).
  \end{equation}
\end{corollary}

\begin{proof}
  We take cohomology in \eqref{equation:orbifold-hkr} and plug in \eqref{eq:InertiaHodgeGroups}.
\end{proof}

We can collect the Hodge groups with coefficients in the bigraded vector space
\begin{equation}
  \label{eq:Hodgegradedstack}
  \Ho^{\#,\star}(\inertia\mathcal{X},\mathcal{E})
  \coloneqq
  \left(\bigoplus_{g\in G} \Ho^{\#,\star}(M^g, E|_{M^g})\right)^G
  \cong
  \bigoplus_{[g]\in \Conj(G)} \Ho^{\#,\star}\left(M^g, E|_{M^g}\right)^{\centraliser(g)},
\end{equation}
where $\Ho^{\#,\star}(M^g, E|_{M^g}))=\Ho^{\star}(M^g,\Omega^{\#}_{M^g}\otimes E_{\mid M^g})$.
Then we can rephrase \cref{corollary:orbifold-hkr-degree} as
\begin{corollary}
  \label{corollary:orbifold-hkr-total}
  Turning the bigraded vector space $\Ho^{\#,\star}(\mathcal{X}, \mathcal{E})$ into a (single-)graded vector space with grading $*=\star-\#$,
  this graded vector space is isomorphic to the Hochschild homology with coefficients in $\cE$:
  \begin{equation}
    \Ho^{\#,\star}(\inertia\mathcal{X}, \mathcal{E})\cong \hochschild_*(\cX, \cE)\quad\text{for $*=\star-\#$}\,.
  \end{equation}
\end{corollary}

%\begin{remark}
%  By specialising $\mathcal{E}=\omega_{\mathcal{X}}^{\vee}[-d_{\mathcal{X}}]$,
%  we get the following Hochschild--Konstant--Rosenberg decomposition for the Hochschild cohomology of $\mathcal{X}=[M/G]$:
%  \begin{equation}
%    \label{equation:orbifold-hkr-hochschild-cohomology}
%    \hochschild^j(\mathcal{X})\cong \bigoplus_{p+q=j} \operatorname{HT}^{p,q}(\inertia\mathcal{X})
%  \end{equation}
%  where
%  \begin{equation}
%    \operatorname{HT}^{p,q}(\inertia\mathcal{X})
%    \coloneqq
%    \left(\bigoplus_{g\in G} \HH^q\left(M^g, \bigwedge^p\tangent_{M^g}\otimes \omega_g[-c_g]\right)\right)^G,
%  \end{equation}
%  where $\omega_g=\det N_{M^g/M}$ and $c_g$ is the codimension of $M^g$ in $M$.
%\end{remark}

\section{Symmetric quotient stacks}
\label{section:proofs}
This section contains the proof of our result \cref{theorem:main-hochschild} describing the Hochschild homology of symmetric quotient stacks with a certain type of coefficients. Afterwards, we work out most of the results stated in the introduction as corollaries.

Let $X$ be a smooth proper variety defined over a field $\field$ of characteristic zero.
For any integer $n\geq 1$, let the symmetric group $\symgr_n$ act on $X^n$ by permuting the factors.
Denote by $[\Sym^nX]\coloneq [X^n/\symgr_n]$ the quotient stack,
called the \emph{$n$th symmetric quotient stack of $X$}.

\begin{notation}
  \label{notation:cyclic-invariants}
  For any $F\in \derived^\bounded(X)$,
  endow $F^{\boxtimes n}\in \derived^\bounded(X^n)$ with the natural $\symgr_n$-linearisation,
  giving rise to an object $F^{\{n\}}\in \derived^\bounded([\Sym^nX])$. The pull-back of $F^{\{n\}}$ to the inertia stack $\inertia[\Sym^nX]$ via the natural morphism $\inertia[\Sym^nX]\to [\Sym^nX]$ is also denoted by the same notation.
\end{notation}

For any element $g\in \symgr_n$,
we view it as a permutation of the set $\{1, \dots, n\}$,
and denote by $\orbits(g)$ the set of the orbits of the permutation $g$.
For each orbit $o\in\orbits(g)$,
let $\#o$ be the cardinality (or length) of the orbit $o$.
In what follows, for any finite set $S$,
we use the notation $X^S$ to denote the variety $\Map(\Spec(\prod_S \field), X)$
parametrizing $\field$-morphisms from $S\coloneq\coprod_S\Spec\field$ to $X$.

\subsection{Yoga on symmetric powers of (multi)graded vector spaces}
\label{subsection:sym-of-vector-spaces}
We recall some basic properties of symmetric powers of (multi)graded vector spaces.

\paragraph{Graded case}
Given a graded vector space $V=\bigoplus_{i\in \IZ} V_i$,
there are two ways to form the symmetric product
\begin{equation}
  \Sym^\bullet(V)=\bigoplus_{n\ge 0}\Sym^nV.
\end{equation}

\begin{itemize}
  \item For the \emph{ordinary} symmetric product,
    we consider the symmetric group $\sym_n$ acting on $V^{\otimes n}$
    by permutation of the tensor factors:
    $\sigma\cdot(v_1\otimes\dots\otimes v_n)\colonequals v_{\sigma^{-1}(1)}\otimes \dots\otimes v_{\sigma^{-1}(n)}$.
    Then $\Sym^nV\subset V^{\otimes n}$ is defined to be the subspace of the invariants under this action.
    As we are working in characteristic zero, this is isomorphic to the coinvariants.

  \item For the \emph{graded} symmetric product,
    we consider the $\sym_n$-action on $V^{\otimes n}$
    by permutation of the tensor factors
    together with a minus sign whenever two factors of odd degree switch positions.
    Concretely, for a transposition of neighbours $\tau=(j,\, j+1)$,
    we set
    $\tau\cdot(v_1\otimes\dots\otimes v_n)\colonequals(-1)^{\deg (v_j) \cdot\deg(v_{j+1})} v_{\tau^{-1}(1)}\otimes \dots\otimes v_{\tau^{-1}(n)}$.
    Then $\Sym^nV\subset V^{\otimes n}$ denotes the space of the invariants under this action.
\end{itemize}
If not stated otherwise $\Sym^\bullet$ always denotes the graded symmetric product in what follows.

\paragraph{Multigraded case}
In several places in this paper,
we will have multigraded vector spaces,
that is,
a $\mathbb{Z}^m$-graded (also referred to as~$m$-multigraded) vector space,
for some integer $m\geq 1$.

Given an $m$-multigraded vector space $V$,
and some multi-index $\mathbf{d}\in \IZ^m$,
we define the \emph{shift} of $V$ by $\mathbf{d}$ as the $m$-multigraded vector space~$V[\mathbf{d}]$
whose graded pieces are given by~$(V[\mathbf{d}])_{\mathbf{i}}\colonequals V_{\mathbf{i}+\mathbf{d}}$,
for any $\mathbf{i}\in\mathbb{Z}^m$.

In the sequel, when we take the symmetric product of some multigraded vector space,
it is taken in the graded sense with respect to some degrees,
but in the ordinary sense with respect to the other degrees.
Let us make this precise.
Let $m\in \IN$, and let
\begin{equation}
  V=\bigoplus_{(i_1,\dots,i_m)\in \IZ^m} V_{(i_1,\dots,i_m)}
\end{equation}
be an $m$-multigraded vector space.
For an homogeneous element $v\in V_{(i_1,\dots,i_m)}$,
and $k\in \{1,\dots, m\}$,
we set $\deg_k(v)\colonequals i_k$.
Let $K\subset \{1,\dots, m\}$.
We define the symmetric product~$\Sym^nV$ which is
\emph{graded with respect to the $K$-degrees, and ordinary with respect to the other degrees}
as the subspace of $\sym_n$-invariants of $V^{\otimes n}$ under the action given for $\tau=(j,\, j+1)$ by
\begin{equation}
  \label{equation:tau-action}
  \tau\cdot(v_1\otimes\dots\otimes v_n)
  \colonequals
  (-1)^{\sum_{k\in K}\deg_k(v_j) \cdot\deg_k(v_{j+1})} v_{1}\otimes \cdots v_{j+1}\otimes v_j\otimes\cdots\otimes v_{n}.
\end{equation}
We stress the dependence of $\Sym^nV$ on the choice of the ``super-degrees'' $K\subset\{1,\dots,m\}$, although not reflected in the notation.

For a fixed subset $K\subset \{1,\dots, m\}$, it is straightforward to check that for two $m$-multigraded vector spaces $U$ and $V$, we have the following canonical isomorphism of $m$-multigraded vector spaces:
\begin{equation}
  \label{eqn:TotalSymOfSum}
  \Sym^\bullet(U\oplus V)\cong \Sym^\bullet (U)\otimes \Sym^\bullet(V).
\end{equation}

\paragraph{Generating series}
The \emph{generating series} of a $m$-multigraded vector space $V$
which is finite-dimensional in each degree
is the Laurent series in $m$ variables $s_1,\dots,s_m$ given by
\begin{equation}
  \mathbb{E}_{V}
  \colonequals
  \sum_{\mathbf{i}\in \IZ^m} \dim_\field(V_{\mathbf{i}}) s^{\mathbf{i}}
  \quad\text{where}\quad s^{\mathbf{i}}=s_1^{i_1}\cdots s_m^{i_m}
  \quad\text{for}\quad\mathbf{i}=(i_1,\dots, i_m)\,.
\end{equation}
For the following two lemmas,
we fix some subset $K\subset \{1,\dots ,m\}$.
The symmetric product of any $m$-multigraded vector space is then meant to be graded with respect to the $K$-degrees
and ordinary with respect to the other degrees.

\begin{lemma}
  \label{lemma:shifted-sym}
  Let $V$ be a finite-dimensional vector space, which we consider as an $m$-multigraded vector space concentrated in degree $\mathbf{0}=(0,\dots,0)$.
  For $\mathbf{d}=(d_1,\dots,d_m)\in \IZ^m$,
  we set $|\mathbf{d}|_K\colonequals\sum_{k\in K} d_k$ to be total $K$-degree.
  We have
  \begin{equation}
    \mathbb{E}_{\Sym^\bullet(V[-\mathbf{d}])}=(1-(-1)^{|\mathbf{d}|_K}s^{\mathbf{d}})^{-(-1)^{|\mathbf{d}|_K}\dim_\field V}\,.
  \end{equation}
\end{lemma}

\begin{proof}
  The multigraded vector space $\Sym^\bullet(V[-\mathbf{d}])$ is concentrated in degrees
  which are non-negative integral multiples of $\mathbf{d}$.
  Furthermore, for any $a \in \mathbb{N}$, we have
  \begin{equation}
    \Sym^\bullet(V[-\mathbf{d}])_{a\mathbf{d}}
    =
    \begin{cases}
      \Sym^a V & \text{if $|\mathbf{d}|_K$ is even,} \\
      \bigwedge^a V & \text{if $|\mathbf{d}|_K$ is odd.}
    \end{cases}
  \end{equation}
  %  $\Sym^\bullet(V[-\mathbf{d}])_{a\mathbf{d}}=\Sym^a V$ or $\Sym^\bullet(V[-\mathbf{d}])_{a\mathbf{d}}=\wedge^a V$, depending on whether $d$ is even or odd.
  Indeed, by \eqref{equation:tau-action}, the action of $\tau=(j,\,j+1)$
  on $v_1\otimes \dots\otimes v_n$ with $v_i\in V[-\mathbf{d}]$ is given by
  \begin{equation}
    \tau\cdot(v_1\otimes\dots\otimes v_n)
    \colonequals
    (-1)^{\sum_{k\in K} d_k^2}  v_{1}\otimes \cdots v_{j+1}\otimes v_j\otimes\cdots\otimes v_{n};
  \end{equation}
  and we have
  \begin{equation}
    (-1)^{\sum_{k\in K} d_k^2}=(-1)^{\sum_{k\in K} d_k}=(-1)^{|\mathbf{d}|_K}\,.
  \end{equation}

  Therefore, we get the asserted
  \begin{equation}
    \sum_{\mathbf{i}\in \IZ^{m}} \dim_\field\bigl(\Sym^\bullet(V[-\mathbf{d}])_{\mathbf{i}}\bigr) s^{\mathbf{i}}
    =
    \begin{cases}
      \sum_{a\in\IN}\dim_\field\bigl(\Sym^a V\bigr) s^{a\mathbf{d}}=(1-s^{\mathbf{d}})^{-\dim_\field V} & \text{for $|\mathbf{d}|_K$ even,}\\
      \sum_{a\in\IN}\dim_\field\bigl(\bigwedge^a V\bigr) s^{a\mathbf{d}}= (1+s^{\mathbf{d}})^{\dim_\field V} & \text{for $|\mathbf{d}|_K$ odd,}
    \end{cases}
  \end{equation}
  where the second equality (in each case) is well-known; see, e.g., \cite[Lemma~A.3]{MR3788855}.
\end{proof}

\begin{lemma}
  \label{lemma:graded-sym-dim}
  Let $W$ be a finite dimensional $m$-multigraded vector space.
  Then the generating series of its total symmetric power $\Sym^\bullet W$ is given as follows:
  \begin{equation}
    \mathbb{E}_{\Sym^\bullet (W)}
    =
    \prod_{\mathbf{d}\in \IZ^m}(1-(-1)^{|\mathbf{d}|_K} s^{\mathbf{d}})^{-(-1)^{|\mathbf{d}|_K}\dim_\field(W_{\mathbf{d}})},
  \end{equation}
  where again $|\mathbf{d}|_K=\sum_{k\in K} d_k$ for $\mathbf{d}=(d_1,\dots, d_m)$.
\end{lemma}

\begin{proof}
  The compatibility of the symmetric product with direct sums \eqref{eqn:TotalSymOfSum} gives an isomorphism of graded vector spaces
  \begin{equation}
    \Sym^\bullet W
    =
    \Sym^\bullet\left(\bigoplus_{\mathbf{d} \in  \IZ^m}W_{\mathbf{d}}[-\mathbf{d}]\right)
    \cong
    \bigotimes_{\mathbf{d}\in \IZ^m} \Sym^\bullet(W_{\mathbf{d}}[-\mathbf{d}]).
  \end{equation}
  Since the generating series of a tensor product is the product of generator series of its factors, we obtain
  \begin{equation}
    \mathbb{E}_{\Sym^\bullet(W)}
    =
    \prod_{\mathbf{d}\in \mathbb{Z}^m}\mathbb{E}_{\Sym^\bullet(W_{\mathbf{d}}[-\mathbf{d}])}.
  \end{equation}
  Combining this with \cref{lemma:shifted-sym} gives the desired formula.
%  \begin{align*}
%  \sum_{\mathbf{i}\in \IZ} \dim\bigl( \Sym^\bullet(W)_{\mathbf{i}} \bigr)s^{\mathbf{i}}
%  &=
%  \prod_{\mathbf{d}\in \IZ^m}\left(\sum_{\mathbf{i}\in \IZ^{m}} \dim\bigl(\Sym^\bullet(W_{\mathbf{d}}[-\mathbf{d}])_{\mathbf{i}}\bigr) s^{\mathbf{i}}\right)
%  \\&=
%  \prod_{\mathbf{d}\in \IZ^m}(1-(-1)^d s^{\mathbf{d}})^{-(-1)^d\dim(W_{\mathbf{d}})}\,. \qedhere
%  \end{align*}
\end{proof}

\subsection{Twisted Hodge groups of the inertia}
\label{subsection:main-result}
We first express the twisted Hodge groups of the inertia stack of the symmetric quotient stack $[\Sym^nX]$ in terms of the twisted Hodge groups of $X$ and the action of the symmetric group.
\begin{proposition}
  \label{prop:HA-Sym-Invariant-Expression}
  Let $F\in \derived^\bounded(X)$ and $F^{\{n\}}$ be as in \cref{notation:cyclic-invariants}.
  We have the following isomorphism of bigraded vector spaces:
  \begin{equation}
  \label{eqn:HA-Sym-1}
    \Ho^{\#,\star}\left(\inertia[\Sym^nX], F^{\{n\}}\right)\cong
    \left(\bigoplus_{g\in \symgr_n}\bigotimes_{o\in \orbits(g)} \Ho^{\#,\star}\left(X, F^o|_X \right)\right)^{\symgr_n}.
  \end{equation}
  Here $F^o=\bigotimes_{t\in o}\pr_t^*(F)\in \derived^\bounded(X^o)$
  and $F^o|_X$ is its restriction to the small diagonal (identified with $X$) of $X^o$.
\end{proposition}

\begin{proof}
  We first observe the following canonical identifications: for any $g\in \symgr_n$,
  \begin{itemize}
    \item $(X^n)^g=X^{\orbits(g)}$, which is a partial diagonal of $X^n$; we denote the closed immersion $i_g\colon X^{\orbits(g)}\hookrightarrow X^n$;
%    \item $d_M=nd_X$ and $d_{M^g}=\#\orbits(g)d_X$;
    \item $\Omega_{X^{\orbits(g)}}=\bigoplus_{o\in \orbits(g)}p_o^*\Omega_X$, where the projection $p_o\colon X^{\orbits(g)}\to X$ is induced by $\{o\}\subset \orbits(g)$;
%     \item $\omega_{M^g}=\bigboxtimes_{o\in \orbits(g)}\omega_X\coloneq\bigotimes_{o\in \orbits(g)}p_o^*\omega_X$.
%     \item $i_g^*\omega_M=\bigboxtimes_{o\in \orbits(g)}(\omega_{X^{o}}|_X)$, where for each $o$, $X$ is identified with the small diagonal of $X^o$. Clearly  $X^o$ is isomorphic to $X^{\#o}$, hence $i_g^*\omega_M\cong \bigboxtimes_{o\in \orbits(g)}\omega_X^{\#o}$. But we need to keep the canonical identification for now in order to keep track of the $\symgr_n$-action for later use.
    \item $i^*_g(F^{\boxtimes n})=\bigboxtimes_{o\in \orbits(g)}F^{o}|_X$ where for each $o$, $X$ is identified with the small diagonal of $X^o$. Note that $F^o|_X$ is isomorphic to $F^{\otimes \# o}$, but we keep the canonical identification here to remember the action.
  \end{itemize}

  Plugging these into \eqref{eq:Hodgegradedstack}, we find that
  \begin{equation}
    \Ho^{\#,\star}(\inertia[\Sym^nX], F^{\{n\}})\cong \left(\bigoplus_{g\in \symgr_n} \Ho^\star\left(X^{\orbits(g)}, \bigwedge^\#\left(\bigoplus_{o\in \orbits(g)}p_o^*\Omega_X\right)\otimes \left( \bigboxtimes_{o\in \orbits(g)} F^o|_X\right)\right)\right)^{ \symgr_n}.
  \end{equation}
  Using that the total wedge power of a direct sum is the tensor product of the total wedge powers of the summands,
  and the Künneth formula, we have the following isomorphisms of bigraded vector spaces:
  % trick from https://tex.stackexchange.com/a/32095/210
  \belowdisplayskip=-12pt
  \begin{equation}
    \begin{aligned}
      \Ho^{\#,\star}(\inertia[\Sym^nX], F^{\{n\}})
      &\cong\left(\bigoplus_{g\in \symgr_n} \Ho^\star\left(X^{\orbits(g)}, \bigboxtimes_{o\in \orbits(g)} \left(\bigwedge^\#\Omega_X\otimes F^o|_X\right)\right)\right)^{\symgr_n} \\
      &\cong\left(\bigoplus_{g\in \symgr_n}\bigotimes_{o\in \orbits(g)} \Ho^\star\left(X, \bigwedge^\#\Omega_X\otimes F^o|_X\right)\right)^{ \symgr_n} \\
      &\cong\left(\bigoplus_{g\in \symgr_n}\bigotimes_{o\in \orbits(g)} \Ho^{\#,\star}\left(X, F^o|_X\right)\right)^{\symgr_n}\,.
    \end{aligned}
  \end{equation}\qedhere
\end{proof}

Before further simplifying \eqref{eqn:HA-Sym-1},
we first recall the following well-known result on centralisers of permutations.

\begin{lemma}\label{lemma:centraliser}
  Let $g\in \symgr_n$ viewed as a permutation on $\{1, \dots, n\}$,
  and let $\orbits(g)$ be the set of its orbits.
  Let $\lambda_j$ be the number of length-$j$ orbits.
  Then the centraliser of $g$ is
  \begin{equation}
    \centraliser(g)=\left(\prod_{o\in \orbits(g)} \mu_o\right)\rtimes \prod_{j=1}^n\symgr_{\lambda_j},
  \end{equation}
  where $\mu_o$ is the cyclic group of order $\#o$ generated by $g|_o$
  (the permutation that is $g$ on $o$ but $\id$ elsewhere),
  and $\symgr_{\lambda_j}$ permutes the length-$j$ orbits (with fixed bijections between elements among those orbits).
\end{lemma}

\begin{proof}
  See, e.g., \cite[Proposition~1.1.1]{MR1824028}.
\end{proof}

\begin{notation}
  \label{notation-cyclic}
  For $n\in \IN\setminus\{0\}$, we consider the cyclic permutation $\sigma_n\colonequals(1\,2\, \cdots \, n)\in \sym_n$ of order $n$.
  For $F\in \derived^\bounded(X)$, we consider $F^{\otimes n} \in \derived^\bounded(X)$
  equipped with the $\symgr_n$-action given by permuting the tensor factors,
  and denote its subobject of $\sigma_n$-invariants by
  \begin{equation}
    F^{\langle n\rangle}\colonequals\bigl(F^{\otimes n}\bigr)^{\langle \sigma_n \rangle}\in \derived^\bounded(X)\,.
  \end{equation}
  To avoid confusion, we stress that $F^{\langle n\rangle}$ is an object on the variety $X$, but not on the orbifold $[\Sym^nX]$.
\end{notation}

\begin{theorem}
  \label{theorem:main}
  Let $X$ be a smooth proper algebraic variety over a field $\field$ of characteristic zero, and let $F\in \derived^\bounded(X)$.
  For every $n\in \mathbb{N}^*$, we have
  \begin{equation}
    \label{eqn:HodgeInertiaSym}
   \Ho^{\#,\star}(\inertia[\Sym^nX], F^{\{ n\}})
    \cong\bigoplus_{\nu\dashv n}\left(\bigotimes_{i\geq 1} \Sym^{\lambda_i}\Ho^{\#,\star}(X, F^{\langle i\rangle})\right)\,.
  \end{equation}
  where $\nu$ runs through all partitions of $n$, and $\lambda_i$ is the number of $i$'s in the partition $\nu$.
\end{theorem}

\begin{proof}
  We use \cref{prop:HA-Sym-Invariant-Expression}. In \eqref{eqn:HA-Sym-1}, by passing to the conjugacy classes, we obtain that
  \begin{equation}
    \begin{aligned}
      \Ho^{\#,\star}(\inertia[\Sym^nX], F^{\{n\}})
      &\cong\bigoplus_{[g]\in \Conj(\symgr_n)} \left(\bigotimes_{o\in \orbits(g)} \Ho^{\#,\star}\left(X,F^o|_X \right)\right)^{\centraliser(g)} \\
      &\cong\bigoplus_{[g]\in \Conj(\symgr_n)} \left(\bigotimes_{o\in \orbits(g)} \Ho^{\#,\star}\left(X,F^o|_X \right)^{\mu_o}\right)^{\symgr_\lambda},
    \end{aligned}
  \end{equation}
  where the groups $\mu_o$ and $\sym_\lambda\colonequals\prod_{j=1}^n\symgr_{\lambda_j}$
  are the semi-direct factors of $\centraliser(g)$ as described in \cref{lemma:centraliser}.

  The cyclic group $\mu_o\cong \mu_{\# o}$ acts on $F^o|_X\cong F^{\#o}$ by cyclic permutation of the tensor factors, and acts trivially on $X$, hence also on $\Omega_X$.
  Hence,
  \begin{equation}
    \Ho^{\#,\star}\left(X,F^o|_X \right)^{\mu_o}\cong \Ho^{\#,\star}\left(X, F^{\langle\#o\rangle}\right)\,,
  \end{equation}
  where $F^{\langle\#o\rangle}$ is as in \cref{notation-cyclic}.
  The group $\symgr_\lambda$ acts by permuting the orbits of the same lengths. Hence,
  \begin{equation}
    \left(\bigotimes_{o\in \orbits(g)} \Ho^{\#,\star}\left(X,F^o|_X \right)^{\mu_o}\right)^{\symgr_\lambda}
    \cong
    \left(\bigotimes_{o\in \orbits(g)} \Ho^{\#,\star}\left(X,F^{\langle\#o\rangle} \right)\right)^{\symgr_\lambda}
    \cong
    \bigotimes_{i\geq 1} \Sym^{\lambda_i}\Ho^{\#,\star}(X, F^{\langle i\rangle})\,.
  \end{equation}
  The assertion follows by the well-known bijection between conjugacy classes of $\sym_n$ and partitions of $n$.
\end{proof}

We can summarise the formulae \eqref{eqn:HodgeInertiaSym} for fixed $X$ and $F$ but varying $n$ in the following way.

\begin{corollary}
  \label{corollary:main}
  We have an isomorphism of trigraded vector spaces
  \begin{equation}\label{eq:FockHodge}
      \bigoplus_{n\geq 0}\Ho^{\#,\star}(\inertia[\Sym^nX], F^{\{ n\}})t^n
    \cong
    \Sym^\bullet\left(\bigoplus_{i\geq 1}\Ho^{\#,\star}(X, F^{\langle i\rangle})t^i\right)\,,
  \end{equation}
  where the symmetric power on the right-hand side is graded with respect to the double grading of  $\Ho^{\#,\star}$, but ordinary with respect to the grading given by exponents of the formal variable $t$.
\end{corollary}

\begin{proof}
  By \eqref{eqn:HodgeInertiaSym} in \cref{theorem:main}, we have
  \begin{equation}
    \bigoplus_{n\geq 0}\Ho^{\#,\star}(\inertia[\Sym^nX], F^{\{ n\}})t^n
    \cong \bigoplus_{\lambda_1, \lambda_2, \cdots}\left(\bigotimes_{i\geq 1} \Sym^{\lambda_i}\Ho^{\#,\star}(X, F^{\langle i\rangle})\right)t^{\sum_{i}i\lambda_i}\,
    \cong \bigoplus_{\lambda_1, \lambda_2, \cdots}\bigotimes_{i\geq 1} \Sym^{\lambda_i}\left(\Ho^{\#,\star}(X, F^{\langle i\rangle})t^i\right)\,
  \end{equation}
  which is the right-hand side of \eqref{eq:FockHodge}.
\end{proof}

\subsection{Hochschild homology with coefficients}
We can deduce our main result from \cref{theorem:main} by collapsing the bigrading to a single grading.

\begin{theorem}
  \label{theorem:main-hochschild}
  Let $X$ be a smooth proper algebraic variety over a field $\field$ of characteristic zero, and let $F\in \derived^\bounded(X)$. For every $n\in \mathbb{N}$, let $F^{\{n\}}\in \derived^\bounded([\Sym^nX])$ and $F^{\langle n \rangle}\in \derived^\bounded(X)$ as in \cref{notation:cyclic-invariants} and  \cref{notation-cyclic} respectively.
  we have
  \begin{equation}
    \label{eqn:HHbox}
    \hochschild_{*}([\Sym^nX], F^{\{ n\}})
    \cong\bigoplus_{\nu\dashv n}\left(\bigotimes_{i\geq 1} \Sym^{\lambda_i}\hochschild_{*}(X, F^{\langle i\rangle})\right)\,,
  \end{equation}
  where $\lambda_i$ is the number of $i$'s in the partition $\nu$. Furthermore, collecting these isomorphisms for varying $n$, we get
  \begin{equation}
    \label{eqn:HHbox-All_n}
    \bigoplus_{n\geq 0}\hochschild_{*}([\Sym^nX], F^{\{ n\}})t^n
    \cong
    \Sym^\bullet\left(\bigoplus_{i\geq 1}\hochschild_{*}(X, F^{\langle i\rangle})t^i\right)\,,
  \end{equation}
  where the symmetric power $\Sym^\bullet$ on the right-hand side is graded with respect to the grading of $\hochschild_*$, but ordinary with respect to the grading given by exponents of the formal variable $t$.
\end{theorem}

\begin{proof}
  Turn the isomorphisms of multigraded vector spaces in \eqref{eqn:HodgeInertiaSym} and \eqref{eq:FockHodge}
  into isomorphisms of (single-)graded (or bigraded if we take into account the powers of $t$ in \eqref{eq:FockHodge}) vector spaces by defining the new grading by $*=\star- \#$. Note that this process of collapsing gradings is compatible with the direct sums, tensor products and symmetric products. By the Hochschild--Kostant--Rosenberg isomorphism for orbifolds \cref{corollary:orbifold-hkr-degree} (or more directly, by \cref{corollary:orbifold-hkr-total}), we get the desired result on the Hochschild homology with coefficients.
\end{proof}

We now can reformulate \cref{theorem:main-hochschild} in terms of the generating function for the dimensions of the graded pieces.

\begin{corollary}
  \label{cor:HHgen}
  For $X$ a smooth proper variety defined over a field $\field$ of characteristic zero, and $F\in \derived^\bounded(X)$, we have the following formula for the generating series of the Hochschild homology of $[\Sym^nX]$:
  \begin{equation}
  \label{eq:HHgen}
  \sum_{n\geq 0}\sum_{i\in \mathbb{Z}}\dim_{\field} \hochschild_{i}([\Sym^nX], F^{\{ n\}})s^it^n
  =
  \prod_{k\geq 1}\prod_{j\in \mathbb{Z}}(1-(-s)^jt^k)^{-(-1)^j\hochschilddim_j(X, F^{\langle k\rangle})}\,,
  \end{equation}
  where $\hochschilddim_j(X,F^{\langle k\rangle})\colonequals\dim_{\field} \hochschild_j(X,F^{\langle k\rangle})$.
\end{corollary}

\begin{proof}
  This follows by applying \cref{lemma:graded-sym-dim} to \eqref{eqn:HHbox-All_n}.
\end{proof}

\paragraph{Line bundle coefficients}
In the case that $F=L$ is a line bundle on the variety $X$,
the $\langle \sigma_i\rangle$-action on $L^{\otimes i}$ permuting the tensor factors is the trivial action, hence
\begin{equation}
  L^{\langle i\rangle}=L^{\otimes i}.
\end{equation}
Therefore, \cref{theorem:main}, \cref{corollary:main} and \cref{theorem:main-hochschild} specialise to the following:

\begin{corollary}
  \label{corollary:line-bundle-coefficients}
  For $L\in \Pic(X)$ and $n\in \IN\setminus\{0\}$, we have isomorphisms of bigraded vector spaces
  \begin{equation}
    \label{equation:hodge-single-n}
    \HH^{\#,\star}(\inertia[\Sym^nX], L^{\{ n\}})
    \cong
    \bigoplus_{\nu\dashv n}\left(\bigotimes_{i\geq 1} \Sym^{\lambda_i}\HH^{\#,\star}(X, L^{\otimes i})\right)\,,
  \end{equation}
  where $\lambda_i$ is the number of $i$'s in the partition $\nu$.
  Collecting these isomorphisms for varying $n$ gives
  \begin{equation}
  \bigoplus_{n\geq 0}\HH^{\#,\star}(\inertia[\Sym^nX], L^{\{ n\}})t^n
  \cong
  \Sym^\bullet\left(\bigoplus_{i\geq 1}\Ho^{\#,\star}(X, L^{\otimes i})t^i\right),
  \end{equation}

  \begin{equation}
    \label{eq:Lgenerating-Sym}
    \bigoplus_{n\geq 0}\hochschild_{*}([\Sym^nX], L^{\{ n\}})t^n
    \cong\
    \Sym^\bullet\left(\bigoplus_{i\geq 1}\hochschild_{*}(X, L^{\otimes i})t^i\right).
  \end{equation}
  The symmetric product on the right-hand sides are graded with respect to the grading of the Hodge groups and the Hochschild homology,
  but ordinary with respect to the grading coming from the exponents of the formal variable $t$.
\end{corollary}

\subsection{Hochschild--Serre cohomology}
\label{subsection:hochschild-serre-description}
We will now specialise the previous results to the setting of Hochschild--Serre cohomology.   Let $X$ be a smooth proper variety of dimension $d_X$.

\begin{lemma}
  \label{lemma:CanonicalOfSym}
  In \cref{notation:cyclic-invariants}, we have:
  \begin{equation}
    \omega_{[\Sym^n X]}[nd_X]\cong (\omega_X[d_X])^{\{ n\}}\,.
  \end{equation}
%   Equivalently, we have an isomorphism of $\symgr_n$-equivariant objects in $\derived^\bounded(X^n)$:
%   \begin{equation}
%     \omega_{X^n}[nd_X]\cong p_1^*(\omega_X[d_X])\otimes \cdots  \otimes p_n^*(\omega_X[d_X]).
%   \end{equation}
\end{lemma}

\begin{proof}
  Note that the canonical bundle of the symmetric quotient stack $[\Sym^n X]$ is given by
  \begin{equation}\label{eq:omegasgn}
    \omega_{[\Sym^n X]}\cong
    \omega_X^{\{ n\}}\otimes \sgn^{\otimes d_X}
    =
    \begin{cases}
      \omega_X^{\{ n\}}             & \text{if $d_X$ is even,} \\
      \omega_X^{\{ n\}}\otimes \sgn & \text{if $d_X$ is odd},
    \end{cases}
  \end{equation}
  where $\sgn$ denotes the alternating representation of $\sym_n$
  (i.e., the one-dimensional representation on which $g\in \sym_n$ acts by multiplication by $\sgn(g)$); see \cite[Lem.\ 5.10]{MR3384503}.
  The lemma follows from the simple fact that for any $F\in \derived^\bounded(X)$ and $d\in \IZ$, we have
  \begin{equation}
    \label{eq:boxparity}
    (F[d_X])^{\{ n\}}\cong F^{\{ n\}}\otimes \sgn^{\otimes d_X} [nd_X]
    =
    \begin{cases}
      F^{\{ n\}}[nd_X]            & \text{for $d_X$ even,}\\
      F^{\{ n\}}\otimes\sgn[nd_X] & \text{for $d_X$ odd.}
    \end{cases}
  \end{equation}
Note that the dimension shift conveniently takes care of the sign, and \eqref{eq:omegasgn} and \eqref{eq:boxparity} together proof the assertion.
%Also
%   note that the usual isomorphism $\omega_{X^n}\cong p_1^*\omega_X\otimes \cdots  \otimes p_n^*\omega_X$ is in general not compatible with the $\symgr_n$-action (differs by a power of alternating representation).
\end{proof}

\begin{lemma}
  \label{lemma:CanonicalCyclic}
  Let $i$ be a positive integer.\\
  If $(k-1)d_X$ is even, then we have an isomorphism of objects in $\derived^\bounded(X)$:
  \begin{equation}
    \Bigl(\omega^{\otimes k-1}_X[(k-1)d_X]\Bigr)^{\langle i\rangle}\cong \omega_X^{\otimes i(k-1)}[i(k-1)d_X].
  \end{equation}
   If $(k-1)d_X$ is odd, then
  \begin{equation}
    \Bigl(\omega^{\otimes k-1}_X[(k-1)d_X]\Bigr)^{\langle i\rangle}
    \cong
    \begin{cases}
      \omega_X^{\otimes i(k-1)}[i(k-1)d_X] & \text{for $i$ odd,} \\
      0                                    & \text{for $i$ even.}
    \end{cases}
  \end{equation}
\end{lemma}

\begin{proof}
  In analogy with \eqref{eq:boxparity}, we also have
  \begin{equation}
  (F[d_X])^{\otimes i}\cong F^{\otimes i}\otimes \sgn^{\otimes d_X} [id_X]
  =
  \begin{cases}
    F^{\otimes i}[id_X]            & \text{for $d_X$ even,}\\
    F^{\otimes i}\otimes\sgn[id_X] & \text{for $d_X$ odd.}
  \end{cases}
  \end{equation}
  The case distinction for $(k-1)d_X$ in the statement comes from the fact that $\sgn(\sigma_i)=(-1)^{i-1}$, where $\sigma_i=(1\, 2\, \dots\, i)$ is the cyclic permutation; see \cref{notation-cyclic}.
  Hence, the $\langle \sigma_i\rangle$-action on $\omega_X^{\otimes i(k-1)}\otimes \sgn$
  is trivial for $i$ odd, but non-trivial for $i$ even .
  This leads to the invariants being the entire $\omega_X^{\otimes i(k-1)}$ or $0$, respectively.
\end{proof}

 Recall \cref{definition:hochschild-serre}:
\begin{equation}
  \hochserre_k^*(\cX)=\hochschild_*(\cX,\omega_{\cX}^{\otimes k-1}[(k-1)d_{\cX}])\,,
\end{equation}
which (for fixed $k$) is a graded vector space,
whereas we write~$\hochserre_k(\cX)$ for the object in~$\derived^\bounded(\field)$.
We can now show \cref{theorem:hochschild-serre}. We restate it with some more details:
\begin{corollary}
  \label{corollary:hochschild-serre}
  Let $X$ be a smooth proper algebraic variety of dimension $d_X$ over a field $\field$ of characteristic zero.
  Let $k$ be a fixed positive integer.
  \begin{enumerate}
    \item[(i)] If $(k-1)d_X$ is even, we have an isomorphism of graded vector spaces:
    \begin{equation}
      \label{equation:hochschild-serre-even}
      \hochserre_{k}([\Sym^nX])
      \cong
      \bigoplus_{\nu\dashv n}\left(\bigotimes_{i\geq 1} \Sym^{\lambda_i}\hochserre_{1+(k-1)i}(X)\right)\,,
    \end{equation}
    where $\lambda_i$ is the number of $i$'s in the partition $\nu$.
    Collecting these isomorphisms together by varying $n$, we have an isomorphism of bigraded vector spaces:
    \begin{equation}
      \label{equation:hochschild-serre-even-all-n}
      \bigoplus_{n\geq 0}\hochserre_{k}([\Sym^nX])t^n
      \cong
      \Sym^\bullet\left(\bigoplus_{i\geq 1}\hochserre_{1+(k-1)i}(X)t^i\right)\,.
    \end{equation}
    Here the $\Sym^\bullet$ is graded with respect to the grading of the Hochschild--Serre cohomology
    and ordinary with respect to the grading given by the exponents of the formal variable $t$.
    \item[(ii)] If $(k-1)d_X$  is odd, we have an isomorphism of graded vector spaces:
    \begin{equation}
      \label{equation:hochschild-serre-odd}
      \hochserre_{k}([\Sym^nX])
      \cong
      \bigoplus_{\substack{\nu\dashv n\\ \text{all $\nu_j$ odd} }}\left(\bigotimes_{\substack{i\geq 1\\ (i \text{ odd})}} \Sym^{\lambda_i}\hochserre_{1+(k-1)i}(X)\right) \\
    \end{equation}
    and
    \begin{equation}
      \label{equation:hochschild-serre-odd-all-n}
      \bigoplus_{n\geq 0}\hochserre_{k}([\Sym^nX])t^n
      \cong
      \Sym^\bullet\left(\bigoplus_{\substack{i\geq 1\\ i \text{ odd}}}\hochserre_{1+(k-1)i}(X)t^i\right)\,.
    \end{equation}
  \end{enumerate}
\end{corollary}

\begin{proof}
  We plug $F=\omega_X^{\otimes k-1}[(k-1)d_X]$ into \cref{theorem:main-hochschild}
  and apply \cref{lemma:CanonicalOfSym,lemma:CanonicalCyclic}.
\end{proof}

\subsection{Hochschild cohomology in low degrees and deformation}
\label{subsection:deformation-theory}
Let us briefly recall why we are interested in Hochschild cohomology.
Originating from the deformation theory of algebras,
introduced by Gerstenhaber,
it is now known that
Hochschild cohomology of abelian (or derived) categories
governs their deformation as a category \cite{MR2238922,MR2183254}.
The second Hochschild cohomology is the deformation space,
and third Hochschild cohomology is the obstruction space.
First Hochschild cohomology has an interpretation as the Lie algebra of autoequivalences \cite{MR2043327}.
There exist geometric definitions of Hochschild cohomology for varieties,
and their agreement with the categorical definition can be found in \cite{MR2183254}.

For a smooth variety $X$ defined over a field of characteristic zero,
by the Hochschild--Kostant--Rosenberg decomposition (see \cref{proposition:hkr-varieties}), we have a decomposition
\begin{equation}
  \hochschild^2(X)
  \cong
  \LaTeXunderbrace{\HH^0\left( X,\bigwedge\nolimits^2\mathrm{T}_X\right) }_{\text{noncommutative}}
  \oplus\LaTeXunderbrace{\HH^1(X,\mathrm{T}_X)}_{\text{geometric}}
  \oplus\LaTeXunderbrace{\HH^2(X,\mathcal{O}_X)}_{\text{gerby}},
\end{equation}
with the deformation-theoretic justification for the interpretation of these terms
given in \cite{MR2477894}.

The noncommutative deformations are also known as Poisson deformations.
A \emph{Poisson structure} on~$X$ is
a global section of~$\bigwedge^2\tangent_X$
for which the Schouten--Nijenhuis self-bracket~$[\sigma,\sigma]$ vanishes.
This obstruction class lives in the global sections of~$\bigwedge^3\tangent_X$.
By Kontsevich's celebrated formality theorem \cite{MR2062626}
this means that it admits a formal (and not just first-order) deformation.

In this paper, for applications in deformation theory we are particularly interested in the way $\hochschild^1([\Sym^nX])$
and~$\hochschild^2([\Sym^nX])$ are determined by the Hochschild--Serre cohomology of $X$.

We specialise \cref{corollary:hochschild-serre} to $k=0$:
\begin{corollary}
  \label{cor:Hochschild-cohomology-Sym}
  Let $X$ be a smooth proper algebraic variety of dimension $d_X$ over a field $\field$ of characteristic zero.
  Let $k$ be a fixed positive integer.
  \begin{enumerate}
    \item[(i)] If $d_X$ is even,
%    we have an isomorphism of graded vector spaces:
%    \begin{equation}
%      \label{equation:hochschild-cohomology-even}
%      \hochschild^*([\Sym^nX])
%      \cong
%      \bigoplus_{\nu\dashv n}\left(\bigotimes_{i\geq 1} \Sym^{\lambda_i}\hochserre^*_{1-i}(X)\right).
%    \end{equation}
%    Varying $n$,
  %  we have an isomorphism of bigraded vector spaces:
    \begin{equation}
      \label{equation:hochschild-cohomology-even-all-n}
      \bigoplus_{n\geq 0}\hochschild^*([\Sym^nX])t^n
      \cong
      \Sym^\bullet\left(\bigoplus_{i\geq 1}\hochserre_{1-i}(X)t^i\right)\,.
    \end{equation}
    %Here the $\Sym^\bullet$ is graded with respect to the grading of the Hochschild cohomology  and ordinary with respect to the grading given by the exponents of the formal variable $t$.
    \item[(ii)] If $d_X$  is odd,
%    we have an isomorphism of graded vector spaces:
%    \begin{equation}
%      \label{equation:hochschild-cohomology-odd}
%      \hochschild^*([\Sym^nX])
%      \cong
%      \bigoplus_{\substack{\nu\dashv n\\ \text{all $\nu_j$ odd} }}\left(\bigotimes_{\substack{i\geq 1\\ (i \text{ odd})}} \Sym^{\lambda_i}\hochserre_{1-i}(X)\right) \\
%    \end{equation}
%    and
    \begin{equation}
      \label{equation:hochschild-cohomology-odd-all-n}
      \bigoplus_{n\geq 0}\hochschild^*([\Sym^nX])t^n
      \cong
      \Sym^\bullet\left(\bigoplus_{\substack{i\geq 1\\ i \text{ odd}}}\hochserre_{1-i}(X)t^i\right)\,.
    \end{equation}
  \end{enumerate}
\end{corollary}

Let us give in the following corollary the explicit results for $\hochschild^1$ and $\hochschild^2$. The interested reader can write down the (more lengthy) description for~$\hochschild^3$, which we will not need it in what follows.
One should really understand the Gerstenhaber bracket too, which involves $\hochschild^3$,
so that the obstruction can be computed,
but this is an open question.
\begin{corollary}
  \label{corollary:first-and-second}
  Let $X$ be a geometrically connected smooth proper variety of dimension at least~1 defined over a field $\field$ of characteristic zero.
  For all~$n\geq 2$ we have that
  \begin{equation}
    \label{equation:first-hochschild}
    \hochschild^1([\Sym^nX])\cong\hochschild^1(X)
  \end{equation}
  and
  \begin{equation}
    \label{equation:second-hochschild}
    \hochschild^2([\Sym^nX])\cong
    \begin{cases}
      \hochschild^2(X)\oplus\bigwedge^2\hochschild^1(X) & d_X\geq 3 \\
      \hochschild^2(X)\oplus\bigwedge^2\hochschild^1(X)\oplus\hochserre_{-1}^2(X) & d_X=2 \\
      \hochschild^2(X)\oplus\bigwedge^2\hochschild^1(X)\oplus\hochserre_{-2}^2(X) & d_X=1
    \end{cases}
  \end{equation}
  except when~$d_X=1$ and~$n=2$, in which case
  \begin{equation}
    \label{equation:second-hochschild-n-2}
    \hochschild^2([\Sym^2X])\cong\hochschild^2(X)\oplus\bigwedge^2\hochschild^1(X).
  \end{equation}
\end{corollary}

\begin{proof}
  We apply \cref{cor:Hochschild-cohomology-Sym}.
  Note that under our assumption on $X$, we have $\hochschild^0(X)=\field$.

  Observe that~$\hochserre_{-k}(X)$ is concentrated in degrees~$\geq k d_X$,
  which explains the absence of those contributions in \eqref{equation:first-hochschild}
  resp.~\eqref{equation:second-hochschild}
  if~$d_X\geq 2$ resp.~$d_X\geq 3$,
  whereas~$\hochschild^1(X)$ is always present.
  To conclude the computation for \eqref{equation:first-hochschild} if~$d_X=1$
  it suffices to observe that we exclude even~$i$,
  therefore~$\hochserre_{-1}(X)$ does not contribute.

  For similar degree reasons we observe that~$\hochschild^2(X)$
  is always present in \eqref{equation:second-hochschild}.
  Now we wish to understand the other contributions to~$\hochschild^2([\Sym^nX])t^n$.
  The first one arises as a summand of~$\Sym^n$ in \eqref{equation:hochschild-serre-even-all-n}
  and \eqref{equation:hochschild-serre-odd-all-n},
  namely
  \begin{equation}
    \Sym^{n-2}(\hochschild^0(X)t)\otimes\Sym^2(\hochschild^1(X)t)\cong\bigwedge^2\hochschild^1(X)t^{n+2},
  \end{equation}
  so it is present for all~$d_X\geq 1$ and~$n\geq 2$.
  The summand
  \begin{equation}
    \Sym^{n-2}(\hochschild^0(X)t)\otimes\hochserre_{-1}^2(X)t^2\cong\hochserre_{-1}^2(X)t^n
  \end{equation}
  of~$\Sym^{n-1}$ in \eqref{equation:hochschild-serre-even-all-n} contributes for~$d_X=2$ only,
  as~$\hochserre_{-1}^2(X)=0$ for~$d_X\geq 3$ by \cref{proposition:hkr-varieties}
  or it is excluded in \eqref{equation:hochschild-serre-odd-all-n}.

  Finally, if~$d_X=1$ and~$n\geq 3$ then for~$i=2$ we have a contribution by~$\hochserre_{-2}^2(X)$ in degree~2,
  so~$\Sym^{n-2}$ in \eqref{equation:hochschild-serre-odd-all-n} will have the summand
  % trick from https://tex.stackexchange.com/a/32095/210
  \belowdisplayskip=-12pt
  \begin{equation}
    \Sym^{n-3}(\hochschild^0(X)t)\otimes\hochserre_{-2}^2(X)t^3\cong\hochserre_{-2}^2(X)t^n.
  \end{equation}
  \qedhere
\end{proof}

\subsection{Consequences for Hilbert schemes of points on surfaces}
\label{subsection:hilbert-consequences}
In the case that $X$ is a smooth projective surface we will write~$X=S$,
and we have the \emph{derived McKay correspondence}
\begin{equation}
  \label{eqn:BridgelandKingReidHaiman}
   \Psi\colon \derived^\bounded([\Sym^nS])\xrightarrow\cong \derived^\bounded(\hilbn{n}{S})
\end{equation}
of \cite{MR1824990,MR1839919} for every $n\in \IN$,
identifying the derived categories of the symmetric quotient stacks
and that of the Hilbert schemes of points. Concretely, $\Psi$ is the Fourier--Mukai transform along \begin{equation}\mathcal O_{\mathcal Z}\in \derived^\bounded_{\mathfrak S_n}(S^n\times \hilbn{n}{S})\cong \derived^\bounded([\Sym^n X]\times \hilbn{n}{S})
\end{equation}
where $\cZ\subset S^n\times \hilbn{n}{S}$ is the universal family of $\sym_n$-clusters.
Using this, we can deduce formulas for Hochschild homology with values in natural line bundles
and the Hochschild--Serre cohomology of Hilbert schemes of points.

Given a line bundle $L\in \Pic S$,
the equivariant line bundle $L^{\{n\}}$ on $S^n$ descends
to a line bundle $L^{(n)}$ on the symmetric quotient \emph{variety} $S^{(n)}=S^n/\sym_n$.
Concretely, if $\pi\colon S^n\to S^{(n)}$ denotes the quotient morphism,
we have $L^{(n)}=\pi_*(L^{\{n\}})^{\sym_n}$.
Pulling back by the Hilbert--Chow morphism $\mu\colon \hilbn{n}{S} \to S^{(n)}$
gives the \emph{natural line bundle on $\hilbn{n}{S} $ induced by $L$} :
\begin{equation}
  L_n\colonequals\mu^*L^{(n)}\in \Pic(\hilbn{n}{S}).
\end{equation}
We first prove a technical result,
of a similar nature to \cref{corollary:hochschild-serre-comparison-hilbert-scheme},
but we explicitly use Fourier--Mukai transforms,
unlike the quoted result.
We will prove it using the formalism described in \cite[\S2.1]{MR1998775}.
Recall that a Fourier--Mukai equivalence~$\Phi_\cP\colon\derived^\bounded(\cX)\to\derived^\bounded(\cY)$ given by kernel $\cP\in \derived^\bounded(\cX\times\cY)$
induces an equivalence~$\Ad_\cP\colon\derived^\bounded(\cX\times\cX)\overset{\simeq}{\to}\derived^\bounded(\cY\times\cY)$ given by the kernel $\cP\boxtimes \cP\in \derived^\bounded(\cX\times \cY\times \cX\times \cY)\cong\derived^\bounded(\cX\times \cX\times \cY\times \cY)$
We denote by~$\circ$ the convolution product of Fourier--Mukai kernels.
The kernel for the inverse of~$\Phi_\cP$
will be denoted~$\cP^{\mathrm{T}}$,
so that~$\cP^{\mathrm{T}}\circ\cP\cong\Delta_*(\mathcal{O}_\cY)$.

\begin{proposition}
  \label{proposition:hochschild-with-values-invariant}
  Let~$\cX,\cY$ be smooth and proper orbifolds,
  such that~$\cP\in\derived^\bounded(\cX\times\cY)$ induces an equivalence~$\Phi_\cP$.
  Let~$\cE\in\derived^\bounded(\cX)$
  and consider a morphism
  \begin{equation}
    \label{equation:intertwining}
    \Delta_*\cE\circ\mathcal{P}
    \to
    \cP\circ\Delta_*(\Phi_\cP(\cE))
  \end{equation}
  in~$\derived^\bounded(\cX\times\cY)$.
  Then~$\Ad_\cP$ induces a morphism
  \begin{equation}
    \label{equation:hochschild-values-morphism}
    \hochschild_*(\cX,\cE)\to\hochschild_*(\cY,\Phi_\cP(\cE)),
  \end{equation}
  which is an isomorphism if \eqref{equation:intertwining} is an isomorphism.
\end{proposition}

\begin{proof}
  We have the sequence of morphisms
  \begin{equation}
    \begin{aligned}
      \hochschild_*(\cX,\cE)
      &=\Hom_{\derived^\bounded(\cX\times\cX)}^*(\Delta_*\mathcal{O}_\cX,\Delta_*(\omega_\cX[d_\cX]\otimes\cE)) \\
      &=\Hom_{\derived^\bounded(\cX\times\cX)}^*(\Delta_*\mathcal{O}_\cX,\Delta_*(\omega_\cX[d_\cX])\circ\Delta_*(\cE))) \\
      &\cong\Hom^*_{\derived^\bounded(\cY\times\cY)}(\Delta_*\mathcal{O}_\cY,\cP^{\mathrm{T}}\circ\Delta_*(\omega_\cX[d_\cX])\circ\Delta_*(\cE)\circ\cP) \\
      &\to\Hom^*_{\derived^\bounded(\cY\times\cY)}(\Delta_*\mathcal{O}_\cY,\cP^{\mathrm{T}}\circ\Delta_*(\omega_\cX[d_\cX])\circ\cP\circ\Delta_*(\Phi_\cP(\cE))) \\
      &\cong\Hom^*_{\derived^\bounded(\cY\times\cY)}(\Delta_*\mathcal{O}_\cY,\cP^{\mathrm{T}}\circ\cP\circ\Delta_*(\omega_\cY[d_\cY])\circ\Delta_*(\Phi_\cP(\cE))) \\
      &=\Hom^*_{\derived^\bounded(\cY\times\cY)}(\Delta_*\mathcal{O}_\cY,\Delta_*(\omega_\cY[d_\cY])\circ\Delta_*(\Phi_\cP(\cE))) \\
      &=\hochschild_*(\cY,\Phi_\cP(\cE)),
    \end{aligned}
  \end{equation}
  where on the third line we used (in the first argument) that~$\Ad_\cP$ preserves the identity functor
  and on the fifth line we used (in the second argument) that the Serre functor commutes with equivalences.
  The morphism on the fourth line is an isomorphism if \eqref{equation:intertwining} is one.
\end{proof}

\begin{remark}
  The condition \eqref{equation:intertwining} implies the following isomorphism of functors:
  \begin{equation}
    \label{equation:intertwining-rephrased}
    \Phi_\cP(-\otimes\cE)\cong\Phi_\cP(-)\otimes\Phi_\cP(\cE).
  \end{equation}
\end{remark}

\begin{corollary}
  \label{cor:DerivedInvariance}
  Let~$\cX,\cY$ be smooth and proper orbifolds,
  such that~$\cP\in\derived^\bounded(\cX\times\cY)$ induces an equivalence~$\Phi_\cP$. Then $\Phi_\cP$ induces also an isomorphism of bigraded vector spaces\footnote{In fact $\Phi_{\cP}$ induces an isomorphism of bigraded \textit{algebras}, see \cref{theorem:hochschild-serre-morita-invariant} and  \cref{theorem:agreement}.} :
  \begin{equation}
    \hochserre_\bullet^*(\cX)\cong \hochserre_\bullet^*(\cY).
  \end{equation}
\end{corollary}
\begin{proof}
  By the uniqueness of the Serre functor (hence it commutes with equivalences),  the condition \eqref{equation:intertwining-rephrased} is satisfied with $\cE=\omega^k_\cX[kd_\cX]$ for any $k\in \mathbb{Z}$.
\end{proof}

Without the intertwining compatibility \eqref{equation:intertwining}
between the derived equivalence and the objects of coefficients, the following example that we do not necessarily get an isomorphism of the Hochschild homologies with coefficients.
\begin{example}
  Let~$E$ be an elliptic curve,
  and consider the Fourier--Mukai equivalence given by the Poincar\'e bundle~$\mathcal{P}$ on~$E\times E$.
  If~$e\in E$ is the origin,
  then~$\Phi_{\mathcal{P}}(\mathcal{O}_e)=\mathcal{O}_E$,
  i.e., the skyscraper at the origin is sent to the structure sheaf.
  Then one can compute that
  \begin{equation}
    \hochschild_*(E,\mathcal{O}_e)\cong\field[0]\oplus\field[1]
  \end{equation}
  whereas
  \begin{equation}
    \hochschild_*(E,\mathcal{O}_E)\cong\hochschild_*(E)\cong\field[-1]\oplus\field^{\oplus2}[0]\oplus\field[1].
  \end{equation}
\end{example}

\begin{proposition}
  \label{proposition:hochschild-with-values-line-bundles-invariant}
  Let~$S$ be a smooth projective surface.
  Let~$L$ be a line bundle on~$S$.
  The derived invariance and agreement from \cref{proposition:hochschild-with-values-invariant}
  induce isomorphisms
  \begin{equation}
    \hochschild_*(\hilbn{n}{S},L_n)\cong\hochschild_*([\Sym^nS],L^{\{n\}}).
  \end{equation}
\end{proposition}
This result does not hold more generally with $L$ replaced by a higher-rank vector bundle $F$. The main reason is that the $\symgr_n$-equivariant vector bundle $F^{\{n\}}$ on $S^n$ is in general not the pull-back of a vector bundle on $S^{(n)}$ and, if we set $F_n\colonequals\Psi(F^{\{n\}})$, the condition \eqref{equation:intertwining-rephrased} does not hold in this situation.
\begin{proof}
  By  \cref{proposition:hochschild-with-values-invariant}, it suffices to check that the Bridgeland--King--Reid--Haiman equivalence \cite{MR1824990, MR1839919} satisfies
  the intertwining property that there is an isomorphism \eqref{equation:intertwining}.
  Recall that $\Psi$ is the Fourier--Mukai transform along $\mathcal O_\cZ$, and that we have a commutative diagram
  \begin{equation}
    \label{equation:derived-mckay-diagram}
    \begin{tikzcd}
      \mathcal{Z} \arrow[r, "p"] \arrow[d, "q"] & S^n \arrow[d, "\pi"] \\
      \hilbn{n}{S} \arrow[r, "\mu"] & S^{(n)}=S^n/\sym_n
    \end{tikzcd}
  \end{equation}
  where $p$ and $q$ are the restrictions of the projections of the product $S^n\times \hilbn{n}{S}$ to its factors.
  Denoting the embedding of the universal family by $i\colon \mathcal Z\hookrightarrow S^n\times \hilbn{n}{S}$,
  a standard computation\footnote{
    For example, one can note that \cite[Proposition~2.1.6]{MR1998775} still works for orbifolds in place of varieties
    and for convolution products of kernels in place of compositions of functors.
    Then \eqref{eq:FMLieberman} are just two instances of this general statement.
  } for Fourier--Mukai kernels shows that
  \begin{equation}
    \label{eq:FMLieberman}
    \Delta_*L^{\{n\}} \circ \mathcal O_\cZ\cong i_*p^*L^{\{n\}}\quad,\quad   \mathcal O_\cZ\circ \Delta_*L_n\cong i_*q^*L_n\,.
  \end{equation}
  As $L^{\{n\}}\cong \pi^*L^{(n)}$ and $L_n=\mu^*L^{(n)}$,
  commutativity of \eqref{equation:derived-mckay-diagram} gives an isomorphsim $p^*L^{\{n\}}\cong q^*L_n$.
  Hence, we have an isomorphism between the two objects in \eqref{eq:FMLieberman},
  which is exactly what we need to conclude by \cref{proposition:hochschild-with-values-invariant}.
\end{proof}

We can now state the two main corollaries of this subsection.
\begin{corollary}
  \label{corollary:HH-Hilb-LineBundle}
 Let $S$ be a smooth projective surface defined over a field $\field$ of characteristic zero.
   For $L\in \Pic(S)$ and $n\in \mathbb{N}\setminus\{0\}$, we have
  \begin{equation}
    \hochschild_{*}(\hilbn{n}{S}, L_n)
    \cong\bigoplus_{\nu\dashv n}\left(\bigotimes_{i\geq 1} \Sym^{\lambda_i}\hochschild_{*}(S, L^{\otimes i})\right)\,,
  \end{equation}
  where $\lambda_i$ is the number of $i$'s in the partition $\nu$.
  Collecting these isomorphisms for varying $n$, we get
  \begin{equation}
    \label{eq:Lgenerating}
    \bigoplus_{n\geq 0}\hochschild_{*}(\hilbn{n}{S} , L_n)t^n
    \cong\Sym^\bullet\left(\bigoplus_{i\geq 1}\hochschild_{*}(S, L^{\otimes i})t^i\right)\,.
  \end{equation}
  The symmetric product on the right-hand side of \eqref{eq:Lgenerating}
  is graded with respect to the grading of the Hochschild homology,
  but ordinary with respect to the grading coming from the exponents of $t$.
  In terms of generating functions for the dimensions of the graded pieces, this means
  \begin{equation}
    \label{equation:hilbert-hochschid-serre-line-bundle-series}
    \sum_{n\geq 0}\sum_{i=-2n}^{2n}\dim_\field\hochschild_{i}(\hilbn{n}{S}, L_n)s^it^n
    =
    \prod_{k\geq 1}\prod_{j=-2}^{2}(1-(-s)^jt^k)^{-(-1)^j\hochschilddim_j(S, L^{\otimes k})}
  \end{equation}
\end{corollary}

\begin{proof}
  This follows from \cref{corollary:line-bundle-coefficients}, \eqref{eqn:BridgelandKingReidHaiman}, and \cref{proposition:hochschild-with-values-line-bundles-invariant}. For \eqref{equation:hilbert-hochschid-serre-line-bundle-series}, we apply \cref{lemma:graded-sym-dim} to \eqref{eq:Lgenerating}.
\end{proof}

%\begin{corollary}
%  \label{corollary:orbifold-hochschild-serre-line-bundle-series}
%  For $X$ a smooth proper variety and $L\in \Pic(X)$, we have
%  \begin{equation}
%  \label{equation:orbifold-hochschild-serre-line-bundle-series}
%    \sum_{n\geq 0}\sum_{i=-nd_X}^{nd_X}\dim_\field\hochschild_{i}([\Sym^nX], L^{\{ n\}})s^it^n
%    =
%    \prod_{k\geq 1}\prod_{j=-d_X}^{d_X}(1-(-s)^jt^k)^{-(-1)^j\hochschilddim_j(X, L^{\otimes k})}
%  \end{equation}
%  where $\hochschilddim_j(X,L^{\otimes k})\colonequals\dim_\field\hochschild_j(X,L^{\otimes k})$.
%\end{corollary}

We also get a formula for the Hochschild--Serre cohomology of the Hilbert schemes of points, which was our original goal.

\begin{corollary}
  \label{corollary:hochschild-serre-hilbert-scheme}
  For every $k\in \mathbb{Z}$ and $n\in \mathbb{N}$, we have
  \begin{equation}
    \label{equation:hilbert-hochschild-serre-even}
    \hochserre_{k}(\hilbn{n}{S} )
    \cong
    \bigoplus_{\nu\dashv n}\left(\bigotimes_{i\geq 1} \Sym^{\lambda_i}\hochserre_{1+(k-1)i}(S)\right)\,.
  \end{equation}
  Collecting the formula for varying $n$, but fixed $k$, we get
  \begin{equation}
    \label{equation:hilbert-hochschild-serre-even-all-n}
    \bigoplus_{n\geq 0}\hochserre_{k}(\hilbn{n}{S} )t^n
    \cong
    \Sym^\bullet\left(\bigoplus_{i\geq 1}\hochserre_{1+(k-1)i}(S)t^i\right)\,.
  \end{equation}
\end{corollary}

\begin{proof}
  This follows from  \cref{corollary:hochschild-serre}, \eqref{eqn:BridgelandKingReidHaiman}, and \cref{cor:DerivedInvariance}.
\end{proof}

\subsection{On the Fock space structure}
\label{subsection:fock-space}
In this section we work over~$\field=\mathbb{C}$.
Let $V$ be a vector space together with a bilinear form $\langle-,-\rangle$.
One associates the \emph{Heisenberg algebra} $H_V$ to the pair~$(V,\langle-,-\rangle)$.
It is generated by elements $p_\alpha^{(n)}$ and $q_{\alpha}^{(n)}$ for $\alpha\in V$ and $n\in\mathbb{N}$.
The $p$-generators commute among each other,
as do the $q$-generators.
In addition, there are more complicated relations between mixed two-letter words in the generators;
see, e.g., \cite[\S2.2]{2105.13334v3} for details.

Note that we do not require $\langle-,-\rangle$ to be symmetric,
and the Mukai pairing on Hochschild homology is indeed not symmetric in general.
This means that we have a well-defined Heisenberg algebra $H_V$,
but not necessarily the Heisenberg Lie algebra $\mathfrak h_V$; see again \cite[\S2.2]{2105.13334v3} for details.

The algebra $H_V$ has an irreducible representation called the \emph{Fock space}.
It is given by the quotient $H_V/I$ where $I$ is the left ideal generated by the $q$-generators.
The Fock space can be identified with the symmetric power of an infinite direct sum of copies of $V$ via the isomorphism of vector spaces
\begin{equation}
  \label{eq:Fockiso}
  F_V=H_V/I\xrightarrow{\simeq}\Sym^\bullet\left(\bigoplus_{i\ge 1} V t^i\right):\overline{ p_{\alpha_1}^{(\nu_1)} p_{\alpha_2}^{(\nu_2)}\cdots p_{\alpha_\ell}^{(\nu_\ell)}}\mapsto (\alpha_1t^{\nu_1})\cdot (\alpha_2t^{\nu_2})\cdots (\alpha_\ell t^{\nu_\ell})
\end{equation}

The main examples for us are:
\begin{itemize}
  \item $V=\Ho^*(S,\IC)$  the singular cohomology of a smooth projective surface together with the cup product pairing.
    By the work of G\"ottsche \cite{MR1032930}, Nakajima \cite{MR1441880}, and Grojnowski \cite{MR1386846},
  the cohomology~$\bigoplus_{n\geq 0}\HH^{*}(\hilbn{n}{S}, \IC)t^n$
  can be identified with the Fock space $F_{\HH^*(S, \IC)}$.
  \item $V=\hochschild_*(X)$ the Hochschild homology of a smooth projective variety $X$ equipped with the Mukai pairing. By \cref{corollary:Hochschild-homology-Fock},
   $ \bigoplus_{n\geq 0}\hochschild_{*}([\Sym^nX])t^n$ can be identified with the Fock space $F_{\hochschild_{*}(X)}$.
\end{itemize}
These two examples are related:
by the Hochschild--Kostant--Rosenberg isomorphism,
we have the isomorphism of ungraded vector spaces~$\hochschild_{*}(\hilbn{n}{S})\cong \HH^*(\hilbn{n}{S})$.
Hence, after summing over $n$,
we obtain the isomorphism
\begin{equation}
  F_{\HH^*(S, \IC)}\cong F_{\hochschild_*(S)}.
\end{equation}
Finally, the derived McKay correspondence of Bridgeland--King--Reid \cite{MR1824990} and Haiman \cite{MR1839919}
gives equivalences $\derived^\bounded(\hilbn{n}{S})\cong\derived^\bounded([\Sym^nS])$
which induce isomorphisms $\hochschild_{*}(\hilbn{n}{S})\cong\hochschild_{*}([\Sym^nS])$ for every $n$,
as discussed in \cref{subsection:hilbert-consequences}.

We formulate the following generalization of \cref{corollary:Hochschild-homology-Fock} as a conjecture.
In \cref{appendix:hochschild-serre-dg} we recall the definition of Hochschild homology in this generality
(and generalize it to Hochschild--Serre cohomology for smooth and proper dg categories,
which we will use shortly).
When~$\mathcal{T}$ is a dg category,
then~$\Sym^n\mathcal{T}$ denotes its~$n$th symmetric power in the sense of Ganter--Kapranov \cite{MR3177367}.
If~$\mathcal{T}$ is smooth and proper, then so is~$\Sym^n\mathcal{T}$.
\begin{conjecture}
  \label{conjecture:HH-Fock-dg}
  Let $\mathcal{T}$ be a smooth proper dg category.
  Then we have
  \begin{equation}
    \label{eq:Hhomsym-dg}
    \bigoplus_{n\geq 0}\hochschild_{*}(\Sym^n\mathcal T)t^n
    \cong
    \Sym^\bullet\left(\bigoplus_{i\geq 1}\hochschild_{*}(\mathcal T)t^i\right)\,.
  \end{equation}
\end{conjecture}

\begin{proposition}
  \label{prop:HH-Fock-dg}
  Let $\mathcal{T}$ be a smooth proper dg category.
  Let $\mathcal{T}=\langle \mathcal{A}_1, \cdots, \mathcal{A}_m\rangle$ be a semiorthogonal decomposition.
  If \cref{conjecture:HH-Fock-dg} holds for all categories except one among $\mathcal{T}, \mathcal{A}_1, \cdots, \mathcal{A}_m$, then it also holds for the remaining one.
\end{proposition}
\begin{proof}
  By induction, we only need to deal with the case $m=2$. Assume $\mathcal{T}=\langle \mathcal{A}, \mathcal{B}\rangle$.
  By \cite[Theorem 1.1]{koseki2023symmetric}, we have a semiorthogonal decomposition for any $n\geq 0$,
  \begin{equation}
    \Sym^n\mathcal{T}=\langle \Sym^n\mathcal{A}, \Sym^{n-1}\mathcal{A}\otimes\mathcal{B}, \cdots, \Sym^n\mathcal{B}\rangle.
  \end{equation}
  By additivity \cite[\S1.5, \S1.12]{MR1667558} and the K\"unneth formula \cite[Theorem~2.8]{MR3020737} for Hochschild homology,
  we get
  \begin{equation}
    \hochschild_{*}(\Sym^n\mathcal{T})\cong \bigoplus_{i=0}^n\hochschild_{*}(\Sym^{i}\mathcal{A})\otimes  \hochschild_{*}(\Sym^{n-i}\mathcal{B}).
  \end{equation}
  Taking the sum over $n$, we get
  \begin{equation}
    \label{eqn:HH-Sym-dg-Additivity}
    \bigoplus_{n\geq 0} \hochschild_{*}(\Sym^n\mathcal{T})t^n\cong   \left(\bigoplus_{n\geq 0} \hochschild_{*}(\Sym^n\mathcal{A})t^n\right) \otimes \left(\bigoplus_{n\geq 0} \hochschild_{*}(\Sym^n\mathcal{B})t^n\right).
  \end{equation}
  On the other hand, again by the additivity of $\hochschild_{*}$, we have
  \begin{equation}
    \label{eqn:Foch-Additivity}
    F_{\hochschild_{*}(\mathcal{T})}= \Sym^\bullet\left(\bigoplus_{i\geq 1}\hochschild_{*}(\mathcal{T})t^i\right)\cong  \Sym^\bullet\left(\bigoplus_{i\geq 1}(\hochschild_{*}(\mathcal{A})t^i\oplus \hochschild_{*}(\mathcal{B})t^i)\right)=  F_{\hochschild_{*}(\mathcal{A})}\otimes F_{\hochschild_{*}(\mathcal{B})}.
  \end{equation}
  By comparing \eqref{eqn:HH-Sym-dg-Additivity} and \eqref{eqn:Foch-Additivity},
  it is clear that if $\mathcal{A}$ and $\mathcal{B}$ satisfy \cref{conjecture:HH-Fock-dg} then so does $\mathcal{T}$.
  Supposing $\mathcal{A}$ and $\mathcal{T}$ satisfy \cref{conjecture:HH-Fock-dg},
  then since the generating series $\sum_{n\geq 0} \dim \hochschild_{*}(\Sym^n\mathcal{A})t^n$
  is an invertible power series (because the constant coefficient is~1),
  $\bigoplus_{n\geq 0} \hochschild_{*}(\Sym^n\mathcal{B})t^n$ and $F_{\hochschild_{*}(\mathcal{B})}$
  have the same generating series,
  hence we have an isomorphism of graded vector spaces as in \cref{conjecture:HH-Fock-dg} for $\mathcal{B}$.
\end{proof}

The following consequence asserts that \cref{conjecture:HH-Fock-dg} holds for
many dg categories arising as the so-called Kuznetsov component of a variety.
\begin{corollary}
  \label{corollary:Kuznetsov-Component}
  Let $X$ be a smooth proper variety defined over a field $\field$ of characteristic zero.
  Suppose there is an exceptional collection $E_1, \ldots, E_m$ in $\derived^\bounded(X)$.
  Denote by $\mathcal{A}_X$ the (left or right) orthogonal of $\langle E_1, \ldots, E_m\rangle$.
  Then \cref{conjecture:HH-Fock-dg} holds for $\mathcal{A}_X$.
\end{corollary}

\begin{proof}
  We can apply \cref{prop:HH-Fock-dg} since \cref{conjecture:HH-Fock-dg} holds for $\derived^\bounded(X)$ by \cref{corollary:Hochschild-homology-Fock} and for each $\langle E_i \rangle\cong \derived^\bounded(\Spec \field)$.
\end{proof}

\begin{corollary}
  \label{cor:phantom}
  Let $\mathcal{T}$ be a smooth proper dg category.
  Let $\mathcal{T}=\langle \mathcal{A}_1, \cdots, \mathcal{A}_m,\mathcal B\rangle$ be a semiorthogonal decomposition
  such that $\cB$ is a quasi-phantom category, i.e., $\hochschild_*(\mathcal B)=0$.
  If \cref{conjecture:HH-Fock-dg} holds for all the categories $\mathcal{T}, \mathcal{A}_1, \ldots, \mathcal{A}_m$,
  then $\Sym^n\mathcal B$ is again a quasi-phantom category for all $n$.
\end{corollary}

\begin{proof}
  By \cref{prop:HH-Fock-dg}, we know that \cref{conjecture:HH-Fock-dg} holds for $\cB$.
  The assertion now just follows from the simple fact the the symmetric power of the zero vector space is the zero vector space.
\end{proof}

If $\mathcal{T}=\derived^\bounded(S)$ for a smooth projective surface $S$,
and $\mathcal A_i=\langle E_i\rangle$ for some exceptional object $E_i$ for every $i=1,\dots, m$,
\cref{cor:phantom} specialises to \cite[Lemma 4.4]{koseki2023symmetric}.
We were inspired to include \cref{cor:phantom} in the second version of this article by Koseki
(the author of op.~cit.) asking in a talk about possible generalisations of his result.

If \cref{conjecture:HH-Fock-dg} holds in full generality,
this would imply that every symmetric power of every quasi-phantom category is again a quasi-phantom category.

%Now note that, as $\hochserre_1=\hochschild_*$,
%the $k=1$ case of \eqref{equation:hochschild-serre-even-all-n} gives
%\begin{equation}
%  \label{eq:Hhomsym}
%  \bigoplus_{n\geq 0}\hochschild_{*}([\Sym^nX])t^n
%  \cong
%  \Sym^\bullet\left(\bigoplus_{i\geq 1}\hochschild_{*}(X)t^i\right).
%\end{equation}
%Hence, the joint Hochschild homology $\bigoplus_{n\geq 0}\hochschild_{*}([\Sym^nX])t^n$
%can be identified with the Fock space representation $F_{\hochschild_*(X)}$ of the Heisenberg algebra associated with $\hochschild_*(X)$.

\paragraph{Is Hochschild--Serre cohomology of Hilbert schemes a Fock space?}
For $k\neq 1$, \cref{corollary:hochschild-serre} still gives
an identification of $\bigoplus_{n\geq 0}\hochserre_{k}([\Sym^nX])t^n$
with a certain total symmetric power, namely, when $d_X$ is even,
\begin{equation}
  \bigoplus_{n\geq 0}\hochserre_{k}([\Sym^nX])t^n
  \cong
  \Sym^\bullet\left(\bigoplus_{i\geq 1}\hochserre_{1+(k-1)i}(X)t^i\right).
\end{equation}
However, the symmetric power is not taken of an infinite direct sum of copies of one vector space,
but of an infinite direct sum of \emph{different} vector spaces.
In other words, we still have a basis of $\bigoplus_{n\geq 0}\hochserre_{k}([\Sym^nX])t^n$
consisting of elements of the form $(\alpha_1t^{\nu_1})\cdot (\alpha_2t^{\nu_2})\cdots (\alpha_\ell t^{\nu_\ell})$
as in \eqref{eq:Fockiso},
but now the $\alpha_i$ are elements of \emph{different} vector spaces,
depending on the exponent $\nu_i$ of $t$.
Hence, it seems to us like there is no identification of $\bigoplus_{n\geq 0}\hochserre_{k}([\Sym^nX])t^n$
with a Fock space for $k\neq 1$.

One could hope that the full Hochschild--Serre cohomology
\begin{equation}
  \label{eq:fullHS}
  \bigoplus_{k\in \IZ}\bigoplus_{n\geq 0}\hochserre_{k}([\Sym^nX])t^n=\bigoplus_{n\geq 0}\hochserre_\bullet([\Sym^nX])t^n
\end{equation}
can be identified with the Fock space associated to $V=\bigoplus_{k\in \IZ}\hochserre_k(X)=\hochserre(X)$,
but this does not work either.
Indeed, \eqref{eq:fullHS} is generated by words of the form
\begin{equation}
  \label{eq:wordinHS}
  (\alpha_1t^{\nu_1})\cdot (\alpha_2t^{\nu_2})\cdots (\alpha_\ell t^{\nu_\ell}),
\end{equation}
but we cannot take arbitrary $\alpha_i\in V$ and integers $\nu_i$.
Instead, there must be some $k$ such that $\alpha_i\in \hochserre_{1+(k-1)\nu_i}(X)$
for all $i=1,\dots,\ell$ if \eqref{eq:wordinHS} is an element of \eqref{eq:fullHS}.
Hence,
\begin{equation}
  \bigoplus_{n\geq 0}\hochserre_\bullet([\Sym^nX])t^n\subsetneq \Sym^\bullet\left(\bigoplus_{i\ge 1}\hochserre_\bullet(X)t^i\right)
\end{equation}
is a proper subspace of the symmetric power.

\paragraph{Geometric and categorical Heisenberg action}
For $X$ a smooth projective variety of arbitrary dimension,
by \cite{MR3774400}, there is a categorical Heisenberg action of $\derived^\bounded(X)$ on $\bigoplus_{n\geq 0}\derived^\bounded([\Sym^nX])$.
It seems worthwile to study how it relates to the Heisenberg action of $\hochschild_{*}(X)$ on $\bigoplus_{n\geq 0}\hochschild_{*}([\Sym^nX])$ given by \cref{corollary:Hochschild-homology-Fock}.

Similarly, in the case $X=S$ is a surface,
one could study how Nakajima's action of $\HH^*(S)$ on $\bigoplus_{n\geq 0}\HH^*(\hilbn{n}{S})$ is
related to these actions, under the identification given by Hochschild--Kostant--Rosenberg isomorphism and the McKay correspondence.

The categorical action of \cite{MR3774400} was generalised in \cite{2105.13334v3}
from derived categories of smooth projective varieties to dg-enhanced triangulated categories. So the same question can be asked whenever we have an instance of \cref{conjecture:HH-Fock-dg}.

\section{Examples}
\label{section:examples}
We will compute some examples,
to illustrate some interesting phenomena.

\subsection{Symmetric square stack of \texorpdfstring{$\mathbb{P}^1$}{the projective line}}
\label{subsection:symmetric-square-P1}
The first case we consider is~$[\Sym^2\mathbb{P}^1]$,
where we will assume~$\field$ algebraically closed.
We will compute its Hochschild cohomology both geometrically using the main theorem of this paper,
and algebraically, using finite-dimensional algebras.
To get started with the latter,
using \cite[Proposition~4.5]{MR3397451} we obtain the following result.
\begin{lemma}
  \label{lemma:symmetric-square-projective-line}
  There exists a derived equivalence~$\derived^\bounded([\Sym^2\mathbb{P}^1])\cong\derived^\bounded(\field Q/I)$
  where~$Q$ is the quiver
  
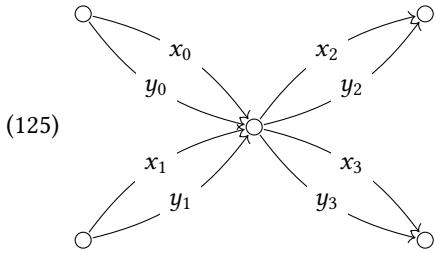
\begin{equation}
    \label{equation:symmetric-square-projective-line-quiver}
    \begin{tikzpicture}[scale=1.5,baseline={([yshift=-.5ex]current bounding box.center)}]
      \draw[circle, minimum size = 7pt, inner sep = 0pt]
        (0,2) node (a) {}
        (0,0) node (b) {}
        (1.5,1) node (c) {}
        (3,2) node (d) {}
        (3,0) node (e) {};
      \draw (a) circle (2pt)
            (b) circle (2pt)
            (c) circle (2pt)
            (d) circle (2pt)
            (e) circle (2pt);
      \draw[->] (a) edge [bend left  = 20] node [fill = white] {$x_0$} (c)
                (a) edge [bend right = 20] node [fill = white] {$y_0$} (c)
                (b) edge [bend left  = 20] node [fill = white] {$x_1$} (c)
                (b) edge [bend right = 20] node [fill = white] {$y_1$} (c)
                (c) edge [bend left  = 20] node [fill = white] {$x_2$} (d)
                (c) edge [bend right = 20] node [fill = white] {$y_2$} (d)
                (c) edge [bend left  = 20] node [fill = white] {$x_3$} (e)
                (c) edge [bend right = 20] node [fill = white] {$y_3$} (e);
    \end{tikzpicture}
  \end{equation}
  and~$I=(x_0y_2-y_0x_2,x_1y_3-y_1x_3,x_0x_3,y_0y_3,x_0y_3+y_0x_3,x_1x_2,y_1y_2,x_1y_2+y_1x_2)$.
%I := [xp*ys - yp*xs,
%      xq*yt - yq*xt,
%      xp*xt, yp*yt, xp*yt + yp*xt,
%      xq*xs, yq*ys, xq*ys + yq*xs];
\end{lemma}

\begin{proof}
  Taking the full exceptional collection~$\mathcal{O}_{\mathbb{P}^1},\mathcal{O}_{\mathbb{P}^1}(1)$
  for~$\derived^\bounded(\mathbb{P}^1)$
  yields the full 3-block exceptional collection
  \begin{equation}
    \derived^\bounded([\Sym^2\mathbb{P}^1])
    =
    \big\langle
      \begin{matrix}
        \mathcal{O}_{\mathbb{P}^1}\boxtimes\mathcal{O}_{\mathbb{P}^1} \\
        \mathcal{O}_{\mathbb{P}^1}\boxtimes\mathcal{O}_{\mathbb{P}^1}\otimes\sgn
      \end{matrix},
      \Ind(\mathcal{O}_{\mathbb{P}^1}\boxtimes\mathcal{O}_{\mathbb{P}^1}(1)),
      \begin{matrix}
        \mathcal{O}_{\mathbb{P}^1}(1)\boxtimes\mathcal{O}_{\mathbb{P}^1}(1) \\
        \mathcal{O}_{\mathbb{P}^1}(1)\boxtimes\mathcal{O}_{\mathbb{P}^1}(1)\otimes\sgn
      \end{matrix}
      \big\rangle,
  \end{equation}
  where objects in the same block are completely orthogonal.
  Denoting~$V=\HH^0(\mathbb{P}^1,\mathcal{O}_{\mathbb{P}^1}(1))$ we have that
  morphisms between objects in consecutive blocks are given by~$V$.
  The other Hom-spaces are accordingly given by~$\Sym^2V$ resp.~$\bigwedge^2V$.
  If we choose a basis~$x,y$ of~$V$
  and label the bases of the Hom-spaces using~$x_i,y_i$,
  we get the quiver in \eqref{equation:symmetric-square-projective-line-quiver}
  with the claimed relations.
\end{proof}

This allows us to compute the Hochschild cohomology of~$[\Sym^2\mathbb{P}^1]$ purely algebraically.

\begin{lemma}
  \label{lemma:hochschild-symmetric-square-projective-line}
  For the finite-dimensional algebra~$\field Q/I$ from \cref{lemma:symmetric-square-projective-line}
  we have that it is of global dimension~2,
  and
  \begin{equation}
    \label{equation:hochschild-symmetric-square-line}
    \sum_{i=0}^2\dim_\field\hochschild^i(\field Q/I)t^i
    =1+3t+3t^2.
  \end{equation}
\end{lemma}

\begin{proof}
  The proof is similar to the method used in \cite[\S3]{MR3988086}.
  The quiver is acyclic, thus~$\hochschild(\field Q/I)\cong\field$.
  The algebra is not hereditary
  by the presence of relations,
  and its global dimension is bounded above by~2 because of the length of a maximal path,
  thus precisely~2.
  Because~$\operatorname{gl\,dim}A=\operatorname{proj\,dim}_{A^\env}A$ for~$A$
  a finite-dimensional algebra over an algebraically closed field \cite[Lemma~1.5]{MR1035222},
  we only need to determine~$\hochschild^1(\field Q/I)$ and~$\hochschild^2(\field Q/I)$.

  By \cite{MR1444101} we have
  \begin{equation}
    \chi(\hochschild^\bullet(A))=-\operatorname{tr}\mathrm{C}_{\field Q/I}=-1
  \end{equation}
  where~$\mathrm{C}_{\field Q/I}$ is the Coxeter matrix,
  thus~$\hochschild^1(\field Q/I)\cong\hochschild^2(\field Q/I)$.
  The Coxeter matrix is readily computed from the Cartan matrix~$\mathrm{A}_{\field Q/I}$,
  as~$\mathrm{C}_{\field Q/I}=-\mathrm{A}_{\field Q/I}^{-1}\mathrm{A}_{\field Q/I}^{\mathrm{t}}$.
  For the choice of basis given by the exceptional objects in \cref{lemma:symmetric-square-projective-line}
  the Cartan matrix is given by
  \begin{equation}
    \mathrm{A}_{\field Q/I}
    =
    \begin{pmatrix}
      1 & 0 & 2 & 3 & 1 \\
      0 & 1 & 2 & 1 & 3 \\
      0 & 0 & 1 & 2 & 2 \\
      0 & 0 & 0 & 1 & 0 \\
      0 & 0 & 0 & 0 & 1
    \end{pmatrix}.
  \end{equation}

  To determine~$\hochschild^1(\field Q/I)$,
  recall that the first Hochschild cohomology
  (as a Lie algebra)
  is isomorphic to the Lie algebra of the outer automorphism group of~$\field Q/I$,
  see, e.g., \cite[Proposition~1.1]{MR1905030}.
  In our case,
  observe that the natural action of~$\operatorname{GL}(V)$ on each individual Hom-space
  gives us an identification~$\operatorname{Out}^0(\field Q/I)\cong\operatorname{PGL}(V)$:
  the ideal~$I$ forces the automorphisms to lie in the diagonal subgroup in the product of the~$\operatorname{GL}(V)$,
  and the passage to outer automorphisms yields~$\operatorname{PGL}(V)$.
  Hence~$\hochschild^1(\field Q/I)=\operatorname{Lie}\operatorname{Out}^0(\field Q/I)\cong\mathfrak{sl}_2$,
\end{proof}

On the other hand, in \cref{corollary:hochschild-serre} the only contribution for~$n=2$
is given by~$i=0$ on the right-hand side by degree reasons,
for which~$\hochserre_0^*(\mathbb{P}^1)\cong\hochschild^*(\mathbb{P}^1)\cong \field[0]\oplus\mathfrak{sl}_2[-1]$.
We obtain the dimensions in \eqref{equation:hochschild-symmetric-square-line},
see also \cref{corollary:first-and-second}.
Thus it agrees with \cref{lemma:hochschild-symmetric-square-projective-line}.

\subsection{Hilbert square of \texorpdfstring{$\mathbb{P}^2$}{the projective plane}}
\label{subsection:hilbert-square-P2}
Another interesting example is that of~$[\Sym^2\mathbb{P}^2]$,
which by \eqref{equation:derived-mckay} is derived equivalent to~$\hilbn{2}{\mathbb{P}^2}$.
Using \cref{subsection:hochschild-serre-description} we can compute the Hochschild cohomology of~$\hilbn{2}{\mathbb{P}^2}$ and obtain
\begin{lemma}
  \label{lemma:hochschild-hilbert-series-hilbert-square-P2}
  We have that
  \begin{equation}
    \label{equation:hochschild-Sym-2-P2}
    \sum_{i=0}^8\dim_\field\hochschild^i(\hilbn{2}{\mathbb{P}^2})t^i
    =1+8t+48t^2+115t^3+83t^4.
  \end{equation}
\end{lemma}
But as mentioned before,
we cannot compute the Hochschild--Kostant--Rosenberg decomposition of~$\hilbn{2}{\mathbb{P}^2}$,
which is a bigraded decomposition,
using symmetric quotient stacks.

To do this, we can use the isomorphism
\begin{equation}
  \hilbn{2}{\mathbb{P}^2}\cong\mathbb{P}_{\mathbb{P}^2}(\Sym^2\Omega_{\mathbb{P}^2}(1)),
\end{equation}
from, e.g.,~\cite[\S3.2]{MR3488782}.
This description is explicit enough to compute the Hochschild--Kostant--Rosenberg decomposition
as in the following proposition.

\begin{proposition}
  \label{proposition:hkr-hilbert-square-P2}
  We have that
  \begin{equation}
    \begin{aligned}
      \HHHH^1(\hilbn{2}{\mathbb{P}^2})
      &\cong\HH^0(\hilbn{2}{\mathbb{P}^2},\tangent_{\hilbn{2}{\mathbb{P}^2}}) \\
      &\cong \field^8 \\
      \HHHH^2(\hilbn{2}{\mathbb{P}^2})
      &\cong\HH^1(\hilbn{2}{\mathbb{P}^2},\tangent_{\hilbn{2}{\mathbb{P}^2}})\oplus\HH^0(\hilbn{2}{\mathbb{P}^2},\bigwedge^2\tangent_{\hilbn{2}{\mathbb{P}^2}}) \\
      &\cong \field^{10}\oplus \field^{38} \\
      \HHHH^3(\hilbn{2}{\mathbb{P}^2})
      &\cong\HH^1(\hilbn{2}{\mathbb{P}^2},\bigwedge^2\tangent_{\hilbn{2}{\mathbb{P}^2}})\oplus\HH^0(\hilbn{2}{\mathbb{P}^2},\bigwedge^3\tangent_{\hilbn{2}{\mathbb{P}^2}}) \\
      &\cong \field^{35}\oplus \field^{80} \\
      \HHHH^4(\hilbn{2}{\mathbb{P}^2})
      &\cong\HH^1(\hilbn{2}{\mathbb{P}^2},\bigwedge^3\tangent_{\hilbn{2}{\mathbb{P}^2}})\oplus\HH^0(\hilbn{2}{\mathbb{P}^2},\bigwedge^4\tangent_{\hilbn{2}{\mathbb{P}^2}}) \\
      &\cong \field^{28}\oplus \field^{55} \\
      \HHHH^{\geq 5}(\hilbn{2}{\mathbb{P}^2})&=0
    \end{aligned}
  \end{equation}
\end{proposition}

\begin{proof}
  The cohomology of~$\tangent_{\hilbn{2}{\mathbb{P}^2}}$ is computed in \cref{corollary:hilbert-square-tangent},
  the cohomology of~$\bigwedge^2\tangent_{\hilbn{2}{\mathbb{P}^2}}$ is computed in \cref{corollary:hilbert-square-bitangent},
  and the cohomology of~$\bigwedge^3\tangent_{\hilbn{2}{\mathbb{P}^2}}$ is computed in \cref{corollary:hilbert-square-tritangent}.
  For~$\bigwedge^4\tangent_{\hilbn{2}{\mathbb{P}^2}}\cong\omega_{\hilbn{2}{\mathbb{P}^2}}$
  the cohomology is readily computed using that~$\hilbn{2}{\mathbb{P}^2}\to\Sym^2\mathbb{P}^2$ is a crepant resolution.
\end{proof}

Because the proof requires some Borel--Weil--Bott calculations we relegate the details of the computation to \cref{appendix:hilbert-square-P2}.

\begin{remark}
  \label{remark:rigidity}
  A similar computation for~$\hilbn{2}{\mathbb{P}^n}\cong\mathbb{P}_{\operatorname{Gr}(2,n+1)}(\Sym^2\mathcal{S})$,
  see \cref{corollary:hilbert-square-tangent} for more details,
  allows to recover the rigidity result from \cite[Corollary~37]{MR3950704} in this special case.
\end{remark}

\subsection{Hilbert squares of bielliptic surfaces}
\label{subsection:bielliptic}
For~$\hochschild_*(\hilbn{n}{S})$ we have that~$\hochschild_*(S)$
determines it completely as a graded vector space.
Thus a natural question to ask is:
\begin{quote}
  \emph{Is~$\hochschild^*(\hilbn{n}{S})$ determined by~$\hochschild^*(S)$ (and~$\hochschild_*(S)$)?}
\end{quote}
Equation \eqref{equation:hochschild-cohomology-hilbert-scheme} in \cref{theorem:intro-main} suggests that the answer to the question is negative.
However, in order to really prove this, we need an example of a pair of surfaces,
which have the same $\hochschild^*$ and $\hochschild_*$, but different negative parts $\hochserre_{\le -1}$ of Hochschild--Serre cohomology.
In the following, we will find such a pair among the class of bielliptic surfaces,
which is a class of surfaces of Kodaira dimension~0.

%One can consider this question with varying amounts of structure on Hochschild cohomology.
%We will show that there exist surfaces~$S$ and $T$ for which
%\begin{equation}
%  \hochschild^*(S)\cong\hochschild^*(T)
%\end{equation}
%yet
%\begin{equation}
%  \hochschild^*(\hilbn{2}{S})\not\cong\hochschild^*(\Hilb^2T),
%\end{equation}
%already \emph{as graded vector spaces}.

% recall some basics on bielliptic surfaces
\paragraph{Bielliptic surfaces}
A \emph{bielliptic surface}~$S$ is a surface of the form~$(E\times F)/G$, where
\begin{itemize}
  \item $E$ and $F$ are elliptic curves,
  \item $G$ is a finite abelian group acting diagonally on~$E\times F$,
  \item $G$ acts by automorphisms on~$E$ (so that $E/G\cong\mathbb{P}^1$) and by translation on~$F$.
\end{itemize}
There are 7 deformation families, given by the Bagnera--de Franchis list.
% complex numbers and algebraically closed field of characteristic 0 gives the same classification
Each family depends on the isomorphism type of~$E$ and the shape of the action of~$G$,
and the full list is given in \cref{table:bielliptic-classification}.

Their Hodge diamond is of the form
\begin{equation}\label{eq:biellHodge}
  \begin{array}{ccccc}
    & & 1 \\
    & 1 & & 1 \\
    0 & & 2 & & 0 \\
    & 1 & & 1 \\
    & & 1
  \end{array}
\end{equation}

\begin{table}
  \centering
  \begin{tabular}{ccccc}
    \toprule
             & $j(E)$    & $G$                                                  & action of $G$                                               & order of $\omega_S$ \\
    \midrule
    1        & arbitrary & $\mathbb{Z}/2\mathbb{Z}$                             & $e\mapsto-e$                                                & 2 \\
    2        & arbitrary & $\mathbb{Z}/2\mathbb{Z}\oplus\mathbb{Z}/2\mathbb{Z}$ & $e\mapsto -e$, $e\mapsto e+c$ where $c=-c$                  & 2 \\
    3        & 0         & $\mathbb{Z}/3\mathbb{Z}$                             & $e\mapsto \zeta e$                                          & 3 \\
    4        & 0         & $\mathbb{Z}/3\mathbb{Z}\oplus\mathbb{Z}/3\mathbb{Z}$ & $e\mapsto\zeta e$, $e\mapsto e+c$ where $c=\zeta c$         & 3 \\
    5        & 1728      & $\mathbb{Z}/2\mathbb{Z}$                             & $e\mapsto\mathrm{i}e$                                       & 4 \\
    6        & 1728      & $\mathbb{Z}/4\mathbb{Z}\oplus\mathbb{Z}/2\mathbb{Z}$ & $e\mapsto\mathrm{i}e$, $e\mapsto e+c$ where $c=\mathrm{i}c$ & 4 \\
    7        & 0         & $\mathbb{Z}/6\mathbb{Z}$                             & $e\mapsto-\zeta e$                                          & 6 \\
    \bottomrule
  \end{tabular}
  \caption{Classification of bielliptic surfaces}
  \label{table:bielliptic-classification}
\end{table}

We recall the properties of bielliptic surfaces as explained in \cite[\S4.D]{MR0565468}.
The Albanese morphism
\begin{equation}
  \alb\colon S\to\Pic^0(S)=C
\end{equation}
is a smooth surjection onto an elliptic curve,
so we have a short exact sequence
\begin{equation}
  0\to\mathcal{O}_S\to\Omega_S^1\to\Omega_{S/C}^1\to 0
\end{equation}
which moreover can be shown to split.
Defining
\begin{equation}
  {L}\colonequals\mathrm{R}^1{\alb_*}\mathcal{O}_S,
\end{equation}
we have~$\Omega_{S/C}^1\cong\alb^*{L}^\vee$.
Thus we obtain
\begin{equation}
  \label{equation:tangent-and-anticanonical-bielliptic}
  \begin{aligned}
    \tangent_S&\cong\mathcal{O}_S\oplus\alb^*{L} \\
    \omega_S^\vee&\cong\alb^*{L}.
  \end{aligned}
\end{equation}

The following straightforward cohomology computation is the key lemma.
\begin{lemma}
  \label{lemma:bielliptic-cohomology-lemma}
  We have that
  \begin{equation}
    \HH^\bullet(S,\alb^*{L})\cong
    \begin{cases}
      0 & \ord\omega_S=3,4,6 \\
      \field[-1]\oplus \field[-2] & \ord\omega_S=2
    \end{cases}
  \end{equation}
  and
  \begin{equation}
    \HH^\bullet(S,\alb^*{L}^{\otimes2})\cong
    \begin{cases}
      0 & \ord\omega_S=4,6 \\
      \field[-1]\oplus \field[-2] & \ord\omega_S=3 \\
      \field[0]\oplus \field[-1] & \ord\omega_S=2.
    \end{cases}
  \end{equation}
\end{lemma}

\begin{proof}
  This follows from the isomorphisms
  \begin{equation}
    \begin{aligned}
      \mathbf{R}{\alb_*}\circ\alb^*{L}&\cong{L}\oplus{L}^{\otimes2}[-1], \\
      \mathbf{R}{\alb_*}\circ\alb^*{L}^{\otimes2}&\cong{L}^{\otimes2}\oplus{L}^{\otimes3}[-1]
    \end{aligned}
  \end{equation}
  obtained via the projection formula,
  using that~$\mathbf{R}{\alb_*}\mathcal{O}_S\cong\mathcal{O}_C[0]\oplus{L}[-1]$ (as~$\dim C=1$ and~$\alb$ has connected fibers),
  and the fact that non-trivial torsion line bundles are cohomology-free on~$C$
  so the only cohomology appears when~${L}^{\otimes i}\cong\mathcal{O}_C$.
\end{proof}

We thus obtain the following.
\begin{lemma}
  \label{lemma:hochschild-cohomology-bielliptic}
  Let~$S$ be a bielliptic surface.
  Then
  \begin{equation}
    \sum_{i=0}^4\dim_\field\hochschild^i(S)t^i
    =
    \begin{cases}
      1+2t+t^2           & \ord\omega_S=3,4,6 \\
      1+2t+2t^2+2t^3+t^4 & \ord\omega_S=2.
    \end{cases}
  \end{equation}
  %\begin{equation}
  %  \hochschild^*(S)\cong
  %  \begin{cases}
  %    \field[0]\oplus \field^2[-1]\oplus \field[-2]
  %    & \ord\omega_S=3,4,6\\
  %    \field[0]\oplus \field^2[-1]\oplus \field^2[-2]\oplus \field^2[-3]\oplus \field[-4]
  %    & \ord\omega_S=2.
  %  \end{cases}
  %\end{equation}
\end{lemma}

\begin{proof}
  This is a straightforward computation using the Hochschild--Kostant--Rosenberg decomposition \eqref{equation:hkr-varieties},
  with the cohomology of~$\tangent_S$ and~$\omega_S^\vee$
  being determined by \cref{lemma:bielliptic-cohomology-lemma},
  which gives
  \begin{equation}\label{eq:biellTcoh}
    \sum_{i=0}^2\hh^i(S,\tangent_S)t^i
    =
    \begin{cases}
      1+t      & \ord\omega_S=3,4,6 \\
      1+2t+t^2 & \ord\omega_S=2
    \end{cases}
  \end{equation}
  and
  \begin{equation}
    \sum_{i=0}^2\hh^i(S,\omega_S^\vee)t^i
    =
    \begin{cases}
      0     & \ord\omega_S=3,4,6 \\
      t+t^2 & \ord\omega_S=2.
    \end{cases}
  \end{equation}
  %See also \cite[Theorem~4.9]{MR0565468} for the cohomology of the tangent bundle.
\end{proof}

Next we compute a slice of the Hochschild--Serre cohomology.
\begin{lemma}
  \label{lemma:hochschild-serre-bielliptic-slice}
  Let~$S$ be a bielliptic surface.
  Then
  \begin{equation}
    \sum_{i=0}^4\dim_\field\hochserre_{-1}^i(S)t^i=
    \begin{cases}
      0
      & \ord\omega_S=4,6 \\
      t^4+2t^5+t^6
      & \ord\omega_S=3 \\
      2t^3+4t^4+2t^5
      & \ord\omega_S=2.
    \end{cases}
  \end{equation}
  %\begin{equation}
  %  \hochserre_{-1}^*(S)\cong
  %  \begin{cases}
  %    0
  %    & \ord\omega_S=4,6 \\
  %    \field[-2]\oplus \field^2[-3]\oplus \field[-4]
  %    & \ord\omega_S=3 \\
  %    \field^2[-1]\oplus \field^4[-2]\oplus \field^2[-3]
  %    & \ord\omega_S=2.
  %  \end{cases}
  %\end{equation}
\end{lemma}

\begin{proof}
  This proof is analogous to that of \cref{lemma:hochschild-cohomology-bielliptic},
  now using \eqref{equation:hochschild-serre-decomposition}
  and \cref{lemma:bielliptic-cohomology-lemma}
  to compute the Hochschild--Kostant--Rosenberg decomposition \eqref{equation:hkr-varieties}.
\end{proof}

As an application of the description in \cref{subsection:hochschild-serre-description}
we obtain the following.
\begin{proposition}
  \label{proposition:hilbert-square-bielliptic}
  Let~$S$ be a bielliptic surface.
  Then
  \begin{equation}
    \begin{aligned}
      &\sum_{i=0}^8\dim_\field\hochschild^i(\hilbn{2}{S})t^i \\
      &\qquad=
      \begin{cases}
        1+2t+2t^2+2t^3+t^4
        & \ord\omega_S=4,6 \\
        1+2t+2t^2+2t^3+2t^4+2t^5+t^6
        & \ord\omega_S=3 \\
        1+2t+3t^2+8t^3+12t^4+8t^5+3t^6+2t^7+t^8
        & \ord\omega_S=2.
      \end{cases}
    \end{aligned}
  \end{equation}
\end{proposition}

\begin{proof}
  We consider the partitions~$\nu=(2)$ resp.~$\nu=(1,1)$, with~$\lambda=(0,1)$ resp.~$\lambda=(2)$.
  The first partition gives a copy of~$\hochserre_{-1}^*(S)$,
  described in \cref{lemma:hochschild-serre-bielliptic-slice}.
  The second partition gives a copy of~$\Sym^2(\hochschild^*(S))$,
  which is computed using \cref{lemma:hochschild-cohomology-bielliptic}.
\end{proof}

All bielliptic surfaces have the same Hodge diamond \eqref{eq:biellHodge}. Hence, by the Hochschild--Kostant--Rosenberg theorem, they also have the same Hochschild homology.
Combining this with \cref{lemma:hochschild-cohomology-bielliptic,proposition:hilbert-square-bielliptic}
we obtain the following.
\begin{corollary}
  \label{corollary:bielliptic-disagreement}
  Let~$S$ be a bielliptic surface with~$\ord\omega_S=3$
  and let~$T$ be a bielliptic surface with~$\ord\omega_T\in\{4,6\}$.
  Then~$\HHHH^*(S)\cong\HHHH^*(T)$ and $\HHHH_*(S)\cong\HHHH_*(T)$ as graded vector spaces,
  but~$\HHHH^*(\hilbn{2}{S})\not\cong\HHHH^*(\hilbn{2}{T})$.
\end{corollary}

\begin{remark}
  We see that~$\HHHH^2(\hilbn{2}{S})$ is strictly bigger than~$\HHHH^2(S)$,
  regardless of~$\ord\omega_S$,
  and we have that~$\HH^2(\hilbn{2}{S},\mathcal{O}_{\hilbn{2}{S}})=0$.
  These considerations in fact hold for arbitrary~$n\geq 2$,
  using the degree considerations in \cref{subsection:hochschild-serre-description}.
  Thus (infinitesimally) there might be new commutative or noncommutative deformations of Hilbert schemes of points
  of bielliptic surfaces.
  We come back to this in \cref{example:bielliptic-hilbert-conjecture}.
\end{remark}

\section{Applications to Hilbert schemes}
\label{section:applications}
In this section we discuss various applications of our computations for the Hochschild(--Serre) cohomology of symmetric quotient stacks
to the study of some \textit{classical problems} about Hilbert schemes of points on surfaces.
In \cref{subsection:Infinitesmial-Aut}, we discuss~$\HHHH^1(\hilbn{n}{S})$ and applications to automorphisms.
In \cref{subsection:deformations} we discuss~$\HHHH^2(\hilbn{n}{S})$
and applications to their deformation theory.
In \cref{subsection:boissiere} we suggest a revised version of a conjecture by Boissi\`ere (see \cref{conjecture:corrected}),
and in \cref{subsection:nieper-wisskirchen} we explain
how results of Nieper-Wi\ss kirchen can be used to provide evidence for the revised conjecture.
Throughout this section, $S$ is a smooth projective surface defined over a field of characteristic zero.

\subsection{Infinitesimal automorphisms of Hilbert schemes}
\label{subsection:Infinitesmial-Aut}
The following corollary to \cref{theorem:intro-main}
is already proven in \cite[Corollaire~1]{MR2932167}
in the context of complex analytic surfaces,
using a generalisation of a computation of G\"ottsche.
For us, using non-commutative methods, it is an easy consequence of the degree shifts
as they appear in \cref{subsection:hochschild-serre-description},
\emph{without} appealing to \cite[Proposition~1]{MR2932167} (see also \cref{proposition:boissiere-global-sections}).
\begin{corollary}[Boissi\`ere]
  \label{corollary:boissiere}
  Let~$S$ be a smooth projective surface.
  Then~$\dim\Aut^0(S)=\dim\Aut^0(\hilbn{n}{S})$ for all~$n\geq 1$.
\end{corollary}

\begin{proof}[Proof of \cref{corollary:boissiere} using \cref{theorem:intro-main}]
  By \eqref{equation:first-hochschild} in \cref{corollary:first-and-second}
  and the Bridgeland--King--Reid--Haiman equivalence \eqref{equation:derived-mckay}
  we obtain the identification
  \begin{equation}
    \hochschild^1(S)\cong\hochschild^1([\Sym^nS])\cong\hochschild^1(\hilbn{n}{S}).
  \end{equation}
  By the Hochschild--Kostant--Rosenberg decomposition we have
  \begin{equation}
    \HHHH^1(\hilbn{n}{S})\cong\HH^1(\hilbn{n}{S},\mathcal{O}_{\hilbn{n}{S}})\oplus\HH^0(\hilbn{n}{S},\tangent_{\hilbn{n}{S}}).
  \end{equation}
  Taking~$\symgr_n$-invariant sections of~$\HH^1(S^n,\mathcal{O}_{S^n})$,
  we get
  \begin{equation}
    \HH^1(\hilbn{n}{S},\mathcal{O}_{\hilbn{n}{S}})\cong\HH^1(S,\mathcal{O}_S).
  \end{equation}
  Hence, we conclude that
  \belowdisplayskip=-12pt
  \begin{equation}
    \dim\Aut^0(S)=\dim_\field\HH^0(S,\tangent_S)=\dim_\field\HH^0(\hilbn{n}{S},\tangent_{\hilbn{n}{S}})=\dim\Aut^0(\hilbn{n}{S}).
  \end{equation}
  \qedhere
\end{proof}

\subsection{Deformations of Hilbert schemes}
\label{subsection:deformations}
It is an interesting question to understand how the deformation theory of a smooth projective surface~$S$
determines the deformation theory of the Hilbert scheme~$\hilbn{n}{S}$.
Using the relative Hilbert scheme one observes that
a deformation of~$S$ induces a deformation of~$\hilbn{n}{S}$,
thus we obtain an \emph{inclusion of deformation functors}~$\Def_S\to\Def_{\hilbn{n}{S}}$.
It is therefore important to understand
whether it is an isomorphism,
or what measures the difference between the two.
In particular, we are interested in comparing~$\HH^1(S,\tangent_S)$
to~$\HH^1(\hilbn{n}{S},\tangent_{\hilbn{n}{S}})$.

It was shown by Fantechi that for \emph{surfaces of general type}
the two deformation theories are the same \cite[Theorems~0.1 and~0.3]{MR1354269}.
More generally, she proves the following.
\begin{theorem}[Fantechi]
  \label{theorem:fantechi}
  Let~$S$ be a surface for which
  \begin{itemize}
    \item $\HH^0(S,\tangent_S)=0$ or $\HH^1(S,\mathcal{O}_S)=0$;
    \item $\HH^0(S,\omega_S^\vee)=0$.
  \end{itemize}
  Then the natural morphism~$\Def_S\to\Def_{\hilbn{n}{S}}$ is an isomorphism.
  In particular
  \begin{equation}
    \label{equation:infinitesimal-fantechi}
    \HH^1(S,\tangent_S)\cong\HH^1(\hilbn{n}{S},\tangent_{\hilbn{n}{S}}).
  \end{equation}
\end{theorem}

On the other hand,
Hitchin proves in \cite[\S4.1]{MR3024823} the following.
\begin{theorem}[Hitchin]
  \label{theorem:hitchin}
  Let~$S$ be a surface for which~$\HH^1(S,\mathcal{O}_S)=0$.
  Then there is a natural split exact sequence
  \begin{equation}
    \label{equation:hitchin-sequence}
    0\to\HH^1(S,\tangent_S)\to\HH^1(\hilbn{n}{S},\tangent_{\hilbn{n}{S}})\to\HH^0(S,\omega_S^\vee)\to 0.
  \end{equation}
\end{theorem}
Thus the existence of Poisson structures on~$S$
gives geometric deformations of~$\hilbn{n}{S}$ which are induced by noncommutative deformations of~$S$.
In \cite{MR3950704} this result is considered from the point-of-view of derived categories.
The geometric deformations of~$\hilbn{n}{S}$ are studied in e.g.~\cite{MR2303228,MR2102090} for noncommutative deformations of~$\mathbb{P}^2$,
and in \cite{MR3669875} for~$S$ a (noncommutative) deformation of a del Pezzo surface.

Finally, the following was shown in \cite{MR1346215,MR1660136} by Bottacin.
\begin{theorem}[Bottacin]
  \label{theorem:bottacin}
  Let~$S$ be a surface, and let~$\sigma$ be a non-trivial Poisson structure on~$S$.
  Then there is a natural non-trivial Poisson structure~$\sigma^{[n]}$ on~$\hilbn{n}{S}$.
\end{theorem}
%But as we have seen in \cref{subsection:hilbert-square-P2}
%there is no identification between~$\HH^0(S,\omega_S^\vee)$
%and~$\HH^0(\hilbn{n}{S},\bigwedge^2\tangent_{\hilbn{n}{S}})$ in general,
%so there might be \emph{more} Poisson structures than the ones
%induced by the construction in op.~cit.
%We will first show that this does not happen for a large class of surfaces.

Using \cref{theorem:intro-main} we can
prove a unified version of \cref{theorem:fantechi,theorem:hitchin}
(and heuristically understand \cref{theorem:bottacin})
using stacky methods
and the Bridgeland--King--Reid--Haiman equivalence \eqref{equation:derived-mckay}.

\begin{proof}[Proof of \cref{corollary:tangent-hilbn-S}]
  By \eqref{equation:second-hochschild} in \cref{corollary:first-and-second},
  the derived McKay correspondence \eqref{equation:derived-mckay},
  the derived invariance of Hochschild cohomology in \cref{corollary:hochschild-serre-comparison-hilbert-scheme},
  and the Hochschild--Kostant--Rosenberg decomposition for~$S$ we have that
  \begin{equation}
    \label{equation:hkr-symn-S}
    \begin{aligned}
      \hochschild^2(\hilbn{n}{S})
      &\cong\hochschild^2([\Sym^nS]) \\
      &\cong\HH^2(S,\mathcal{O}_S)\oplus\HH^1(S,\tangent_S)\oplus\HH^0(S,\omega_S^\vee) \\
      &\qquad\oplus\bigwedge^2\HH^1(S,\mathcal{O}_S)\oplus\bigwedge^2\HH^0(S,\tangent_S)\oplus\left(\HH^1(S,\mathcal{O}_S)\otimes_\field\HH^0(S,\tangent_S)\right) \\
      &\qquad\oplus\HH^0(S,\omega_S^\vee).
    \end{aligned}
  \end{equation}
  We wish to match this to the Hochschild--Kostant--Rosenberg decomposition for~$\hilbn{n}{S}$,
  which reads
  \begin{equation}
    \label{equation:hkr-hilbn-S}
    \hochschild^2(\hilbn{n}{S})
    \cong
    \HH^2(\hilbn{n}{S},\mathcal{O}_{\hilbn{n}{S}})\oplus\HH^1(\hilbn{n}{S},\tangent_{\hilbn{n}{S}})\oplus\HH^0(\hilbn{n}{S},\bigwedge^2\tangent_{\hilbn{n}{S}}).
  \end{equation}
  For the first summand in \eqref{equation:hkr-hilbn-S} we can take~$\symgr_n$-invariant sections of $\HH^2(S^n,\mathcal O_{S^n})$ to obtain
  \begin{equation}
    \label{equation:gerby-hilbn-S}
    \HH^2(\hilbn{n}{S},\mathcal{O}_{\hilbn{n}{S}})
    \cong
    \HH^2(S,\mathcal{O}_S)
    \oplus\bigwedge^2\HH^1(S,\mathcal{O}_S).
  \end{equation}
 The third summand in \eqref{equation:hkr-hilbn-S} is given by
  \begin{equation}
    \label{equation:bivectors-hilbn-S}
    \HH^0(\hilbn{n}{S},\bigwedge^2\tangent_{\hilbn{n}{S}})
    \cong
    \bigwedge^2\HH^0(S,\tangent_S)\oplus \HH^0(S,\omega_S^\vee).
  \end{equation}
as can be extracted from \cite[Proposition~1]{MR2932167}, or, more directly as the $\#=2$ case of \cref{cor:polyvectorfields} below.
 Thus, by cancelling \eqref{equation:bivectors-hilbn-S} and \eqref{equation:gerby-hilbn-S}
  in \eqref{equation:hkr-symn-S} we obtain the identification \eqref{equation:tangent-hilbn-S}.
\end{proof}

\begin{remark}
  An argument similar to the computation for \eqref{equation:bivectors-hilbn-S} tells us that
  \begin{equation}
    \HH^0(\hilbn{n}{S},\bigwedge^3\tangent_{\hilbn{n}{S}})
    \cong
    \begin{cases}
      \HH^0(S,\tangent_S)\otimes_\field\HH^0(S,\omega_S^\vee) & n=2 \\
      \HH^0(S,\tangent_S)\otimes_\field\HH^0(S,\omega_S^\vee)\oplus\bigwedge^3\HH^0(S,\tangent_S) & n\geq 3.
    \end{cases}
  \end{equation}
  Thus, if~$\HH^0(S,\tangent_S)=0$ every pre-Poisson structure on~$\hilbn{n}{S}$ is automatically Poisson.
  It would be interesting to compare this to \cref{theorem:bottacin},
  which is a more intrinsic recipe to induce Poisson structures on~$\hilbn{n}{S}$ from~$S$.
\end{remark}

Let us point out the following observation.
\begin{example}
  \label{example:bielliptic-hilbert-conjecture}
  The case of bielliptic surfaces (see \cref{subsection:bielliptic})
  was not yet covered by \cref{theorem:fantechi,theorem:hitchin}.
  We already computed in \cref{proposition:hilbert-square-bielliptic} that~$\hochschild^2(\hilbn{2}{S})$
  is strictly bigger than~$\hochschild^2(S)\cong\HH^1(S,\tangent_S)$.
  Now, by \cref{corollary:tangent-hilbn-S}
  we in fact have for all~$n\geq 2$ that
  \begin{equation}
    \begin{aligned}
      \hochschild^2(\hilbn{n}{S})
      &\cong\HH^1(\hilbn{n}{S},\tangent_{\hilbn{n}{S}}) \\
      &\cong\HH^1(S,\tangent_S)\oplus\HH^0(S,\tangent_S)\otimes_\field\HH^1(S,\mathcal{O}_S),
    \end{aligned}
  \end{equation}
  thus all first-order deformations are commutative,
  and
  \begin{equation}
    \hh^1(\hilbn{n}{S},\tangent_{\hilbn{n}{S}})
    =
    \begin{cases}
      2 & \ord\omega_S=3,4,6 \\
      3 & \ord\omega_S=2,
    \end{cases}
  \end{equation}
  by \eqref{eq:biellTcoh}.
  Thus we get precisely one new deformation direction.

  By \cite[Corollary~2]{MR1144440} we have that
  deformations of smooth projective varieties with torsion canonical bundle are unobstructed.
  Thus, Hilbert schemes of points on bielliptic surfaces
  have \emph{genuinely new deformations}.
  It would be interesting to understand these deformations.
\end{example}

\subsection{On Boissi\`ere's conjecture}
\label{subsection:boissiere}
Recall from \cref{subsection:hilbert-consequences} that for every line bundle $L$ on a smooth projective surface $S$,
there is an associated line bundle $L_n$ on the Hilbert scheme $\hilbn{n}{S}$.

Boissi\`ere conjectured in \cite[Conjecture~1]{MR2932167}
a generating formula for the twisted Hodge numbers
\begin{equation}
  \label{equation:twisted-hodge}
  \hh^{p,q}(\hilbn{n}{S},{L}_n)\colonequals\hh^q(\hilbn{n}{S},\Omega_{\hilbn{n}{S}}^p\otimes{L}_n),
\end{equation}
which predicts the equality
\begin{equation}
  \label{equation:boissiere-conjecture}
  \sum_{n=0}^{+\infty}\sum_{p=0}^{2n}\sum_{q=0}^{2n}\hh^{p,q}(\hilbn{n}{S},{L}_n)x^py^qt^n
  =
  \prod_{k\ge 1}\prod_{p=0}^2\prod_{q=0}^2\left( 1-(-1)^{p+q}x^{p+k-1}y^{q+k-1}t^k \right)^{-(-1)^{p+q}\hh^{p,q}(S,{L})}.
\end{equation}
The case of~${L}=\mathcal{O}_S$ thus describes the Hodge numbers,
and this is precisely the result of \cite{MR1219901}.

The case~${L}=\omega_S^{\otimes i}$ relates to the Hochschild--Serre cohomology of~$\hilbn{n}{S}$:
for~${L}=\omega_S^{\otimes i}$ where~$i\in\mathbb{Z}$
we get
\begin{equation}
  (\omega_S^{\otimes i})_n\cong\omega_{\hilbn{n}{S}}^{\otimes i},
\end{equation}
and thus the twisted Hodge numbers in \eqref{equation:twisted-hodge}
are the dimensions of different pieces of the Hochschild--Serre cohomology; see \eqref{equation:hochschild-serre-decomposition} and \eqref{equation:hkr-twisted-hodge}.

In \cite[Appendix~B]{MR3778120} a counterexample was found to the the conjecture,
using~$S$ an Enriques surface and~$n\geq 2$.
Alternatively, the computation in \cref{subsection:hilbert-square-P2} gives another explicit counterexample.

\begin{example}
  \label{example:counterexample-boissiere}
  The right-hand side of \eqref{equation:boissiere-conjecture} for~$\mathbb{P}^2$
  and~${L}=\omega_{\mathbb{P}^2}^\vee=\mathcal{O}_{\mathbb{P}^2}(3)$
  where we take only~$k=1,2$ into account becomes
  % sage: R.<x,y> = PowerSeriesRing(ZZ)
  % sage: S.<t> = PowerSeriesRing(R)
  % sage: f = (1/(1-t))^10 * (1/(1+x*t))^-8 * (1/(1-x^2*t)) * (1/(1-x*y*t^2))^10 * (1/(1+x^2*y*t^2))^-8 * (1/(1-x^3*y*t^2))
  % sage: print(f.truncate(3))
  \begin{equation}
    \begin{aligned}
      &\left( \frac{1}{1-t} \right)^{10}\left( \frac{1}{1+xt} \right)^{-8}\left( \frac{1}{1-x^2t} \right)
      \left( \frac{1}{1-xyt^2} \right)^{10}\left( \frac{1}{1+x^2yt^2} \right)^{-8}\left( \frac{1}{1-x^3yt^2} \right) \\
      &\qquad\equiv1 + (10 + 8x + x^2)t + (55 + 80x + 38x^2 + 10xy + 8x^3 + 8x^2y + x^4 + x^3y)t^2\bmod t^3.
    \end{aligned}
  \end{equation}
  Using \cref{proposition:hkr-hilbert-square-P2} we see that the coefficient of~$x^3yt^2$ should be~10 (not~1),
  the coefficient of~$x^2yt^2$ should be~35 (not~8)
  and the coefficient of~$xyt^2$ should be~28 (not~10).
\end{example}

\paragraph{A new conjecture}
Taking inspiration from our main result we propose a new conjecture,
which is compatible with collapsing the bigrading on twisted Hodge numbers
into a single grading on Hochschild homology with coefficients.

\begin{conjecture}
  \label{conjecture:corrected}
  Let $S$ be a smooth projective surface,
  and ${L}\in \Pic S$.
  Then
  \begin{equation}
    \label{equation:corrected}
    \begin{aligned}
      \sum_{n\geq 0}\sum_{p=0}^{2n}\sum_{q=0}^{2n}\hh^{p,q}(\hilbn{n}{S},{L}_n)x^py^qt^n
      &=
      \prod_{k\ge 1}\prod_{p=0}^2\prod_{q=0}^2\left( 1-(-1)^{p+q}x^{p+k-1}y^{q+k-1}t^k\right)^{-(-1)^{p+q}\hh^{p,q}(S,{L}^{\otimes k})}.
    \end{aligned}
  \end{equation}
\end{conjecture}

Note that the only difference to Boissi\`ere's conjecture \eqref{equation:twisted-hodge} is the occurence of the exponent $k$ of ${L}$ on the right.

By \cref{lemma:graded-sym-dim}, we see that \cref{conjecture:corrected} is equivalent to the claim that we have an isomorphism
\begin{equation}
  \bigoplus_{n\geq 0}\HH^{\#,\star}(\hilbn{n}{S}, {L}_n)t^n
 \overset{?}\cong
  \Sym^\bullet\left(\bigoplus_{k\geq 1}\Ho^{\#,\star}(S, {L}^{\otimes k})[1-k,1-k]t^k\right)\,,
\end{equation}
where $[1-k,1-k]$ denotes the shift of both gradings of $\Ho^{\#,\star}(S, {L}^{\otimes k})$ by the same value $1-k$, and the symmetric power is taken in the graded sense with respect to the bigrading of $\Ho^{\#,\star}(S, {L}^{\otimes k})$, but in the ordinary sense with respect to the grading by powers of $t$.
Looking at the individual Hilbert schemes, this can also be formulated as
\begin{equation}\label{eq:Boiconjsinglen}
\HH^{\#,\star}(\hilbn{n}{S}, {L}^{\{ n\}})
\overset{?}\cong\bigoplus_{\nu\dashv n}\left(\bigotimes_{i\geq 1} \Sym^{\lambda_i}\HH^{\#,\star}(S, {L}^{\otimes i})\right)\left[ \sum_{i}(1-i)\lambda_i,\sum_{i}(1-i)\lambda_i \right]\,
\end{equation}
for every $n\in \IN$, where $\nu=1^{\lambda_1}2^{\lambda_2}\cdots$.
Note that the right-hand side of \eqref{eq:Boiconjsinglen} almost exactly matches the right-hand side of
\eqref{equation:hodge-single-n}, the only difference being the degree shift. This difference can be fixed by introducing, for $\cX=[M/G]$ a quotient stack of a smooth proper variety $M$ by a finite group $G$, and $\cE\in \derived^\bounded(\cX)$, the \emph{twisted orbifold Hodge groups}
\begin{equation}\label{eq:Hodgeorbidef}
  \HH^{p,q}_{\orb}(\mathcal{X}, \mathcal{E})\coloneqq \left(\bigoplus_{g\in G} \HH^{p-\age(g),q-\age(g)}(M^g, E|_{M^g})\right)^G\cong \bigoplus_{[g]\in \Conj(G)} \HH^{p-\age(g),q-\age(g)}\left(M^g, E|_{M^g}\right)^{\centraliser(g)}
\end{equation}
The \emph{age function} $\age(-)$, which a priori takes value in $\mathbb{Q}$, is determined by the eigenvalues of the action of $g$ on the normal bundle $\mathrm{N}_{M^g/M}$; see \cite{MR1971293} for details.
Note that, $\HH^{\#,\star}_{\orb}(\mathcal{X}, \mathcal{E})$ and $\HH^{\#,\star}(\inertia\mathcal{X}, \mathcal{E})$ are equal as ungraded vector spaces. But in $\HH^{\#,\star}_{\orb}(\mathcal{X}, \mathcal{E})$ shifts for the varying direct summands of $\HH^{\#,\star}(\inertia\mathcal{X}, \mathcal{E})$ corresponding to the components of $\inertia\cX$ are introduced.
This is not just some random definition that we make in order to make certain degrees match, but the shift by $\age(g)$ also occurs in the standard convention for the grading of the orbifold cohomology; see \cite{MR1971293}.

For $M=S^n$, and $g\in \sym_n$ with cycle class $\lambda$, the age is given by
\begin{equation}
  \age(g)=\sum_{i\ge 1}(i-1)\lambda_i\,;
\end{equation}
see \cite{MR1971293,MR4033827}.
Hence, \eqref{equation:hodge-single-n} can be rephrased as
\begin{equation}
  \label{equation:orbifold-with-values-decomposition}
  \HH^{\#,\star}_{\orb}([\Sym^nX], {L}^{\{ n\}})
  \cong
  \bigoplus_{\nu\dashv n}\left(\bigotimes_{i\geq 1} \Sym^{\lambda_i}\HH^{\#,\star}(S, {L}^{\otimes i})\right)\left[ \sum_{i}(1-i)\lambda_i,\sum_{i}(1-i)\lambda_i \right].
\end{equation}
Thus, \cref{conjecture:corrected} is equivalent to the following conjecture that the derived McKay correspondence for Hilbert schemes of points on surfaces preserves the twisted Hodge groups.
\begin{conjecture}
  Let $S$ be a smooth projective surface and ${L}\in \Pic S$.
  Then, for every $n\geq 1$, we have an isomorphism of bigraded vector spaces
  \begin{equation}
    \HH^{\#,\star}(\hilbn{n}{S}, {L}_n)\cong \HH^{\#,\star}_{\orb}([\Sym^nS], {L}^{\{ n\}}).
  \end{equation}
\end{conjecture}

\paragraph{Some evidence}
We describe some evidence in favour of \cref{conjecture:corrected},
given by various specializations of the identity.
%Let the notation be as in \cref{conjecture:corrected}.

The first specialization reduces to Hochschild homology.
\begin{proposition}
  \Cref{conjecture:corrected} holds when specialising to $x=y^{-1}$.
\end{proposition}

\begin{proof}
  Indeed, due to the Hochschild--Kostant--Rosenberg isomorphism formulated in terms of twisted Hodge groups \eqref{equation:hkr-twisted-hodge},
  setting $x=s^{-1}$ and $y=s$ in \eqref{equation:corrected} recovers \eqref{equation:hilbert-hochschid-serre-line-bundle-series}.
\end{proof}

The second specialization reduces to the~$\chi_y$-genus.
Here, for a line bundle $L$ on a variety $X$,
let $\chi_{y}(X,L)$ denote Hirzebruch's $\chi_y$-genus $\sum_{p\geq 0} \chi(X, \Omega_X^p\otimes L)y^p$.
\begin{proposition}
  \Cref{conjecture:corrected} holds when specialising to $y=-1$.\\
   In other words, (renaming the variable $x$ by $-y$), we have
    \begin{equation}
    \label{eqn:Chi-y-genus-i}
    \sum_{n\geq 0}\chi_{-y}(\hilbn{n}{S},{L}_n)t^n
    =
    \prod_{k\ge 1}\prod_{p=0}^2\left( 1-y^{p+k-1}t^k\right)^{-(-1)^{p}\chi(S,\Omega_S^p\otimes {L}^{\otimes k})}.
  \end{equation}
\end{proposition}

\begin{proof}
  This is essentially due to G\"ottsche \cite{epiga:6830}.
  Note that in his notation, $\mu(L)=\det(L^{[n]})\otimes \det(\mathcal{O}_S^{[n]})^\vee$
  is precisely our $L_n$, thanks to the formula before \cite[Lemma 5.1]{MR1795551}.
  Indeed, by performing the change of variables $p=ty$ in G\"ottsche's \cite[Corollary 1.2]{epiga:6830}
  we have that $\sum_{n\geq 0}\chi_{-y}(\hilbn{n}{S},{L}_n)t^n$ is equal to
  \begin{equation}
    \label{eqn:Gottsche-Chi-y}
    \prod_{k\ge 1}
    \left(\frac{(1-t^ky^k)^2}{(1-t^ky^{k+1})(1-t^ky^{k-1})}\right)^{\frac{k^2}{2}(L^2)}
    \left(\frac{1-t^ky^{k-1}}{1-t^ky^{k+1}}\right)^{\frac{k}{2}(LK_S)}
    \left(1-t^ky^k\right)^{(K_S^2)}
    \left((1-t^ky^k)^{10}(1-t^ky^{k+1})(1-t^ky^{k-1})\right)^{-\chi(S,\mathcal{O}_S)}.
  \end{equation}
  It suffices to show the right-hand side of \eqref{eqn:Chi-y-genus-i} is equal to \eqref{eqn:Gottsche-Chi-y}.
  To this end, we apply the Hirzebruch--Riemann--Roch formula and Noether formula to obtain the following:
  \begin{equation}
  \label{eqn:RR-formula-Surface}
    \begin{aligned}
    \chi(S, L^{\otimes k})=\chi(S, \mathcal{O}_S)+\frac{k^2}{2}(L^2)-\frac{k}{2}(LK_S);\\
    \chi(S, \Omega^1_S\otimes L^{\otimes k})= -10 \chi(S, \mathcal{O}_S)+(K_S^2)+k^2(L^2);\\
    \chi(S, \Omega_S^2\otimes L^{\otimes k})= \chi(S, \mathcal{O}_S)+\frac{k^2}{2}(L^2)+\frac{k}{2}(LK_S).
    \end{aligned}
  \end{equation}
  Plugging \eqref{eqn:RR-formula-Surface} into the right-hand side of \eqref{eqn:Chi-y-genus-i} readily gives \eqref{eqn:Gottsche-Chi-y}.
\end{proof}

\begin{remark}
  From \eqref{eqn:Chi-y-genus-i}, we can deduce that the $\chi_y$-genus of the line bundle ${L}_n$ on $\hilbn{n}{S}$
  is determined by the $\chi_y$-genera of all the line bundles ${L}^{\otimes k}$ on $S$ for $k\in \mathbb{N}$.
  More precisely,
  we can re-express this identity as
  \begin{equation}
    \label{equation:ChiyGenus}
    \sum_{n\geq 0}\chi_{-y}(\hilbn{n}{S},{L}_n)t^n
    =
    \exp\left(\sum_{m=1}^{\infty}\frac{t^m}{m}\sum_{k=1}^{\infty}(ty)^{(k-1)m}\chi_{-y^m}(S, {L}^{\otimes k})\right).
  \end{equation}
  Compare this to \cite[Theorem 2.3.14 (4)]{MR1312161},
  which is the case~${L}=\mathcal{O}_S$.
\end{remark}

The third specialization is concerned with the cohomology of~$L_n$ on~$\hilbn{n}{S}$.
\begin{proposition}
  \Cref{conjecture:corrected} holds when specialising to $x=0$, which says that
  \begin{equation}
    \begin{aligned}
      \sum_{n\geq 0}\sum_{q=0}^{2n}\hh^{q}(\hilbn{n}{S},{L}_n)y^qt^n
      &=
      \prod_{q=0}^2\left( 1-(-1)^{q}y^{q}t\right)^{-(-1)^{q}\hh^{q}(S,{L})} \\
      &=
      \frac{(1+yt)^{\hh^1(S,{L})}}{(1-t)^{\hh^0(S,{L})}(1-y^2t)^{\hh^2(S,{L})}}.
    \end{aligned}
  \end{equation}
\end{proposition}

\begin{proof}
  Indeed, let $\rho\colon \hilbn{n}{S}\to S^{(n)}$ be the Hilbert--Chow morphism.
  The above equality follows from the isomorphism
  \begin{equation}
    \HH^*(\hilbn{n}{S}, {L}_n)
    \cong
    \HH^*(\hilbn{n}{S}, \rho^*{L}^{(n)})
    \cong
    \HH^*(S^{(n)}, \mathbf{R}\rho_*\rho^*{L}^{(n)})
    \cong
    \HH^*(S^{(n)}, {L}^{(n)})
    \cong
    \Sym^n\HH^*(S, {L}),
  \end{equation}
  where we used the fact that $S^{(n)}$ has rational singularities hence $\mathbf{R}\rho_*\mathcal{O}_{\hilbn{n}{S}}=\mathcal{O}_{S^{(n)}}$,
  and $\Sym^n$ is taken in the graded sense.
  %so in particular $\Sym^*\HH^i(S, {L})=\bigwedge^*\HH^i(S, {L})$ as vector spaces an odd number $i$.
\end{proof}

The final specialization is concerned with the global sections of~$\Omega_{\hilbn{n}{S}}^p\otimes L_n$.
\begin{proposition}
  \label{proposition:boissiere-global-sections}
  \Cref{conjecture:corrected} holds when specialising to $y=0$, which says that
  \begin{equation}
    \label{equation:boissiere-global-sections}
    \begin{aligned}
      \sum_{n\geq 0}\sum_{p=0}^{2n}\hh^{0}(\hilbn{n}{S}, \Omega^p_{\hilbn{n}{S}}\otimes {L}_n)x^pt^n
      &=
      \prod_{p=0}^2\left( 1-(-1)^{p}x^{p}t\right)^{-(-1)^{p}\hh^{p,0}(S,{L})} \\
      &=\frac{(1+xt)^{\hh^0(S,\Omega_S^1\otimes{L})}}{(1-t)^{\hh^0(S,{L})}(1-x^2t)^{\hh^0(S,\omega_S^\vee\otimes{L})}}
    \end{aligned}
  \end{equation}
\end{proposition}

\begin{proof}
  This is proved by Boissi\`ere in \cite[Proposition 1]{MR2932167}.
  The key point is that by \cite[Lemma 1.11]{MR0485870} we have an isomorphism
    $\rho_*\Omega^p_{\hilbn{n}{S}}\cong \widetilde{\Omega}^p_{S^{(n)}}$,
  where the right-hand side denotes the sheaf of reflexive $p$-differentials on $S^{(n)}$.
  Hence,
  \begin{equation}
    \HH^0(\hilbn{n}{S}, \Omega^p_{\hilbn{n}{S}}\otimes {L}_n)
    \cong
    \HH^0(S^{(n)}, \widetilde{\Omega}^p_{S^{(n)}}\otimes {L}^{(n)})
    \cong
    \HH^0(S^n, \Omega^p_{S^n}\otimes {L}^{\boxtimes n})^{\symgr_n},
  \end{equation}
  and the K\"unneth formula allows us to conclude.
\end{proof}

As a corollary, there is the following formula for polyvector fields on the Hilbert scheme.

\begin{corollary}\label{cor:polyvectorfields}
For every $n\in \mathbb N$, we have an isomorphism of graded vector spaces
\begin{equation}
 \HH^{0}\left(\hilbn{n}{S},\bigwedge^{2n-\#} \tangent_{\hilbn{n}{S}}\right)\cong \Sym^n\left(\LaTeXunderbrace{\HH^0(S,\omega_S^\vee)}_{\deg 0}\oplus \LaTeXunderbrace{\HH^0(S,\tangent_S)}_{\deg 1}\oplus \LaTeXunderbrace{\HH^0(S,\mathcal O_S)}_{\deg 2} \right)
\end{equation}
\end{corollary}

\begin{proof}
Translating from generating functions to vector spaces, \cref{proposition:boissiere-global-sections} says that the $\#=0$ case of \eqref{eq:Boiconjsinglen} is true. Due to the shifts on the right-hand side, the only contribution to $\#=0$ comes when $\lambda_i=0$ for all $i>1$, in which case $\lambda_1=n$. Hence, we have
\begin{equation}\label{eq:Boiconjsinglen0}
 \HH^{0}\left(\hilbn{n}{S},\Omega^{\#}_{\hilbn{n}{S}}\otimes L_n\right)\cong \Sym^n\left(\LaTeXunderbrace{\HH^0(S,L)}_{\deg 0}\oplus \LaTeXunderbrace{\HH^0(S,\Omega_S^1\otimes L)}_{\deg 1}\oplus \LaTeXunderbrace{\HH^0(S,\omega_S\otimes L)}_{\deg 2}\right)
\end{equation}
The assertion is just the case $L=\omega_S^\vee$ of \eqref{eq:Boiconjsinglen0}.
\end{proof}

The results of \cref{subsection:hilbert-square-P2} give an explicit instance in which the conjecture is checked.
\begin{example}
  The case~$S=\mathbb{P}^2$, $n=2$, and ${L}=\omega_S^\vee$ can be compared with the results in \cref{subsection:hilbert-square-P2}:
  we have that
  \begin{equation}
    (\mathrm{h}^{p,q}(\mathbb{P}^2,\omega_S^\vee))_{p,q}
    =
    \begin{pmatrix}
      10 & 0 & 0 \\
      8 & 0 & 0 \\
      1 & 0 & 0
    \end{pmatrix}
  \end{equation}
  and
  \begin{equation}
    (\mathrm{h}^{p,q}(\mathbb{P}^2,\omega_S^\vee\otimes\omega_S^\vee))_{p,q}
    =
    \begin{pmatrix}
      28 & 0 & 0 \\
      35 & 0 & 0 \\
      10 & 0 & 0
    \end{pmatrix},
  \end{equation}
  and a computation shows the output of \cref{conjecture:corrected}
  agrees with \cref{proposition:hkr-hilbert-square-P2}.
\end{example}

\begin{remark}
  By taking ${L}=\mathcal{O}_S$
  (hence ${L}_n=\mathcal{O}_{\hilbn{n}{S}}$),
  \cref{conjecture:corrected} reduces to the classical formula of Hodge numbers of Hilbert schemes
  due to G\"ottsche and Soergel \cite{MR1219901}.
  More generally, by using \cite[Remark~2.7]{MR2578804},
  we will see in \cref{theorem:TwistHodgeHilb} that \cref{conjecture:corrected}
  holds in the case that ${L}$ admits a unitary flat connection.
\end{remark}

%\begin{remark}
%  Another consequence to be tested is as follows: by specialising to $x=y=s$, we have
%  \begin{equation}
%    \label{equation:Betti}
%    \sum_{n\geq 0}\sum_{i=0}^{4n}\sum_{p+q=i}\hh^{p,q}(\hilbn{n}{S},{L}_n)s^it^n
%    =
%    \prod_{k\ge 1}\prod_{i=0}^4\left( 1-(-1)^{i}s^{i+2k-2}t^k\right)^{-(-1)^{i}\sum_{p+q=i}\hh^{p,q}(S,{L}^k)}.
%  \end{equation}
%\end{remark}
%
%As far as we are aware, \eqref{equation:ChiyGenus} and \eqref{equation:Betti}
%are both open.

\subsection{Relation to Nieper--Wi\ss kirchen's work}
\label{subsection:nieper-wisskirchen}
Recall that for a smooth projective variety $X$
equipped with a rank-1 local system $\mathbb{L}$ of complex vector spaces,
the cohomology group $\HH^k(X, \mathbb{L})$ carries a weight-$k$ $\mathbb{C}$-Hodge structure\footnote{A $\mathbb{C}$-Hodge structure of weight $k$ is nothing but a $\mathbb{C}$-vector space together with a direct sum decomposition into subspaces $V=\bigoplus_{p+q=k}V^{p,q}$.}.
The Hodge components are described as follows.
Under the non-abelian Hodge correspondence,
$\mathbb{L}$ corresponds to a Higgs line bundle $({L}, \theta)$,
where ${L}$ is a holomorphic line bundle of degree 0
and $\theta\in \HH^0(X, \Omega^1_X)$ is a holomorphic 1-form.
Then the Hodge decomposition takes the following form:
\begin{equation}
  \label{eqn:HodgeDecomp}
  \HH^k(X, \mathbb{L})\cong \bigoplus_{p+q=k} \HH^{p,q}(X, ({L}, \theta)),
\end{equation}
where $\HH^{p,q}(X, ({L}, \theta))$ is by definition
the cohomology of the complex
\begin{equation}
  \HH^q(X, {L}\otimes \Omega_X^{p-1})
  \xrightarrow{\theta}
  \HH^q(X, {L}\otimes \Omega_X^p)
  \xrightarrow{\theta}
  \HH^q(X, {L}\otimes \Omega_X^{p+1}).
\end{equation}
Therefore, we denote
\begin{equation}
  \HH^{p,q}(X, \mathbb{L})\colonequals\HH^{p,q}(X, ({L}, \theta)).
\end{equation}
Again by the non-abelian Hodge correspondence, $\mathbb{L}$ is unitary if and only if ${L}\cong \mathbb{L}\otimes \mathcal{O}_X$ and $\theta=0$. In this case, \eqref{eqn:HodgeDecomp} becomes
\begin{equation}
  \HH^k(X, \mathbb{L})\cong \bigoplus_{p+q=k} \HH^{p,q}(X, {L}),
\end{equation}
where $\HH^{p,q}(X, {L})=\HH^q(X, {L}\otimes \Omega_X^p)$,
consistent with our notation in previous sections.

In the rest of this section, we specialise to the case of dimension 2.
Let $S$ be a smooth projective surface and $\mathbb{L}$ a rank-1 local system of $\mathbb{C}$-vector spaces on $S$.

As $\pi_1(\hilbn{n}{S})\cong \pi_1(S^{(n)})\cong \pi_1(S)^{\operatorname{ab}}\cong\HH_1(S, \mathbb{Z})$
(see \cite[\S 6, Lemma 1]{MR0730926}),
the local system $\mathbb{L}$
induces a rank-1 local system on $\hilbn{n}{S}$, denoted by $\mathbb{L}_n$.
Similarly,
any line bundle ${L}$
induces a line bundle ${L}_n$ on $\hilbn{n}{S}$,
and we have a canonical identification $\HH^{1,0}(S)\cong\HH^{1,0}(\hilbn{n}{S})$. One can easily check the following compatibility:

\begin{lemma}
  \label{lemma:TakingHilb}
  Let $\mathbb{L}$ be a rank-1 local system on $S$ corresponding to the Higgs line bundle $({L}, \theta)$.
  Then the local system $\mathbb{L}_n$ on $\hilbn{n}{S}$ corresponds to the Higgs line bundle $({L}_n, \theta)$.
  In particular, if $\mathbb{L}$ is unitary, then so is $\mathbb{L}_n$
  and ${L}_n\cong \mathbb{L}_n\otimes \mathcal{O}_{\hilbn{n}{S}}$, $\theta=0$.
\end{lemma}

In \cite[Theorem 1.2, Remark 2.7]{MR2578804}, Nieper--Wi\ss kirchen proved the following result.
\begin{theorem}[Nieper--Wi\ss kirchen]
  \label{theorem:Nieper}
  In the above notation,
  \begin{equation}
    \label{equation:Nieper-Hodge}
    \sum_{n\geq 0}\sum_{p=0}^{2n}\sum_{q=0}^{2n}\hh^{p,q}(\hilbn{n}{S},\mathbb{L}_n)x^py^qt^n
    =
    \prod_{k\ge 1}\prod_{p=0}^2\prod_{q=0}^2\left( 1-(-1)^{p+q}x^{p+k-1}y^{q+k-1}t^k\right)^{-(-1)^{p+q}\hh^{p,q}(S,\mathbb{L}^{\otimes k})}.
  \end{equation}
  In particular, by specialising to $x=y$, we have
  \begin{equation}
    \label{equation:Nieper-Betti}
    \sum_{n\geq 0}\sum_{i=0}^{4n}\hh^{i}(\hilbn{n}{S},\mathbb{L}_n)x^it^n
    =
    \prod_{k\ge 1}\prod_{i=0}^4\left( 1-(-1)^{i}x^{i+2k-2}t^k\right)^{-(-1)^{i}\hh^{i}(S,\mathbb{L}^{\otimes k})}.
  \end{equation}
\end{theorem}

Combining \cref{lemma:TakingHilb} and \cref{theorem:Nieper}, we obtain that \Cref{conjecture:corrected} holds for line bundles arising from rank-1 unitary local systems:

\begin{theorem}
  \label{theorem:TwistHodgeHilb}
  Let $\mathbb{L}$ be a rank-1 unitary local system on a smooth projective surface $S$.
  Define ${L}\colonequals\mathbb{L}\otimes \mathcal{O}_S$.
  Let ${L}_n$ be the associated line bundle on $\hilbn{n}{S}$.
  Then we have the following equality.
  \begin{equation}
    \label{equation:Hodge}
    \sum_{n\geq 0}\sum_{p=0}^{2n}\sum_{q=0}^{2n}\hh^{p,q}(\hilbn{n}{S},{L}_n)x^py^qt^n
    =
    \prod_{k\ge 1}\prod_{p=0}^2\prod_{q=0}^2\left( 1-(-1)^{p+q}x^{p+k-1}y^{q+k-1}t^k\right)^{-(-1)^{p+q}\hh^{p,q}(S,{L}^{\otimes k})}.
  \end{equation}
\end{theorem}

%\begin{remark}
%  \fixthis{can Lie or Andreas clarify this remark a bit, so that its purpose is clearer? thanks!}
%  It is very interesting to compare \cref{theorem:TwistHodgeHilb}
%  to \cref{theorem:main}
%  (or rather with \cref{cor:GeneratingSeriesHodgeNumber}).
%  In the situation of \cref{theorem:TwistHodgeHilb},
%  we deduce that for any integers $p, q, n$, we have the following equality of Hodge numbers:
%  \begin{equation}
%    \hh^{p,q}_{\orb}([\Sym^nS], {L}^{\{n\}})=\hh^{p,q}(\hilbn{n}{S},{L}_n).
%  \end{equation}
%  %However, note that under the McKay correspondence $\derived^\bounded([\Sym^nS])\cong \derived^\bounded(\hilbn{n}{S})$, the vector bundle $\Omega^p_{\hilbn{n}{S}}$ does not corresponds to
%\end{remark}

\appendix

\section{Hochschild--Serre cohomology for dg categories and functorialities}
\label{section:functoriality}
The definition of Hochschild--Serre cohomology in \cref{definition:hochschild-serre}
using Fourier--Mukai functors
has a counterpart using dg~bimodules as alluded to in \cref{remark:dg-generalisation}.
In \cref{appendix:hochschild-serre-dg} we will describe this,
and use it to show that it is a derived invariant.
We will also consider more generally Hochschild homology with coefficients,
and discuss how it is a derived invariant under a certain natural compatibility
between the equivalence and the coefficients.

Using the Fourier--Mukai definition,
we will point out in \cref{section:etale-functoriality}
how Hochschild--Serre cohomology moreover satisfy functorialities for  \'etale morphisms.
\subsection{Hochschild--Serre cohomology for dg categories and Morita invariance}
\label{appendix:hochschild-serre-dg}
Throughout we fix a base field $\field$.
Let~$\cA$ be a dg category,
and~$\Perf\cA$ be the subcategory of~$\derived(\cA)$
consisting of perfect right~$\cA$-modules.
We let~$\cA^\env$ denote its enveloping category,
defined as~$\cA\otimes_k\cA^\opp$.
For an introduction to the Morita theory of dg categories
in the setup that we will use,
we refer to \cite[\S2]{MR3911626}.

% this is not the original reference, but it is convenient
As in \cite[\S2.3]{MR3911626},
if~$\cA$ is smooth and proper, we will consider
\begin{itemize}
  \item the diagonal bimodule~$\mathcal{A}$
    as an object of~$\Perf\cA^\env$,
    representing the identity functor;
  \item the right dual~$\mathcal{A}^*$
    as an object of~$\Perf\cA^\env$,
    representing the Serre functor;
  \item the left dual~$\mathcal{A}^!$
    as an object of~$\Perf\cA^\env$,
    representing the inverse Serre functor.
\end{itemize}
This allows us to generalise \cref{definition:hochschild-serre} as follows,
where we use the tensor product of dg~bimodules
to encode iterated (inverse) powers of the Serre functor.
\begin{definition}
  \label{definition:hochschild-serre-dg}
  Let~$\cA$ be a smooth and proper dg~category,
  and let~$\serre_\cA$ denote its Serre functor.
  The \emph{Hochschild--Serre cohomology of~$\cA$} is
  the bigraded algebra
  \begin{equation}
    \hochserre_\bullet^*(\cA)
    \colonequals
    \bigoplus_{k\in\mathbb{Z}}\hochserre_k^*(\cA)
    =
    \bigoplus_{k\in\mathbb{Z}}\bigoplus_{j\in\mathbb{Z}}\hochserre_k^j(\cA)
  \end{equation}
  where
  \begin{equation}
    \hochserre_k^j(\cA)
    \colonequals
    \HH^j\RHom_{\cA^\env}(\identity_\cA,\serre_\cA^{\circ k})
    =
    \begin{cases}
      \HH^j\RHom_{\cA^\env}(\cA,\cA^{*,\otimes^\LLL k})
      & k\geq 0 \\
      \HH^j\RHom_{\cA^\env}(\cA,\cA^{!,\otimes^\LLL -k})
      & k<0. \\
    \end{cases}
  \end{equation}
  and the multiplication is induced by the composition in~$\Perf\cA^\env$.
\end{definition}
As in \cref{remark:hochschild-serre-generalises-classical},
this definition incorporates the Hochschild cohomology (resp.~homology)
of the dg category~$\cA$,
as
\begin{equation}
  \hochschild^*(\cA)\cong\hochserre_0^*(\cA)
\end{equation}
resp.
\begin{equation}
  \hochschild_*(\cA)\cong\hochserre_1^*(\cA).
\end{equation}

\begin{remark}
  The definition in \cref{definition:hochschild-serre-dg} suffices for our purposes,
  we do not need to give a chain-level definition
  via a generalization of the Hochschild (co)chain complex.
  To study certain algebraic and higher structures
  which exist on Hochschild (co)homology,
  and might possess a generalization to Hochschild--Serre cohomology,
  this could however be useful.

  Moreover, both Hochschild homology and cohomology can be defined for arbitrary dg~categories,
  not just smooth and proper ones.
  It would be interesting to find a definition of Hochschild--Serre cohomology
  which works in this generality.
\end{remark}

\paragraph{K\"unneth formula}
Let $\cA$ and $\cB$ be smooth proper dg categories.
We have a canonical equivalence
\begin{equation}
	(\cA\otimes \cB)^{\env}\cong \cA^{\env}\otimes \cB^{\env},
\end{equation}
under which the bimodule $\cA\otimes \cB$ is identified with $\cA\boxtimes \cB$,
and similarly $(\cA\otimes \cB)^*$  with $\cA^*\boxtimes \cB^*$,
and $(\cA\otimes \cB)^!$ with $\cA^!\boxtimes \cB^!$.
In other words, as dg endofunctors of $\cA\otimes \cB$,
\begin{equation}
  \begin{aligned}
  	\identity_{\cA\otimes \cB}&\cong\identity_{\cA}\otimes \identity_{\cB},\\
  	\serre_{\cA\otimes \cB}&\cong\serre_{\cA}\otimes \serre_{\cB}.\\
  \end{aligned}
\end{equation}
As a consequence, we have the K\"unneth formula for Hochschild--Serre cohomology:
\begin{proposition}
	Let $\cA$ and $\cB$ be smooth proper dg  categories. We have an isomorphism of bigraded algebras:
	\begin{equation}
		\hochserre_\bullet^*(\cA\otimes \cB)\cong \hochserre_\bullet^*(\cA)\otimes  \hochserre_\bullet^*(\cB).
	\end{equation}
\end{proposition}

\paragraph{Morita invariance}
The following theorem explains that, as generalisation of Hochschild (co)homology, Hochschild--Serre cohomology
is still a Morita invariant for dg categories.
This is a generalisation of \cite[Theorem~2.1.8]{MR1998775}
which proves it for derived categories of smooth projective varieties.

\begin{theorem}[Morita invariance]
  \label{theorem:hochschild-serre-morita-invariant}
  Let $F\colon\cA\to\cB$ be a Morita equivalence between smooth and proper dg~categories.
  Then there is a naturally induced isomorphism of bigraded algebras
  \begin{equation}
    \label{equation:hochschild-serre-isomorphism}
    \hochserre_\bullet^*(\cA)\cong \hochserre_\bullet^*(\cB).
  \end{equation}
\end{theorem}

\begin{proof}
  % induced equivalence
  The functor~$F$ induces a functor~$-\otimes_\cA^\LLL M\colon\Perf\cA\to\Perf\cB$,
  with quasi-inverse~$G$ given as~$-\otimes_\cB^\LLL N$,
  where the bimodules~$M$ and~$N$ are perfect.
  Then the Morita equivalence~$F$ induces a Morita equivalence
  \begin{equation}
    F^\env\colon\Perf\cA^\env\to\Perf\cB^\env
  \end{equation}
  which we can write using bimodules as~$N\otimes_\cA^\LLL-\otimes_\cA^\LLL M$.

  % preserves identity, and (inverse) Serre functor
  By the Morita theory for dg~categories
  this functor always preserves the identity functor.
  The categories~$\cA^\env$ and~$\cB^\env$ are again smooth and proper,
  thus the equivalence~$F^\env$ is compatible with duality,
  and therefore sends the bimodules~$\cA^*$ (resp.~$\cA^!$)
  representing the Serre functor (resp.~inverse Serre functor)
  to~$\cB^*$ (resp.~$\cB^!$).
  % vector space isomorphism
  Therefore,
  we obtain an isomorphism of vector spaces
  \begin{equation}
    \hochserre_k^j(\cA)\overset{\sim}{\to}\hochserre_k^j(\cB),
  \end{equation}
  and taking the direct sum we obtain
  the isomorphism of vector spaces in~\eqref{equation:hochschild-serre-isomorphism}.

  % algebra structure
  The isomorphism of bigraded vector spaces is compatible with
  the algebra structure,
  as this is induced from the compositions
  in the equivalent categories~$\derived^\bounded(\cA^\env)$ and~$\derived^\bounded(\cB^\env)$.
\end{proof}

\paragraph{Agreement between the two approaches}
We wish to show how the dg version from \cref{definition:hochschild-serre-dg}
agrees with the Fourier--Mukai approach \cref{definition:hochschild-serre}
for smooth and proper orbifolds.
The case of Hochschild cohomology was already discussed in \cite[Appendix~A]{MR4057490},
and the method is very similar.

\begin{theorem}[Agreement]
  \label{theorem:agreement}
  Let~$\cX$ be a smooth and proper orbifold,
  and let~$\derived^\bounded(\cX)$ be a dg enhancement of
  its bounded derived category of coherent sheaves.
  Then \cref{definition:hochschild-serre-dg} for~$\cA=\derived^\bounded(\cX)$
  and \cref{definition:hochschild-serre}
  are isomorphic as bigraded algebras.
\end{theorem}

Before giving the proof we first explain why we can apply \cref{definition:hochschild-serre-dg}.
\begin{itemize}
  \item
    To consider~$\derived^\bounded(\cX)$ as a dg category,
    we can use the enhancement from \cite[Example~5.5]{MR3573964}.
    By \cite[Proposition~6.10]{MR3861804}
    the derived category has a unique dg enhancement,
    i.e.,
    all dg enhancements are quasi-equivalent.
  \item
    In \cite[Theorem~6.6]{MR3573964} it is shown that the derived category
    of a smooth and proper orbifold
    is an admissible subcategory in the derived category of
    a smooth and proper variety.
    As smoothness and properness for dg categories is inherited by admissible subcategories,
    see, e.g., \cite[Proposition~5.20]{MR3573964}
    we obtain that any dg enhancement of~$\derived^\bounded(\cX)$
    is indeed a smooth and proper dg~category.
\end{itemize}

\begin{proof}[Proof of \cref{theorem:agreement}]
  By the Morita invariance from \cref{theorem:hochschild-serre-morita-invariant}
  and uniqueness of the dg~enhancement
  we can ignore the choice of enhancement.
  By \cite[Theorem~1.2]{MR2669705}
  we have that~$\derived^\bounded(\cX)^\env\cong\derived^\bounded(\cX\times\cX)$.
  Next, observe that
  \begin{itemize}
    \item the diagonal bimodule
      corresponds to the identity functor~$\Delta_*\mathcal{O}_\cX$;
    \item the Serre functor
      corresponds to~$\Delta_*\omega_\cX[d_X]$;
    \item the inverse Serre functor
      corresponds to~$\Delta_*\omega_\cX^\vee[-d_X]$.
  \end{itemize}
  The correspondence for the Serre functor follows
  from,
  e.g., \cite[Proposition~2.31]{0811.1955v2} when~$\field$ is algebraically closed,
  or \cite[Theorem~1]{MR4514209} when~$\field$ is arbitrary
  and~$\cX$ has projective coarse moduli space.
  Thus the definitions in \cref{definition:hochschild-serre-dg}
  and \cref{definition:hochschild-serre} agree,
  because the bigraded algebra structures correspond to
  composition in the derived category~$\derived^\bounded(\cX\times\cX)$.
\end{proof}

We record the following corollary of the Bridgeland--King--Reid--Haiman equivalence \eqref{equation:derived-mckay},
which motivates the approach taken in the main body of the paper.
\begin{corollary}
  \label{corollary:hochschild-serre-comparison-hilbert-scheme}
  Let~$S$ be a smooth, projective surface.
  Then for all~$n\geq 0$ there exists an isomorphism
  \begin{equation}
    \hochserre_\bullet^*(\hilbn{n}{S})
    \cong
    \hochserre_\bullet^*([\Sym^nS])
  \end{equation}
  of bigraded algebras.
\end{corollary}

\begin{proof}
  The (necessarily unique) dg enhancements of~$\derived^\bounded(\hilbn{n}{S})$
  and~$\derived^\bounded([\Sym^nS])$
  are unique,
  so by the agreement of the definitions in \cref{theorem:agreement} we are done.
\end{proof}
%In \addreference we generalise this to a version which allows coefficients in line bundles.

\subsection{\'Etale functoriality of Hochschild--Serre cohomology}
\label{section:etale-functoriality}
The original definition of Hochschild--Serre cohomology of smooth projective varieties
was only shown to be functorial for equivalences.
Hochschild homology on the other hand is functorial for arbitrary morphisms,
and even functors \cite{MR1667558}.
As suggested in \cite[Claim in \S8.4]{MR2062626},
Hochschild cohomology should be functorial for \'etale morphisms.

In this section we show that Hochschild--Serre cohomology (and thus Hochschild cohomology)
is indeed functorial for \'etale morphisms,
at least as vector spaces.
This takes on two forms:
\begin{itemize}
  \item a covariant functoriality,
    for which we provide a more general criterion in \cref{prop:HA-Pushforward}
    leading to an \emph{\'etale pushforward} in \cref{corollary:etale-pushforward};
  \item a contravariant functoriality in \cref{proposition:etale-pullback},
    i.e., an \emph{\'etale pullback}.
\end{itemize}
Already for Hochschild cohomology we are not aware of a written reference where this is proven,
even on the level of vector spaces.
We are content with showing that there is a naturally induced morphism for Hochschild--Serre cohomology,
we do not work out the compatibility with composition.
We also do not address whether Hochschild--Serre cohomology satisfies
some functoriality with respect to fully faithful functors,
which is discussed for Hochschild cohomology in \cite{keller-dih}.

\begin{proposition}[Covariant functoriality]
  \label{prop:HA-Pushforward}
  Let~$f\colon X\to Y$ be a morphism between smooth projective varieties.
  Let~$\omega_f\colonequals\omega_X\otimes f^*\omega_Y^\vee$ be the relative canonical bundle,
  and let~$d_f\colonequals\dim(X)-\dim(Y)$ be the relative dimension.

  For integers $k, m\in \mathbb{Z}$
  and an element~$\sigma\in \HH^{m+(1-k)d_f}(X, \omega_f^{\otimes (1-k)})$,
  there is a natural morphism for any $i\in \mathbb{Z}$:
  \begin{equation}
    f^\sigma_*\colon\hochserre^i_k(X)\to \hochserre^{i+m}_k(Y).
  \end{equation}
\end{proposition}

\begin{proof}
  First note that we have canonical isomorphisms:
  \begin{equation}
    \begin{aligned}
      \Hom_Y(\mathbf{R}f_*(\omega_X[d_X])^{\otimes k}, (\omega_Y[d_Y])^{\otimes k}[m])
      &\cong\Hom_X(\omega_X^{\otimes k}, f^!\omega_Y^{\otimes k}[m-kd_f]) \\
      &\cong\Hom_X(\omega_X^{\otimes k}, f^*\omega_Y^{\otimes k}\otimes \omega_f[d_f+m-kd_f]) \\
      &\cong\HH^{m}(X, (\omega_f[d_f])^{\otimes (1-k)}).
    \end{aligned}
  \end{equation}
  Hence we can view the element $\sigma$ as a morphism in $\derived^\bounded(Y)$
  \begin{equation}
    \label{eqn:sigma}
    \sigma\colon \mathbf{R}f_*(\omega_X[d_X])^{\otimes k}\to (\omega_Y[d_Y])^{\otimes k}[m].
  \end{equation}
  Now given an element $\alpha\in \hochserre_k^i(X)$, viewed as a morphism in $\derived^\bounded(X\times X)$,
  \begin{equation}
    \alpha\colon \Delta_{X, *}\mathcal{O}_X \to \Delta_{X, *}(\omega_X[d_X])^{\otimes k}[i],
  \end{equation}
  we apply the derived pushforward along~$f\times f\colon X\times X\to Y\times Y$ to get
  \begin{equation}
    \mathbf{R}(f\times f)_*\Delta_{X, *}\mathcal{O}_X \xrightarrow{\mathbf{R}(f\times f)_*(\alpha)}\mathbf{R}(f\times f)_* \Delta_{X, *}(\omega_X[d_X])^{\otimes k}[i].
  \end{equation}
  Since $(f\times f)\circ \Delta_{X}=\Delta_Y\circ f$, we have by functoriality
  \begin{equation}
    \Delta_{Y, *}\mathbf{R}f_*\mathcal{O}_X \xrightarrow{\mathbf{R}(f\times f)_*(\alpha)}\Delta_{Y, *}\mathbf{R}f_*(\omega_X[d_X])^{\otimes k}[i].
  \end{equation}
  Using the natural map $\mathcal{O}_Y\to \mathbf{R}f_*\mathcal{O}_X$ and the  morphism \eqref{eqn:sigma} we get a morphism
  \begin{equation}
    \Delta_{Y,*}\mathcal{O}_Y\to\Delta_{Y, *}\mathbf{R}f_*\mathcal{O}_X \xrightarrow{\mathbf{R}(f\times f)_*(\alpha)}\Delta_{Y, *}\mathbf{R}f_*(\omega_X[d_X])^{\otimes k}[i]\xrightarrow{\Delta_{Y, *}(\sigma)} \Delta_{Y, *}(\omega_Y[d_Y])^{\otimes k}[i+m],
  \end{equation}
  which can be viewed as an element in~$\hochserre^{i+m}_k(Y)$, and defined to be the image of~$\alpha$.
  It is clear that the obtained map~$\hochserre^i_k(X)\to \hochserre^{i+m}_k(Y)$ is linear.
\end{proof}

In the above proposition, the case $k=1$ and $m=0$ (with $\sigma$ the canonical element of $\HH^0(X,\mathcal O_X)$ corresponding to the constant function 1)
should recover the covariant functoriality of Hochschild homology for a morphism.
Another interesting instance is the case where $f$ is \'etale.
\begin{corollary}[\'Etale pushforward]
  \label{corollary:etale-pushforward}
  An \'etale morphism $f\colon X\to Y$ between smooth projective varieties induces a natural morphism
  \begin{equation}
    f_*\colon \hochserre^i_k(X)\to \hochserre^i_k(Y).
  \end{equation}
\end{corollary}

\begin{proof}
  Take $m=0$. It suffices note that $d_f=0$ and $\omega_f\cong\mathcal{O}_X$ in this case.
  Here $\sigma$ is taken to be the canonical element corresponding to the constant function 1.
\end{proof}

\begin{remark}
  Using the Hochschild--Kostant--Rosenberg decomposition for Hochschild--Serre cohomology \eqref{equation:hochschild-serre-decomposition},
  one can get an a priori different covariant functoriality for Hochschild--Serre cohomology.
  Let the assumption be as in \cref{prop:HA-Pushforward}.
  Then for each $p, q\in \mathbb{Z}$, we have natural morphisms
  \begin{equation}
    \begin{aligned}
      \HH^p(X, \bigwedge^q\tangent_X\otimes \omega_X^{\otimes k})
      &\cong\HH^p(Y, \mathbf{R}f_*(\bigwedge^q\tangent_X\otimes \omega_X^{\otimes k})) \\
      &\to\HH^p(Y, \mathbf{R}f_*(f^*\bigwedge^q\tangent_Y\otimes \omega_X^{\otimes k})) \\
      &\cong \HH^p(Y, \bigwedge^q\tangent_Y\otimes \mathbf{R}f_*(\omega_X^{\otimes k})) \\
      &\to\HH^p(Y, \bigwedge^q\tangent_Y\otimes \omega_Y^{\otimes k}[m-kd_f]) \\
      &=\HH^{p+m-kd_f}(Y, \bigwedge^q\tangent_Y\otimes \omega_Y^{\otimes k}),
    \end{aligned}
  \end{equation}
  where the last morphism uses the input datum $\sigma\in  \HH^{m}(X, (\omega_f[d_f])^{\otimes (1-k)})$ as in \eqref{eqn:sigma}.
  Taking the direct sum over all $p$, $q$ with $p+q=i+kd_X$,  or equivalently $p+q+m-kd_f=i+m+kd_Y$, by \eqref{equation:hochschild-serre-decomposition},
  we get a morphism
  \begin{equation}
    {}'f_*^\sigma\colon \hochserre^i_k(X)\to \hochserre^{i+m}_k(Y).
  \end{equation}
  It is an interesting question to compare ${}'f_*^\sigma$ with $f_*^{\sigma}$,
  or rather,
  in case they are different, to find a suitable modified isomorphism as in \eqref{equation:hochschild-serre-decomposition} to make them equal.
\end{remark}

As for contravariant functoriality, we have the following.
Observe that strictly speaking we prove \'etale functoriality only
\emph{after} the Hochschild--Kostant--Rosenberg decomposition
for Hochschild--Serre cohomology,
but this suffices to get contravariant \'etale functoriality on the level of vector spaces.

\begin{proposition}[\'Etale pullback]
  \label{proposition:etale-pullback}
  Let $f\colon X\to Y$ be an \'etale morphism between smooth projective varieties.
  Then $f$ induces a natural morphism
  \begin{equation}
    f^*\colon \hochserre_k^i(Y)\to \hochserre_k^i(X).
  \end{equation}
\end{proposition}

\begin{proof}
  For any $p, q\in \mathbb{Z}$, we have the following natural morphisms:
  \begin{equation}
    \begin{aligned}
      \HH^p(Y, \bigwedge^q\tangent_Y\otimes \omega_Y^{\otimes k})
      &\to\HH^p(Y, \mathbf{R}f_*\mathcal{O}_X\otimes \bigwedge^q\tangent_Y\otimes \omega_Y^{\otimes k}) \\
      &\cong\HH^p(Y, \mathbf{R}f_*f^*(\bigwedge^q\tangent_Y\otimes \omega_Y^{\otimes k})) \\
      &\cong\HH^p(X,f^*(\bigwedge^q\tangent_Y\otimes \omega_Y^{\otimes k})) \\
      &\cong\HH^p(X, \bigwedge^q\tangent_X\otimes \omega_X^{\otimes k}),
    \end{aligned}
  \end{equation}
  where the first morphism is induced by the natural map $\mathcal{O}_Y\to \mathbf{R}f_*\mathcal{O}_X$,
  and the last isomorphism uses that $f$ is \'etale
  hence induces isomorphisms $f^*\tangent_Y\cong\tangent_X$ and $f^*\omega_Y\cong\omega_X$.

  Taking the direct sum over $p, q$ with $p+q=i+kd_X=i+kd_Y$, and use \eqref{equation:hochschild-serre-decomposition}, we get the desired morphism $f^*$.
\end{proof}

% to make the ToC fit on one page
\addtocontents{toc}{\protect\enlargethispage{\baselineskip}}
\section{Computations for \texorpdfstring{\cref{subsection:hilbert-square-P2}}{Subsection~\ref{subsection:hilbert-square-P2}}}
\label{appendix:hilbert-square-P2}
In this appendix we collect the computations for \cref{proposition:hkr-hilbert-square-P2,remark:rigidity}.
We will denote~$\mathbb{P}=\mathbb{P}^n=\mathbb{P}(V)$
where~$V$ is an~$(n+1)$~\dash dimensional vector space~$V$.
We write
\begin{equation}
  \begin{aligned}
    G&\colonequals\operatorname{Gr}(2,V) \\
    H&\colonequals\hilbn{2}{\mathbb{P}^n},
  \end{aligned}
\end{equation}
where the former comes equipped with the tautological sub- and quotient bundles~$\mathcal{S}$ and~$\mathcal{Q}$.

We are mostly interested in the case~$n=2$ (for the proof of \cref{proposition:hkr-hilbert-square-P2}),
where we have~$G=\mathbb{P}^{2,\vee}$,
and~$\mathcal{S}\cong\Omega_G^1(1)$ resp.~$\mathcal{Q}\cong\mathcal{O}_G(1)$.
The methods in this section can be used more generally
to compute (pieces of) the Hochschild--Kostant--Rosenberg decomposition
of~$\hochschild^*(\hilbn{2}{\mathbb{P}^n})$, but we will not work out all the details,
and we are content with establishing the rigidty result from \cref{remark:rigidity} for~$n\geq 3$.

Consider the morphism
\begin{equation}
  \label{equation:projective-bundle}
  \pi\colon H\to G
\end{equation}
obtained by sending 2 points, possibly infinitesimally near, to the line they span.
The following lemma (for~$n$ arbitrary) is probably well-known,
and a weaker version was already used in \cite[\S5]{MR4155174}.
For~$n=2$ it also appears in \cite[\S3.2]{MR3488782}.
% Lemma 7 in the notes
\begin{lemma}
  \label{lemma:projective-bundle}
  The morphism~$\pi$ in \eqref{equation:projective-bundle} is identified with
  the~$\mathbb{P}^2$-bundle~$\mathbb{P}_G(\Sym^2\mathcal{S})\to G$.
\end{lemma}

\begin{proof}
  We can write~$H$ as~$(\mathbb{P}_G(\mathcal{S})\times_G\mathbb{P}_G(\mathcal{S}))/{(\mathbb{Z}/2\mathbb{Z})}$.
  We can rewrite this as
  \begin{equation}
    \begin{aligned}
      (\mathbb{P}_G(\mathcal{S})\times_G\mathbb{P}_G(\mathcal{S}))/{(\mathbb{Z}/2\mathbb{Z})}
      &\cong\Proj_G(\Sym^\bullet\mathcal{S}^\vee)\times_G\Proj_G(\Sym^\bullet\mathcal{S}^\vee)/(\mathbb{Z}/2\mathbb{Z}) \\
      &\cong\Proj_G(\Sym^\bullet\mathcal{S}^\vee\otimes\Sym^\bullet\mathcal{S}^\vee)/(\mathbb{Z}/2\mathbb{Z}) \\
      &\cong\Proj_G(\Sym^2\Sym^\bullet\mathcal{S}^\vee) \\
      &\cong\Proj_G(\Sym^\bullet\Sym^2\mathcal{S}^\vee) \\
      &\cong\mathbb{P}_G(\Sym^2\mathcal{S}),
    \end{aligned}
  \end{equation}
  where we used Hermite reciprocity (e.g., as in \cite[Exercise~6.18]{MR1153249})
  which states that taking~$\Sym^q$ and~$\Sym^p$ of a rank-2 bundle commutes for all~$p$ and~$q$.
\end{proof}

From \cref{lemma:projective-bundle} we obtain the relative Euler sequence
\begin{equation}
  \label{equation:relative-euler-sequence}
  0\to\mathcal{O}_H\to \pi^*(\Sym^2\mathcal{S})\otimes\mathcal{O}_\pi(1)\to\tangent_\pi\to 0
\end{equation}
and the relative tangent sequence
\begin{equation}
  \label{equation:relative-tangent-sequence}
  0\to\tangent_\pi\to\tangent_H\to\pi^*\tangent_G\to 0.
\end{equation}

Let us first compute the cohomology of the tangent bundle for~$H$.
The following lemma is standard,
where we denote by~$\mathbb{S}_\lambda$ the Schur functor
associated to a partition~$\lambda=(\lambda_1\geq\ldots\geq\lambda_\ell)$ of length~$\ell$.
Recall that~$\mathbb{S}_\lambda\mathcal{E}^\vee\cong\mathbb{S}_\mu\mathcal{E}$
where~$\mu$ is~$-\lambda$ reordered so that the entries are decreasing.
% Lemma 8 in the notes
\begin{lemma}
  \label{lemma:bwb-on-G}
  We have that
  \begin{equation}
    \HH^\bullet(G,\tangent_G)\cong\HH^0(G,\tangent_G)\cong\mathbb{S}_{(1,0,\ldots,0,-1)}V
  \end{equation}
  and
  \begin{equation}
    \HH^\bullet(G,\bigwedge^2\tangent_G)
    \cong\HH^0(G,\bigwedge^2\tangent_G)
    \cong
    \begin{cases}
      \mathbb{S}_{(2,-1,-1)}V & n=2 \\
      \mathbb{S}_{(2,0,\ldots,-1,-1)}V\oplus\mathbb{S}_{(1,1,0,\ldots,0,-2)}V & n\geq 3.
    \end{cases}
  \end{equation}
\end{lemma}

\begin{proof}
  This conveniently follows from the description in \cite[Theorem~B]{1911.09414v1}
  as the Grassmannian~$G$ is cominuscule.
\end{proof}

Now we compute the cohomology of the first term in
the relative tangent sequence \eqref{equation:relative-tangent-sequence}.
% Lemma 9 in the notes
\begin{lemma}
  \label{lemma:relative-tangent}
  We have that
  \begin{equation}
    \begin{aligned}
      \RRR\pi_*\tangent_\pi
      &\cong\RR^0\pi_*\tangent_\pi \\
      &\cong\mathbb{S}_{(2,-2)}\mathcal{S}\oplus\mathbb{S}_{(1,-1)}\mathcal{S},
    \end{aligned}
  \end{equation}
  and
  \begin{equation}
    \HH^\bullet(H,\tangent_\pi)
    \cong
    \begin{cases}
      \HH^1(H,\tangent_\pi)\cong\mathbb{S}_{(1,1,-2)}V & n=2 \\
      0 & n\geq 3.
    \end{cases}
  \end{equation}
\end{lemma}

\begin{proof}
  By applying~$\RRR\pi_*$ to the relative Euler sequence \eqref{equation:relative-euler-sequence}
  we obtain,
  using the identifications~$\RRR\pi_*\mathcal{O}_H\cong\mathcal{O}_G$
  and~$\RRR\pi_*\mathcal{O}_\pi(1)\cong\Sym^2\mathcal{S}^\vee$,
  the short exact sequence
  \begin{equation}
    \label{equation:short-exact-sequence}
    0\to\mathcal{O}_G\to\Sym^2\mathcal{S}\otimes\Sym^2\mathcal{S}^\vee\to\RR^0\pi_*\tangent_\pi\to 0.
  \end{equation}
  The middle term can be rewritten as
  \begin{equation}
    \label{equation:equation:sym-two-sym-two}
    \begin{aligned}
      \Sym^2\mathcal{S}\otimes\Sym^2\mathcal{S}^\vee
      &\cong
      \Sym^2\mathcal{S}\otimes\Sym^2\mathcal{S}\otimes(\det\mathcal{S})^{\otimes-2} \\
      &\cong
      (\Sym^4\mathcal{S}\otimes(\det\mathcal{S})^{\otimes-2})\oplus
      (\Sym^2\mathcal{S}\otimes(\det\mathcal{S})^{\otimes-1})\oplus
      \mathcal{O}_G
    \end{aligned}
  \end{equation}
  where we used \cite[Exercise~11.11]{MR1153249}.
  The inclusion in \eqref{equation:short-exact-sequence}
  allows us to cancel~$\mathcal{O}_G$ in \eqref{equation:equation:sym-two-sym-two}.
  The rest follows by applying Borel--Weil--Bott.
\end{proof}

Combining these two lemmas we obtain the following
from the long exact sequence associated to the relative tangent sequence \eqref{equation:relative-tangent-sequence}.
% Proposition 14 in the notes
\begin{corollary}
  \label{corollary:hilbert-square-tangent}
  If~$n=2$ then
  \begin{equation}
    \HH^i(H,\tangent_H)
    \cong
    \begin{cases}
      \mathbb{S}_{(1,0,-1)}V & i=0 \\
      \mathbb{S}_{(1,1,-2)}V & i=1 \\
      0 & i\geq2 \\
    \end{cases}
  \end{equation}
  whilst for~$n\geq 3$
  \begin{equation}
    \HH^i(H,\tangent_H)\cong
    \begin{cases}
      \mathbb{S}_{(1,0,\ldots,0,-1)}V & i=0 \\
      0 & i\geq 1.
    \end{cases}
  \end{equation}
\end{corollary}

Next we compute the cohomology of~$\bigwedge^2\tangent_H$.
The second exterior square of \eqref{equation:relative-tangent-sequence}
produces a filtration with associated graded pieces
\begin{itemize}
  \item $\bigwedge^2\tangent_\pi$,
  \item $\tangent_\pi\otimes\pi^*\tangent_G$, and
  \item $\pi^*\bigwedge^2\tangent_G$.
\end{itemize}
The cohomology of the first is computed as follows.
% Lemma 10 in the notes
\begin{lemma}
  \label{lemma:relative-bitangent}
  We have that
  \begin{equation}
    \begin{aligned}
      \RRR\pi_*\bigwedge^2\tangent_\pi
      &\cong\RR^0\pi_*\bigwedge^2\tangent_\pi \\
      &\cong\mathbb{S}_{(3,-3)}\mathcal{S}\oplus\mathbb{S}_{(1,-1)}\mathcal{S}.
    \end{aligned}
  \end{equation}
  In particular we have for~$n=2$ that
  \begin{equation}
    \HH^\bullet(H,\bigwedge^2\tangent_\pi)
    \cong\HH^1(H,\bigwedge^2\tangent_\pi)
    \cong\mathbb{S}_{(2,1,-3)}V.
  \end{equation}
\end{lemma}

\begin{proof}
  The second exterior square of the relative Euler sequence \eqref{equation:relative-euler-sequence} induces the short exact sequence
  \begin{equation}
    0\to\tangent_\pi\to\pi^*(\bigwedge^2\Sym^2\mathcal{S})\otimes\mathcal{O}_\pi(2)\to\bigwedge^2\tangent_\pi\to 0.
  \end{equation}
  By applying~$\RRR\pi_*$ to it
  we obtain,
  using the description of~$\RRR\pi_*\tangent_\pi$ from \cref{lemma:relative-tangent},
  the short exact sequence
  \begin{equation}
    \label{equation:short-exact-sequence-2}
    0
    \to\mathbb{S}_{(2,-2)}\mathcal{S}\oplus\mathbb{S}_{(1,-1)}\mathcal{S}
    \to\bigwedge^2\Sym^2\mathcal{S}\otimes\Sym^2\Sym^2\mathcal{S}^\vee
    \to\pi_*(\bigwedge^2\tangent_\pi)
    \to 0.
  \end{equation}
  The vanishing of~$\RR^{\geq 1}\pi_*(\bigwedge^2\tangent_\pi)$
  can be shown using cohomology and base change.
  Using \cref{lemma:relative-tangent}
  and the isomorphisms
  \begin{equation}
    \begin{aligned}
      \bigwedge^2\Sym^2\mathcal{S}&\cong\mathbb{S}_{(3,1)}\mathcal{S} \\
      \Sym^2\Sym^2\mathcal{S}^\vee&\cong\mathbb{S}_{(4,0)}\mathcal{S}^\vee\oplus\mathbb{S}_{(2,2)}\mathcal{S}^\vee
    \end{aligned}
  \end{equation}
  which are standard plethysms,
  we can rewrite \eqref{equation:short-exact-sequence-2}
  as
  \begin{equation}
    0
    \to\mathbb{S}_{(2,-2)}\mathcal{S}\oplus\mathbb{S}_{(1,-1)}\mathcal{S}
    \to\mathbb{S}_{(3,-3)}\mathcal{S}\oplus\mathbb{S}_{(2,-2)}\mathcal{S}\oplus(\mathbb{S}_{(1,-1)}\mathcal{S})^{\oplus2}
    \to\pi_*(\bigwedge^2\tangent_\pi)
    \to 0.
  \end{equation}
  As all the morphisms in this sequence are equivariant
  we can cancel the corresponding summands
  and thus we obtain the first part of the lemma.
  The rest is an application of the Borel--Weil--Bott theorem.
\end{proof}

For the second graded piece we have the following.
% Lemma 13 in the notes
\begin{lemma}
  \label{lemma:tangent-tensor-relative-tangent}
  We have that
  \begin{equation}
    \begin{aligned}
      \RRR\pi_*(\tangent_\pi\otimes\pi^*\tangent_G)
      &\cong\RR^0\pi_*(\tangent_\pi\otimes\pi^*\tangent_G) \\
      &\cong\mathcal{Q}\otimes\left( \mathbb{S}_{(2,-3)}\mathcal{S}\oplus(\mathbb{S}_{(1,-2)}\mathcal{S})^{\oplus2}\oplus\mathcal{S}^\vee \right).
    \end{aligned}
  \end{equation}
  In particular we have for~$n=2$ that
  \begin{equation}
    \HH^\bullet(H,\tangent_\pi\otimes\pi^*\tangent_G)
    \cong\HH^0(H,\tangent_\pi\otimes\pi^*\tangent_G)
    \cong(\mathbb{S}_{(1,1,-2)}V)^{\oplus2}\oplus\mathbb{S}_{(1,0,-1)}V.
  \end{equation}
\end{lemma}
We leave the case~$n\geq 3$ to the interested reader, here,
and in what follows.

\begin{proof}
  Using the description from \cref{lemma:relative-tangent} we obtain that
  the derived direct image is concentrated in degree zero,
  and that it is isomorphic to
  \begin{equation}
    (\mathbb{S}_{(2,-2)}\mathcal{S}\oplus\mathbb{S}_{(1,-1)}\mathcal{S})\otimes\mathcal{S}^\vee\otimes\mathcal{Q},
  \end{equation}
  using~$\tangent_G\cong\mathcal{S}^\vee\otimes\mathcal{Q}$.
  This proves the first part of the lemma.
  The rest is an application of the Borel--Weil--Bott theorem.
\end{proof}

Because the cohomology of the third term can be computed using \cref{lemma:bwb-on-G}
we can use the filtration and the previous two lemmas to obtain the following.
% Proposition 15 in the notes
\begin{corollary}
  \label{corollary:hilbert-square-bitangent}
  If~$n=2$ then
  \begin{equation}
    \HH^i(H,\bigwedge^2\tangent_H)
    \cong
    \begin{cases}
      (\mathbb{S}_{(1,1,-2)}V)^{\oplus2}\oplus\mathbb{S}_{(1,0,-1)}V\oplus\mathbb{S}_{(2,-1,-1)}V & i=0 \\
      \mathbb{S}_{(2,1,-3)}V & i=1 \\
      0 & i\geq2 \\
    \end{cases}
  \end{equation}
\end{corollary}

Finally, to compute the cohomology of~$\bigwedge^3\tangent_H$
one could use the methods used before in the proof of \cref{corollary:hilbert-square-tangent},
starting from the isomorphism~$\bigwedge^3\tangent_H\cong\Omega_H^1\otimes\omega_H^\vee$,
and twisting the duals of \eqref{equation:relative-euler-sequence} and \eqref{equation:relative-tangent-sequence}
by~$\omega_H^\vee\cong\mathcal{O}_\pi(3)$.
But by bootstrapping from \cref{lemma:hochschild-hilbert-series-hilbert-square-P2}
we can now also compute the cohomology of~$\bigwedge^3\tangent_H$ using a shortcut.
First we need the following Euler characteristic calculation.
\begin{lemma}
  \label{lemma:euler-characteristic}
  We have that
  \begin{equation}
    \chi(H,\bigwedge^3\tangent_H)=52.
  \end{equation}
\end{lemma}

\begin{proof}
  This can be computed using the Schubert2 package of Macaulay2 \cite{schubert2}:
  \code
\end{proof}

\begin{corollary}
  \label{corollary:hilbert-square-tritangent}
  We have that
  \begin{equation}
    \HH^i(H,\bigwedge^3\tangent_H)\cong
    \begin{cases}
      \field^{80} & i=0 \\
      \field^{28} & i=1 \\
      0 & i\geq 2.
    \end{cases}
  \end{equation}
\end{corollary}

\begin{proof}
  By \cref{lemma:hochschild-hilbert-series-hilbert-square-P2}
  we get the vanishing for~$i\geq 3$.
  Combining \cref{corollary:hilbert-square-tangent,corollary:hilbert-square-bitangent}
  with \cref{lemma:hochschild-hilbert-series-hilbert-square-P2}
  we obtain~$\HH^0(H,\bigwedge^3\tangent_H)\cong\field^{80}$,
  thus \cref{lemma:euler-characteristic} allows us to compute~$\HH^1$.
\end{proof}

\printbibliography

\emph{Pieter Belmans}, \url{pieter.belmans@uni.lu} \\
Department of Mathematics, Universit\'e du Luxembourg, Avenue de la Fonte 6, L-4364 Esch-sur-Alzette, Luxembourg

\emph{Lie Fu}, \url{lie.fu@math.unistra.fr} \\
Institut de recherche math\'ematique avanc\'ee (IRMA), Universit\'e de Strasbourg, 7 rue Ren\'e-Descartes, 67084 Cedex, Strasbourg, France

\emph{Andreas Krug}, \url{krug@math.uni-hannover.de} \\
Institut f\"ur algebraische Geometrie, Gottfried Wilhelm Leibniz Universit\"at Hannover, Welfengarten 1, 30167 Hannover, Germany

\end{document}